\ifdefined\XeTeXrevision\else
  \ifdefined\pdfoutput\pdfoutput=1\fi
\fi
\documentclass[preprint,11pt]{elsarticle}

\usepackage[T1]{fontenc}
\usepackage{lmodern}
\usepackage{amsmath,amssymb,amsthm,mathtools}
\usepackage{mathrsfs}
\usepackage[expansion=false]{microtype}
\usepackage{enumitem}
\usepackage[margin=1.15in]{geometry}
\usepackage[hyphens]{url}
\usepackage[hidelinks]{hyperref}
\hypersetup{
  pdftitle={Cotangent cohomology and higher-order obstructions of matroids},
  pdfauthor={Qihang Wang},
  pdfsubject={Cotangent cohomology and deformation theory of matroids},
  pdfkeywords={Matroid, Stanley--Reisner ring, Cotangent cohomology, Deformation theory, Obstruction, Cup product, Cocircuit}
}
\journal{Journal of Algebra}

\allowdisplaybreaks
\setlist[enumerate]{leftmargin=2.2em}

\newtheorem{theorem}{Theorem}[section]
\newtheorem{proposition}[theorem]{Proposition}
\newtheorem{lemma}[theorem]{Lemma}
\newtheorem{corollary}[theorem]{Corollary}
\theoremstyle{definition}
\newtheorem{definition}[theorem]{Definition}
\newtheorem{example}[theorem]{Example}
\theoremstyle{remark}
\newtheorem{remark}[theorem]{Remark}

\DeclareMathOperator{\rk}{rk}
\DeclareMathOperator{\corank}{corank}
\DeclareMathOperator{\cl}{cl}
\DeclareMathOperator{\supp}{supp}
\DeclareMathOperator{\lk}{lk}
\DeclareMathOperator{\Hilb}{Hilb}
\DeclareMathOperator{\im}{im}
\newcommand{\one}{\mathbf 1}
\newcommand{\ee}{\mathbf e}
\newcommand{\NN}{\mathbb N}
\newcommand{\ZZ}{\mathbb Z}
\newcommand{\wtH}{\widetilde H}
\newcommand{\Deltao}{\Delta}

\begin{document}
\raggedbottom

\begin{frontmatter}

\title{Cotangent cohomology and higher-order obstructions of matroids}

\author{Qihang Wang\corref{cor1}}
\ead{2501110049@stu.pku.edu.cn}
\cortext[cor1]{Corresponding author.}

\affiliation{organization={School of Mathematical Sciences, Peking University},
  addressline={No.~5 Yiheyuan Road, Haidian District},
  city={Beijing},
  postcode={100871},
  country={China}}

\begin{abstract}
Matroid Stanley--Reisner rings admit essential deformation obstructions of arbitrarily high order over every field, already at corank three. For each $m\ge3$, we construct a connected cosimple matroid and $m$ fine-homogeneous first-order classes whose proper subcollections integrate jointly and whose pairwise primary cup products vanish. A specified joint $(m-1)$-jet nevertheless has no lift to order $m$, even without preserving the grading. The construction rests on formulas for every multigraded component of the second cotangent cohomology $T^2$ and of its primary cup image, expressed through flats and parallel classes of the dual matroid. These formulas show that $T^2$ vanishes exactly when every connected component has corank at most two, or equivalently when there is no $U_{1,4}$-minor. This proves a conjecture of Constantinescu, Klein, Nguyen, Singh, and Venturello (2025) and characterizes formal smoothness of the affine deformation functor. For the duals of rank-three truncations of $U_{2,m}\oplus U_{2,n}$, $m,n\ge3$, each nonzero homogeneous class modulo the primary cup image is realized as a higher-order obstruction, and the quotient has an explicit presentation with $m+n-2$ minimal homogeneous generators.

The main results of this paper were obtained through a generative-AI workflow. The models used in this work are OpenAI \mbox{GPT-5.6} Sol, Anthropic Claude Fable~5, Grok~4.6, OpenAI \mbox{GPT-6} Astra, Anthropic Claude Fable~5.1, OpenAI \mbox{GPT-6.1} Sol, and Anthropic Claude Opus~5.5; details appear in the declaration at the end of the paper.
\end{abstract}


\end{frontmatter}

\section{Introduction}

When do compatible infinitesimal deformations extend to a joint flat family? The primary cup product detects quadratic obstructions, but its vanishing need not settle the higher-order lifting problem. We exhibit this distinction for matroid Stanley--Reisner rings at fixed corank three. For arbitrarily many fine-homogeneous first-order classes, every proper subcollection integrates and every pairwise primary product vanishes, yet an explicitly constructed joint jet has an essential obstruction at an arbitrarily late order. The nonliftability persists when nonhomogeneous correction terms are allowed. The construction rests on a computation of every fine-graded component of the obstruction module and of the part detected by primary cup products.

Our starting point is the relative-cohomology description and cup-product formula of Altmann and Christophersen \cite[Proposition~11, Theorem~13, and Section~6]{AltmannChristophersen}. Brehm and Constantinescu computed the first cotangent cohomology of matroids \cite[Theorem~4.7]{BrehmConstantinescu}; after dualization, its nonzero negative supports have rank one. For second cotangent cohomology, Constantinescu, Klein, Nguyen, Singh, and Venturello treated uniform matroids and proved vanishing in corank at most two. They conjectured the converse: a connected matroid with vanishing $T^2$ has corank at most two \cite[Proposition~21 and Conjecture~22]{CKNSV}. We determine the full module's fine components and its vanishing locus, then use explicit cup products and lifting calculations to identify actual obstructions. The latter step is necessary: a nonzero obstruction space alone does not imply that a deformation is obstructed.

Explicit deformation bases for Stanley--Reisner rings were constructed by Altmann and Christophersen for specified combinatorial surfaces, together with normal forms for polygon cones \cite[Theorem~6.4 and Proposition~6.6]{AltmannChristophersenDeforming}. Our examples instead keep matroid corank fixed while making the first failed extension of a specified jet arbitrarily late. A cyclic relation among circuit generators produces the obstruction; integral flatness and a socle calculation show that no correction removes it. This higher-order phenomenon concerns lifting order, not the index of cotangent cohomology. The distinction separates it from the characteristic-zero link-localization theorem of Ilten, Meazzini, and Petracci \cite[Theorem~4.1]{IltenMeazziniPetracci}, which applies to higher cotangent-cohomology groups and yields descriptions of $T^3$ and $T^4$ for flag complexes.

\begin{theorem}[Unbounded obstruction order at corank three]\label{thm:intro-higher}
For every $m\ge3$ there is a connected cosimple matroid $M$ of rank $m$ and corank three with nonzero classes $\alpha_i\in T^1_{-\ee_i}(M)$, $0\le i<m$, over every field, such that:
\begin{enumerate}[label=\textup{(\arabic*)}]
\item every proper subcollection admits a joint deformation over its polynomial parameter ring;
\item all the classes admit a flat joint deformation over $k[t_0,\ldots,t_{m-1}]/(t_0\cdots t_{m-1})$;
\item the resulting deformation modulo $(t_0,\ldots,t_{m-1})^m$ cannot lift modulo $(t_0,\ldots,t_{m-1})^{m+1}$, even in the ungraded category.
\end{enumerate}
Within the fine-multigraded category, with $\deg t_i=\ee_i$, the jet in \textup{(3)} is uniquely determined by its first-order classes. Its first essential obstruction has order $m$, although all pairwise primary cup products vanish.
\end{theorem}

The matroid in Theorem~\ref{thm:intro-higher} is dual to the rank-three truncation of $U_{2,m}\oplus U_{2,3}$. The family is governed by explicit cyclic polynomials and a telescoping identity whose residual is the product of all parameters. Two circuit relations force contradictory corrections of its coefficient, proving nonliftability. Section~\ref{sec:higher-obstructions} gives the construction and its proof. Here an essential obstruction of order $m$ means a nonzero obstruction to extending the specified $(m-1)$-jet to order $m$. The uniqueness statement is graded; we do not identify all ungraded higher jets with the same tangent classes or a full miniversal base.

Theorem~\ref{thm:homogeneous-realization} strengthens this construction
to every nonzero homogeneous class of the complete primary quotient
$T^2(M_{m,n})/\Omega$, including degrees with positive coordinates.
Here $M_{m,n}$ is dual to $\operatorname{Tr}_3(U_{2,m}\oplus U_{2,n})$,
whose two disjoint lines $L,H$ have sizes $m,n$; put $R=k[M_{m,n}]$,
and let $\Omega$ be the image of $T^1\otimes_R T^1\to T^2$.
A prescribed class whose negative support has size $r$ occurs as the
only nonzero order-$r$ obstruction coefficient of a specified joint jet
in $r$ parameters.
All its selected tangent classes have zero primary products, and every
proper subcollection integrates formally. For purely negative degrees
the stronger polynomial integration holds
(Theorem~\ref{thm:negative-realization}); the purely negative quotient
has exponential dimension. Thus these dimensions count actual higher
obstruction values, not merely an ambient obstruction space. Classical
link reduction computes the positive-supported component, but constructing
a deformation of the original ring requires the variable-weight cyclic
family and Artinian relation-lifting argument below. The realizations are
degree by degree, not simultaneous realizations of sums.

For $M_{m,n}$ the quotient has a more economical algebraic
description than its negative dimension suggests. Theorem~\ref{thm:primary-quotient-module}
gives an explicit fine-graded $R$-module presentation with exactly
$m+n-2$ minimal generators, concentrated in the two degrees
$-\one_L$ and $-\one_H$. It also determines the radical annihilator:
every irreducible component of the support of this quotient has codimension
two in $\operatorname{Spec}R$. The decisive step is a Taylor coboundary
identity describing multiplication between negative degrees; the graph
dimensions alone do not determine it. The coordinate modules in the
presentation also record the positive-support restriction maps.

Our structural starting point reduces the relative-cohomology calculation to connected components of explicit graphs. Positive support is removed by passing to a link; duality then organizes the remaining relative complex by hyperplanes. This reduction is the common source of the formula, the vanishing classification, and the cup-product coordinates below.

Let $N$ have rank $s$, let $\varnothing\ne B\subseteq E(N)$, and let $H$ be a hyperplane disjoint from $B$.  Define $\mathcal G_B(H)$ to be the rank-$(s-2)$ flats $G\subset H$ such that $\rk_N(G\cup B)=s$.  Make this set the vertex set of a graph $\Gamma_B(H)$ by joining distinct $G,G'$ when
\[
\rk_N\bigl((G\cap G')\cup B\bigr)=s.
\]
Empty graphs contribute zero to reduced zeroth cohomology.

\begin{theorem}[Complete hyperplane-graph formula]\label{thm:intro-graph}
Write a fine degree uniquely as $\mathbf c=\mathbf a-\mathbf b$, where $\mathbf a,\mathbf b\in\NN^{E(M)}$ have disjoint supports.  A component $T^2_{\mathbf c}(M)$ vanishes unless $\mathbf b=\one_B$ is nonzero and squarefree, $A=\supp\mathbf a$ is independent, and $B$ consists of actual vertices of $L=\lk_M A$.  Under these conditions, set $N=L^*$.  Then, over every field,
\[
T^2_{\mathbf c}(M)
\cong
\bigoplus_{\substack{H\text{ hyperplane of }N\\H\cap B=\varnothing}}
\wtH^0(\Gamma_B(H);k).
\]
The analogous formula for $T^1$, which recovers in dual form the description of Brehm and Constantinescu \cite{BrehmConstantinescu}, is given in Theorem~\ref{thm:universal-graph}.
\end{theorem}

In fact these graphs are empty, discrete, or connected according as the negative support has dual rank at most one, two, or at least three (Lemma~\ref{lem:graph-trichotomy}). Hence only rank-two supports contribute to $T^2$, and its dimensions are flat counts rather than connectivity computations. The dual rank of the negative support is unchanged by passing to the relevant link. This gives the complete fine Hilbert series over every field (Theorem~\ref{thm:rank-two-support}).

When $B=\{p,q\}$, the graph $\Gamma_B(H)$ is discrete.  This special case detects vanishing.

\begin{theorem}[Vanishing and the excluded minor]\label{thm:intro-vanishing}
For a finite matroid $M$, the following are equivalent:
\begin{enumerate}[label=\textup{(\arabic*)}]
\item $T^2(M)=0$;
\item $T^2_{-\ee_p-\ee_q}(M)=0$ for all distinct $p,q\in E(M)$;
\item every connected component of $M$ has corank at most two;
\item every cocircuit of $M$ has size at most three;
\item $M$ has no minor isomorphic to $U_{1,4}$.
\end{enumerate}
\end{theorem}

We next identify actual quadratic obstructions rather than merely nonzero ambient obstruction spaces.  The singleton part of the graph model identifies $T^1_{-\ee_p}(M)$ with functions on certain dual hyperplanes, modulo constants.  Substitution in the Altmann--Christophersen cup formula gives explicit block coordinates and a universal surjectivity statement.

\begin{theorem}[Universal pure-pair cup product]\label{thm:intro-cup}
For every matroid $M$, every pair of distinct actual vertices $p,q$, and every field $k$, the cup product
\[
T^1_{-\ee_p}(M)\otimes_k T^1_{-\ee_q}(M)
\longrightarrow T^2_{-\ee_p-\ee_q}(M)
\]
is surjective.
\end{theorem}

The proof constructs every coordinate delta vector as a decomposable cup product.  It uses only coefficients $0,\pm1$ in the monomial presentation, and therefore works in every characteristic.  Iterated link localization transports the statement to arbitrary positive support.

\begin{theorem}[Degree-zero obstructions in every rank]\label{thm:intro-all-rank}
Let $M$ be connected with $\rk M\ge2$ and $\corank M\ge3$.  There are distinct $e,p,q\in E(M)$ for which
\[
T^1_{\ee_e-\ee_p}(M)\otimes_kT^1_{\ee_e-\ee_q}(M)
\twoheadrightarrow
T^2_{2\ee_e-\ee_p-\ee_q}(M)\ne0.
\]
All three fine degrees have standard total degree zero.  Over every field, some standard-degree-zero first-order deformation fails to lift from $k[\epsilon]/(\epsilon^2)$ to $k[\epsilon]/(\epsilon^3)$.
\end{theorem}

No characteristic restriction is needed.  The primary obstruction map is quadratic with the cup product as polarization, and it doubles fine degrees.  Every nonzero fine component of $T^1$ has nonzero negative support, so the doubled degree fails the squarefree support criterion and the self-obstruction of every fine-homogeneous class vanishes.  The mixed cup term is therefore the entire primary obstruction, in every characteristic, including two (Lemma~\ref{lem:graded-polarization}).  Combining this with the vanishing classification yields a complete smoothness criterion.

\begin{corollary}[Smoothness dichotomy]\label{cor:intro-dichotomy}
For every field $k$ and every matroid $M$, the affine deformation functor of $k[M]$ is formally smooth if and only if every connected component of $M$ has corank at most two, if and only if $M$ has no $U_{1,4}$-minor.
\end{corollary}

Formal smoothness is in principle sensitive to the ground field; for matroids, Corollary~\ref{cor:intro-dichotomy} decides it by one excluded minor, uniformly in the characteristic.

These primary results coexist with the higher obstructions of Theorem~\ref{thm:intro-higher}. Every fine-homogeneous class has an all-order formal representative. If two classes have disjoint negative supports and positive supports avoiding their union, their sum admits an all-order formal arc whenever their cup product vanishes (Theorem~\ref{thm:arcs}). This one-parameter assertion is distinct from joint integration over independent parameters. The cyclic family shows that pairwise vanishing does not guarantee extension of a specified joint jet. Theorem~\ref{thm:full-primary-image} determines the primary image in every fine degree: it fills supports meeting exactly two dual parallel classes, and misses every support meeting at least three, regardless of the positive exponents. The lifting argument uses the circuit description of $T^1$ from \cite{BrehmConstantinescu} to show that first cohomology is generated in purely negative degrees. Corollary~\ref{cor:primary-quotient} then gives the degreewise quotient by the primary image and its characteristic-independent dimensions. For a connected matroid, the standard-degree-preserving functor is formally smooth precisely when $\rk M\le1$ or $\corank M\le2$ (Corollary~\ref{cor:graded-dichotomy}).

The vanishing theorem also gives explicit certificates. If $D$ is a cocircuit of size at least four and $A$ is a basis of its complementary hyperplane, then $\dim_kT^2_{\one_A-\ee_p-\ee_q}(M)=|D|-3$ for distinct $p,q\in D$. These witnesses can be checked with a rank oracle, and together they recover all cocircuits of size at least four. Their total dimension $\Theta(M)$ has sharp uniform-matroid bounds. The pure-pair invariant $\mathfrak q(M)=\sum_{\{p,q\}\subseteq E(M)}\dim_kT^2_{-\ee_p-\ee_q}(M)$ instead measures the target dimension of an integrally split direct sum of cup maps, not the entire quadratic obstruction map; its direct-sum and rank-two extremal properties are treated in Section~\ref{sec:pair-rank}.

The combinatorial applications include a formula for $\dim_kT^2(M)_0$ in rank three, reconstruction of simple rank-three matroids on at least seven points from labeled nonvanishing data, and unique minimization of the degree-zero dimension by the near-pencil on $n\ge8$ points. The weighted cocircuit profile underlying $\Theta$ is a linear specialization of the Derksen--Fink $\mathcal G$-invariant but is not determined by the Tutte polynomial. In contrast, $\mathfrak q$ is neither valuative nor determined by the Tutte polynomial. These results describe information in the obstruction module beyond its vanishing locus.

There are two reading routes through the paper. For deformation theory, Sections~\ref{sec:preliminaries} and \ref{sec:graph-model} give the support criterion, graph formula, and singleton coordinates for $T^1$; Section~\ref{sec:cup} computes the pair products. Section~\ref{sec:primary-image} then determines the full primary image, leading to the cyclic construction in Section~\ref{sec:higher-obstructions}, homogeneous-class realization in Sections~\ref{sec:negative-realization} and \ref{sec:positive-realization}, and the quotient-module presentation in Section~\ref{sec:primary-quotient-module}. The intervening extremal and valuative results are not prerequisites for this route.

For the combinatorial consequences, Sections~\ref{sec:pairs} and \ref{sec:vanishing} prove the pair formula and vanishing classification. Sections~\ref{sec:theta} and \ref{sec:rank-two} compute positive-support strands, while Sections~\ref{sec:algorithms}--\ref{sec:valuative-profile} develop constructive witnesses, extremal bounds, rank-three reconstruction, the near-pencil theorem, and valuative distinctions. Section~\ref{sec:verification} describes the finite checks that accompany the proofs. The questions in Section~\ref{sec:questions} concern the remaining module and lifting information.

\section{Preliminaries}\label{sec:preliminaries}

\subsection{Matroids and links}

Let $M$ be a finite matroid on ground set $E$. We write $(u)_+=\max\{u,0\}$ and let $\kappa(\Gamma)$ denote the number of connected components of a graph $\Gamma$. For the empty graph, both $\kappa$ and reduced zeroth cohomology are zero. We identify $M$ with its independence complex. Its rank and corank are
\[
\rk M=\max\{|I|:I\in M\},\qquad \corank M=|E|-\rk M.
\]
The circuits are the minimal nonfaces. We write $\mathcal L(M)$ for the lattice of flats of $M$. The dual matroid is denoted $M^*$; the cocircuits of $M$ are the circuits of $M^*$, equivalently the complements of the hyperplanes of $M$.

For $A\in M$, the simplicial link is
\[
\lk_M A=\{F\subseteq E\setminus A:A\cup F\in M\}.
\]
It is the independence complex of $M/A$ after deleting contraction loops. We write $V(\Delta)$ for the actual vertex set of a simplicial complex and always use this convention for links. If $A$ is a basis of a hyperplane $H$ and $D=E\setminus H$, then $\lk_M A$ is the zero-dimensional complex $Z_D$ on the vertex set $D$.

Fix a basis $B$ of $M$.  For $e\notin B$, write $C(e;B)$ for the fundamental circuit in $B\cup\{e\}$.  The fundamental graph $G_B(M)$ is bipartite on $B\sqcup(E\setminus B)$, with $b$ adjacent to $e$ precisely when $b\in C(e;B)$.  We use the standard facts
\[
M\text{ connected}\iff G_B(M)\text{ connected}
\]
and, for $A_0\subseteq B$ and every nonloop $e$ of $M/A_0$,
\[
C_{M/A_0}(e;B\setminus A_0)=C_M(e;B)\setminus A_0.
\]

\subsection{Cotangent cohomology and the support criterion}

Let $S=k[x_e:e\in E]$ and let $I_M\subset S$ be the Stanley--Reisner ideal generated by the squarefree circuit monomials. Put $R=k[M]=S/I_M$ and $T^i(M)=\operatorname{Ext}^i_R(\mathbb L_{R/k},R)$, where $\mathbb L_{R/k}$ is the cotangent complex. The fine $\ZZ^E$-grading induces
\[
T^i(M)=\bigoplus_{\mathbf c\in\ZZ^E}T^i_{\mathbf c}(M).
\]
For a scalar subscript, $T^i(M)_d$ denotes the standard total-degree
component $\bigoplus_{\sum_e c_e=d}T^i_{\mathbf c}(M)$.
Write each degree uniquely as $\mathbf c=\mathbf a-\mathbf b$, with $\mathbf a,\mathbf b\in\NN^E$ and disjoint supports $A,B$.  The Altmann--Christophersen support criterion \cite[Proposition~11 and Theorem~13]{AltmannChristophersen} gives
\[
T^i_{\mathbf a-\mathbf b}(M)=0\quad(i=1,2)
\]
unless $\mathbf b$ is nonzero and squarefree, $A\in M$, and $B$ is contained in the actual vertex set of $\lk_M A$.  When these conditions hold,
\begin{equation}\label{eq:link-formula}
T^i_{\mathbf a-\one_B}(M)\cong T^i_{-\one_B}(\lk_M A),\qquad i=1,2.
\end{equation}
The isomorphism depends on the support $A$, not on the positive exponents.

For a simplicial complex $\Delta$ and nonempty $B$, define
\[
N_B(\Delta)=\{F\in\Delta:F\cap B=\varnothing,\ F\cup B\notin\Delta\}
\]
and
\[
\widetilde N_B(\Delta)=
\{F\in N_B(\Delta):F\cup B'\notin\Delta
\text{ for some }B'\subsetneq B\}.
\]
If $|B|>1$, then \cite[Theorem~9]{AltmannChristophersen}
\begin{equation}\label{eq:AC-relative}
T^i_{-\one_B}(\Delta)
\cong
H^{i-1}(|N_B(\Delta)'|,|\widetilde N_B(\Delta)'|;k),
\qquad i=1,2.
\end{equation}
Here a family of faces is ordered by inclusion, and the prime denotes
its order complex. We also write $\Deltao P$ for the order complex of a
finite poset $P$. Reduced cohomology is augmented, so
$\wtH^{-1}(\varnothing;k)=k$ and $\wtH^0(\varnothing;k)=0$.
The original formulation in \cite{AltmannChristophersen} is over
$\mathbb C$; the same combinatorial cochain presentation with coefficients
in an arbitrary field is recorded in \cite[Sections~2.2--2.3]{CKNSV}.
The singleton formula uses the reduced version of the same complex;
we record its matroid form in Theorem~\ref{thm:universal-graph}.

\begin{lemma}[Localization and cup products]\label{lem:localization}
Let $R$ be a finitely presented $k$-algebra and $U\subset R$ a multiplicative set.  For $i=1,2$, flat localization gives natural isomorphisms
\[
U^{-1}T^i(R/k;R)\cong T^i(U^{-1}R/k;U^{-1}R),
\]
compatible with cup products.  For a vertex $v$ of a simplicial complex $\Delta$,
\[
k[\Delta]_{x_v}\cong k[\lk_\Delta v][x_v^{\pm1}]
\]
after deleting ghost vertices, and the corresponding positive cotangent cohomology is the Laurent extension of that of the link. If a fine degree $\mathbf a-\one_B$ satisfies the support criterion and $v\in\supp\mathbf a$, localization followed by setting $x_v=1$ identifies this component with the component of the link obtained by deleting its $v$-coordinate. Iteration over $\supp\mathbf a$ gives \eqref{eq:link-formula}.
\end{lemma}

\begin{proof}
Choose a finite-free model for the cotangent complex through the degrees needed for $T^1$ and $T^2$.  Flat localization commutes with its dual and with cohomology.  The cup product is induced by composition of cochain-level lifts, so it also commutes with localization.

Inverting $x_v$ kills variables which are not vertices of the link and turns the remaining monomial relations into those of $\lk_\Delta v$. The transitivity triangle for the Laurent extension splits into the base-changed cotangent complex of the link and the free summand generated by $dx_v$. The free summand has no positive cotangent cohomology. If $v\in\supp\mathbf a$, increasing its positive exponent does not change the relative complex in \eqref{eq:AC-relative}; multiplication by $x_v$ is the identity under these identifications. Thus the localization direct limit is already represented by the original fine component. Removing the Laurent monomial by setting $x_v=1$ deletes that coordinate. The positive-support hypothesis is essential: localization need not preserve a component with zero $v$-coordinate.
\end{proof}

\begin{lemma}[Zero-dimensional calculation]\label{lem:zero-dimensional}
Let $Z_D$ be the zero-dimensional complex on $d=|D|$ vertices.  For a nonzero squarefree $\mathbf b$,
\[
\dim_kT^2_{-\mathbf b}(Z_D)=
\begin{cases}
(d-3)_+,&\supp\mathbf b=\{p,q\},\ p\ne q,\\
0,&\text{otherwise}.
\end{cases}
\]
All components with nonempty positive support vanish.  Thus $T^2(Z_D)=0$ if and only if $d\le3$.
\end{lemma}

\begin{proof}
For $B=\{p,q\}$, the support is a minimal nonface and $\widetilde N_B$ is the discrete set indexed by $D\setminus\{p,q\}$.  The relevant group is its reduced zeroth cohomology, of dimension $(d-3)_+$.  Singleton support has vanishing first cohomology.  A support of size at least three is a nonminimal nonface, so the relative pair agrees and contributes zero.  A nonempty positive support has a simplex as its link.
\end{proof}

We also use two standard facts. First, matroids of corank at most two have vanishing $T^2$ \cite[Proposition~21]{CKNSV}. Second, for matroids on disjoint ground sets,
\begin{equation}\label{eq:direct-sum-T2}
T^2(M_1\oplus M_2)\cong
\bigl(T^2(M_1)\otimes_k k[M_2]\bigr)
\oplus
\bigl(k[M_1]\otimes_kT^2(M_2)\bigr).
\end{equation}
This follows from the split exact sequences for the cotangent cohomology of a tensor product of algebras over a field \cite{ChristophersenIlten}, since $k[M_1\oplus M_2]=k[M_1]\otimes_kk[M_2]$.  In particular, $T^2(M_1\oplus M_2)=0$ exactly when both summands vanish.

\section{The complete hyperplane-graph model}\label{sec:graph-model}

Let $N$ be a matroid of rank $s$ on $E$ and let $\varnothing\ne B\subseteq E$.  For a hyperplane $H$ of $N$ with $H\cap B=\varnothing$, put
\[
\mathcal G_B(H)=
\{G\in\mathcal L(N):\rk_NG=s-2,\ G\subset H,\ \rk_N(G\cup B)=s\}.
\]
If $s<2$, this set is empty. Define a graph $\Gamma_B(H)$ on $\mathcal G_B(H)$ by joining distinct vertices according to
\begin{equation}\label{eq:graph-adjacency}
G\sim G'
\quad\Longleftrightarrow\quad
\rk_N((G\cap G')\cup B)=s.
\end{equation}
An empty graph contributes zero to $\wtH^0$.

\begin{theorem}[Universal hyperplane-graph formula]\label{thm:universal-graph}
Let $M$ be a matroid, write $\mathbf c=\mathbf a-\mathbf b$ with disjoint nonnegative parts, and put $A=\supp\mathbf a$.  If $\mathbf b=0$, if $\mathbf b$ is not squarefree, or if $A\notin M$, then $T^i_{\mathbf c}(M)=0$ for $i=1,2$.  Suppose $\mathbf b=\one_B$ is nonzero, set $L=\lk_MA$, and assume that $B$ lies in the actual vertex set of $L$.  Put $N=L^*$.

If $|B|\ge2$, then
\begin{align}
T^1_{\mathbf c}(M)
&\cong
\bigoplus_{\substack{H\text{ hyperplane of }N\\
H\cap B=\varnothing\\\mathcal G_B(H)=\varnothing}}k,
\label{eq:universal-T1}\\
T^2_{\mathbf c}(M)
&\cong
\bigoplus_{\substack{H\text{ hyperplane of }N\\H\cap B=\varnothing}}
\wtH^0(\Gamma_B(H);k).
\label{eq:universal-T2}
\end{align}
Consequently,
\begin{equation}\label{eq:universal-dimension}
\dim_kT^2_{\mathbf c}(M)=
\sum_{\substack{H\text{ hyperplane of }N\\H\cap B=\varnothing}}
\bigl(\kappa(\Gamma_B(H))-1\bigr)_+.
\end{equation}
If $B=\{p\}$ and $\mathscr H_p(N)$ is the discrete set of hyperplanes avoiding $p$, then
\begin{equation}\label{eq:singleton-formula}
T^1_{\mathbf c}(M)\cong\wtH^0(\mathscr H_p(N);k),
\qquad T^2_{\mathbf c}(M)=0.
\end{equation}
\end{theorem}

\begin{proof}
The support criterion and \eqref{eq:link-formula} reduce the problem to $T^i_{-\one_B}(N^*)$.  Assume first that $|B|\ge2$.  For $F\subseteq E\setminus B$, set
\[
X=E\setminus(F\cup B).
\]
Complementation and the dual-rank identity identify the Altmann--Christophersen pair, with its order reversed, with
\[
P_B=\{X\subseteq E\setminus B:\rk_NX<s,\ \rk_N(X\cup B)=s\}
\]
and
\[
Q_B=\{X\in P_B:\rk_N(X\cup\{b\})<s
\text{ for some }b\in B\}.
\]
Indeed, if $B'\subsetneq B$ and $D=B\setminus B'$, then
\[
E\setminus(F\cup B')=X\cup D.
\]
The condition $\rk_N(X\cup D)<s$ implies the singleton condition for every $b\in D$; conversely a singleton condition is obtained from $B'=B\setminus\{b\}$.  Thus
\begin{equation}\label{eq:PB-QB-relative}
T^i_{-\one_B}(N^*)\cong
H^{i-1}(|P_B'|,|Q_B'|;k),\qquad i=1,2.
\end{equation}

Define
\[
c(X)=\cl_N(X)\setminus B.
\]
Because $X\subseteq c(X)\subseteq\cl_NX$, the sets $X$ and $c(X)$ have the same closure.  Hence $c$ preserves every rank condition defining $P_B$ and $Q_B$; it is an ascending idempotent on $P_B$ and restricts to one on $Q_B$.  The closure-operator lemma \cite[Corollary~10.12, p.~1852]{Bjorner} gives homotopy equivalences on both spaces.  Applying the five lemma to the two long exact sequences of pairs yields
\begin{equation}\label{eq:relative-closure}
H^\bullet(|P_B'|,|Q_B'|;k)
\cong H^\bullet(|c(P_B)'|,|c(Q_B)'|;k).
\end{equation}

If $X\in P_B\setminus Q_B$, then $\rk_N(X\cup\{b\})=s$ for every $b\in B$.  A single element raises rank by at most one, so $\rk_NX=s-1$, and no element of $B$ lies in $\cl_NX$.  Conversely, every hyperplane disjoint from $B$ has these properties.  Since $c$ preserves $Q_B$,
\begin{equation}\label{eq:relative-hyperplanes}
c(P_B)\setminus c(Q_B)
=\{H:H\text{ is a hyperplane of }N,\ H\cap B=\varnothing\}.
\end{equation}

Fix such an $H$ and set
\[
\mathcal F_B(H)=
\{F\in\mathcal L(N):F\subsetneq H,\ \rk_N(F\cup B)=s\}.
\]
The elements of $c(P_B)$ strictly below $H$ are exactly the members of $\mathcal F_B(H)$.  Every relative chain contains a unique hyperplane from \eqref{eq:relative-hyperplanes}, and that hyperplane is its maximum.  Deleting it produces a chain in $c(Q_B)$, while deleting another vertex remains in the same block.  Thus the relative cochain complex splits, already at chain level, into the blocks
\[
\bigl(\operatorname{cone}\Deltao\mathcal F_B(H),
\Deltao\mathcal F_B(H)\bigr).
\]
It follows that
\begin{align}
T^1_{-\one_B}(N^*)&\cong
\bigoplus_{H\cap B=\varnothing}
\wtH^{-1}(\Deltao\mathcal F_B(H);k),
\label{eq:block-T1-flat}\\
T^2_{-\one_B}(N^*)&\cong
\bigoplus_{H\cap B=\varnothing}
\wtH^0(\Deltao\mathcal F_B(H);k).
\label{eq:block-T2-flat}
\end{align}

For $G\in\mathcal G_B(H)$, let
\[
\mathcal U_G=\Deltao\{F\in\mathcal F_B(H):F\subseteq G\}.
\]
Each $\mathcal U_G$ is a cone with apex $G$.  These cones cover $\Deltao\mathcal F_B(H)$: any qualifying flat extends within $H$ to a rank-$(s-2)$ flat, and the spanning condition is preserved under enlargement.  For $G_1,\ldots,G_m$, put $K=\bigcap_iG_i$.  The corresponding intersection is nonempty exactly when $\rk_N(K\cup B)=s$; in that case $K$ is its maximum.  Hence all nonempty finite intersections are contractible, and the finite nerve lemma \cite[Theorem~10.6(i), p.~1850]{Bjorner} applies.  The one-skeleton of the nerve is precisely $\Gamma_B(H)$.  Reduced $H^0$ depends only on this one-skeleton, so
\[
\wtH^0(\Deltao\mathcal F_B(H);k)
\cong\wtH^0(\Gamma_B(H);k).
\]
Moreover $\mathcal F_B(H)$ is empty exactly when $\mathcal G_B(H)$ is empty.  Equations \eqref{eq:block-T1-flat} and \eqref{eq:block-T2-flat} now prove \eqref{eq:universal-T1}--\eqref{eq:universal-dimension}.  The low-rank cases are included by the empty-graph convention.

If $B=\{p\}$, adding $p$ raises rank by at most one, so every member of the complementary poset has rank $s-1$ and has closure avoiding $p$.  Closure retracts this poset to the antichain of hyperplanes avoiding $p$.  The reduced singleton form of the Altmann--Christophersen formula gives \eqref{eq:singleton-formula}.  Finally, \eqref{eq:link-formula} transports all formulas back to $M$.
\end{proof}

\begin{remark}[Relation to the known $T^1$ description]\label{rem:T1-attribution}
For $T^1$, the formulas \eqref{eq:universal-T1} and \eqref{eq:singleton-formula} are, after dualization and \eqref{eq:link-formula}, a reformulation of the combinatorial description of the first cotangent cohomology of matroids obtained by Brehm and Constantinescu \cite{BrehmConstantinescu}.  We include the $T^1$ statements because the proof above yields both cohomological degrees simultaneously and because the singleton identification \eqref{eq:singleton-formula} supplies the coordinates for the cup products of Section~\ref{sec:cup}.  The graph $\Gamma_B(H)$ and the resulting $T^2$ statements are the contributions developed here.
\end{remark}

\begin{corollary}[Complete fine Hilbert series]\label{cor:complete-Hilbert}
For $A\in M$ and $B\subseteq V(\lk_MA)$ with $|B|\ge2$, write $\Gamma^A_B(H)$ for the graph of \eqref{eq:graph-adjacency} associated with the dual matroid $(\lk_MA)^*$ and the set $B$, and define
\[
\omega_M(A,B)=
\sum_{\substack{H\text{ hyperplane of }(\lk_MA)^*\\H\cap B=\varnothing}}
\bigl(\kappa(\Gamma^A_B(H))-1\bigr)_+.
\]
Then
\begin{equation}\label{eq:complete-Hilbert}
\Hilb_{T^2(M)}(\mathbf t)=
\sum_{A\in M}
\left(\prod_{a\in A}\frac{t_a}{1-t_a}\right)
\left(
\sum_{\substack{B\subseteq V(\lk_MA)\\|B|\ge2}}
\omega_M(A,B)\prod_{b\in B}t_b^{-1}
\right).
\end{equation}
In particular, every multigraded dimension of $T^2(M)$ is independent of $\operatorname{char}k$.
\end{corollary}

\begin{proof}
The support criterion leaves precisely the displayed supports.  For fixed $A$, positive exponents contribute
\[
\prod_{a\in A}(t_a+t_a^2+\cdots)=
\prod_{a\in A}\frac{t_a}{1-t_a}.
\]
The graph formula supplies the coefficient of every squarefree negative support.  Singleton supports contribute nothing to $T^2$.
\end{proof}

The formula admits a useful closure compression.  For a flat $K$ of $N$ and a hyperplane $H$, define $\Gamma_K(H)$ by replacing $B$ with $K$ in all rank conditions in \eqref{eq:graph-adjacency}, and put
\[
\omega_N(K,H)=\dim_k\wtH^0(\Gamma_K(H);k).
\]
Let
\[
\Sigma_{K,H}(\mathbf z)=
\sum_{\substack{B\subseteq K\setminus H,\ |B|\ge2\\\cl_NB=K}}
\prod_{b\in B}z_b.
\]

\begin{corollary}[Closure-stratified series]\label{cor:closure-strata}
For $M=N^*$,
\begin{equation}\label{eq:closure-stratified}
\sum_{\substack{B\subseteq V(M)\\|B|\ge2}}
\dim_kT^2_{-\one_B}(M)\prod_{b\in B}z_b
=
\sum_{K\in\mathcal L(N)}\sum_{H\text{ hyperplane}}
\omega_N(K,H)\Sigma_{K,H}(\mathbf z).
\end{equation}
If $I\subseteq B$ and $\cl_NI=\cl_NB$, then
\[
\dim_kT^2_{-\one_B}(M)\le
\dim_kT^2_{-\one_I}(M).
\]
If $B$ spans $N$, then $T^2_{-\one_B}(M)=0$.
\end{corollary}

\begin{proof}
Every graph rank condition depends on $B$ only through $\cl_NB$.  For fixed $K,H$, the supports contributing to the block are precisely the subsets of $K\setminus H$ which span $K$, proving \eqref{eq:closure-stratified}.  If $I\subseteq B$ has the same closure, the graphs agree and the hyperplanes avoiding $B$ form a subcollection of those avoiding $I$.  If $B$ spans, the bottom flat is a maximum-compatible cone point in every nonempty flat block, so each block is contractible.
\end{proof}

\section{The pair slice and cocircuit witnesses}\label{sec:pairs}

Suppose $B=\{p,q\}$.  Distinct rank-$(s-2)$ flats $G,G'$ satisfy $\rk(G\cap G')\le s-3$, and adjoining two elements raises rank by at most two.  Therefore $\Gamma_{\{p,q\}}(H)$ is discrete.  We obtain the following exact pair formula.

\begin{corollary}[Two-element block formula]\label{cor:pair-block}
Let $M=N^*$ and let $p,q$ be distinct.  Then
\begin{equation}\label{eq:pair-block}
T^2_{-\ee_p-\ee_q}(M)
\cong
\bigoplus_{\substack{H\text{ hyperplane of }N\\H\cap\{p,q\}=\varnothing}}
k^{\mathcal G_{\{p,q\}}(H)}/k\one,
\end{equation}
and
\begin{equation}\label{eq:pair-dim}
\dim_kT^2_{-\ee_p-\ee_q}(M)=
\sum_{H\cap\{p,q\}=\varnothing}
(|\mathcal G_{\{p,q\}}(H)|-1)_+.
\end{equation}
If $\{p,q\}$ is dependent in $N$, this component vanishes.
\end{corollary}

\begin{proof}
Only the last assertion needs comment.  If $p,q$ are parallel nonloops, a rank-$(s-2)$ flat together with $\{p,q\}$ cannot span rank $s$, so all vertex sets are empty.  If one is a loop, no hyperplane avoids the pair.
\end{proof}

\begin{proposition}[Purely negative cocircuit witnesses]\label{prop:pure-cocircuit}
Let $D$ be a cocircuit of $M$, let $d=|D|\ge4$, and let $p,q\in D$ be distinct.  Then
\[
\dim_kT^2_{-\ee_p-\ee_q}(M)\ge d-3.
\]
\end{proposition}

\begin{proof}
In $N=M^*$, the set $D$ is a circuit.  Put $I=D\setminus\{p,q\}$ and extend $I\cup\{p\}$ to a basis $I\cup\{p\}\cup Z$.  Let
\[
H=\cl_N(I\cup Z),\qquad
G_x=\cl_N((I\setminus\{x\})\cup Z),\quad x\in I.
\]
The fundamental circuit of $q$ with respect to this basis is $D$. Removing
$p$ from the basis gives the independent set $I\cup Z$ of rank $\rk N-1$;
neither $p$ nor $q$ belongs to its closure, the latter because its fundamental
circuit contains $p$. For $x\in I$, the set $(I\setminus\{x\})\cup Z$ has
rank $\rk N-2$, while adjoining $p,q$ gives a basis: this is fundamental
exchange at $x\in D$. Thus $G_x$ is a vertex of
$\Gamma_{\{p,q\}}(H)$. If $x\ne y$, then $x\in G_y$ but $x\notin G_x$,
so these vertices are distinct. There are $d-2$ of them, and
\eqref{eq:pair-dim} gives the bound.
\end{proof}

\begin{proposition}[Exact rank-one-link component]\label{prop:exact-cocircuit}
Let $D$ be a cocircuit of $M$, let $A$ be a basis of the complementary hyperplane, and put $d=|D|$.  For distinct $p,q\in D$,
\begin{equation}\label{eq:exact-cocircuit}
\dim_kT^2_{\one_A-\ee_p-\ee_q}(M)=(d-3)_+.
\end{equation}
\end{proposition}

\begin{proof}
The link $\lk_MA$ is the zero-dimensional complex on $D$.  Apply \eqref{eq:link-formula} and Lemma~\ref{lem:zero-dimensional}.
\end{proof}

\section{Large cocircuits and the vanishing classification}\label{sec:vanishing}

\begin{lemma}[A bipartite graph lemma]\label{lem:bipartite}
Let $G$ be a connected bipartite graph with bipartition $L\sqcup R$, where $|L|\ge2$ and $|R|\ge3$.  There are distinct $x,y\in L$ such that
\[
N_G(x)\cap N_G(y)\ne\varnothing,
\qquad |N_G(x)\cup N_G(y)|\ge3.
\]
\end{lemma}

\begin{proof}
Choose three vertices of $R$, a spanning tree of $G$, and the minimal subtree $T$ containing the chosen vertices.  Every leaf of $T$ is one of the chosen vertices and hence lies in $R$.

If some $z\in R\cap V(T)$ has degree at least two in $T$, choose distinct neighbors $x,y\in L$.  Neither is a leaf, so each has a further neighbor in $R$; the two further neighbors are distinct because $T$ has no cycle.  Otherwise every $R$-vertex of $T$ has degree one.  Bipartiteness then forces $T$ to be a star with center $x\in L$, adjacent to all three chosen vertices.  Since $|L|\ge2$ and $G$ is connected, a path from $x$ to another $L$-vertex begins $x,z,y$.  Thus $x,y$ have a common neighbor, while $x$ already has the three chosen neighbors.
\end{proof}

\begin{proposition}[Rank-two reduction]\label{prop:rank-two-reduction}
Let $M$ be connected of rank $r\ge2$ and corank at least three.  From any basis $B$ one can find distinct $x,y\in B$ such that, after contracting $A_0=B\setminus\{x,y\}$ and deleting loops, the resulting rank-two matroid has at least five elements and at least three parallel classes; in particular, it is connected.
\end{proposition}

\begin{proof}
Apply Lemma~\ref{lem:bipartite} to the connected fundamental graph $G_B(M)$, with sides $B$ and $E\setminus B$.  The chosen $x,y$ have a common nonbasis neighbor and at least three nonbasis neighbors in total.  After contracting $A_0$, all these neighbors remain nonloops, so together with $x,y$ there are at least five elements.  If $e$ is a common neighbor, then
\[
C_{M/A_0}(e;\{x,y\})=\{x,y,e\},
\]
so $x,y,e$ lie in three different parallel classes.  Finally, a loopless rank-two matroid with at least three parallel classes is connected: a disconnected one is a direct sum of two nonempty loopless matroids of rank one and has exactly two parallel classes.
\end{proof}

\begin{theorem}[Constructive large-cocircuit lemma]\label{thm:large-cocircuit}
Every connected matroid of corank at least three has a cocircuit of size at least four.  If $r=\rk M\ge2$, there is an independent $(r-1)$-set $A$ such that
\[
D=E(M)\setminus\cl_MA
\]
is such a cocircuit.
\end{theorem}

\begin{proof}
A connected rank-zero matroid has at most one element, so the corank
hypothesis excludes that case. If $r=1$, connectedness gives $M=U_{1,n}$
with $n\ge4$, and $E$ itself is a cocircuit. Assume $r\ge2$ and use
Proposition~\ref{prop:rank-two-reduction}. Let $N$ be the loopless rank-two
contraction, let $m=|E(N)|\ge5$, and choose a smallest parallel class $P$.
Since there are at least three classes,
\[
|E(N)\setminus P|\ge m-\lfloor m/3\rfloor\ge4.
\]
Choose $v\in P$ and put $A=A_0\cup\{v\}$.  The loops created by contracting $A_0$ already lie in $\cl_MA_0$, while the other elements in $\cl_MA$ correspond exactly to $P$.  Thus
\[
E(M)\setminus\cl_MA=E(N)\setminus P.
\]
The closure of $A$ is a hyperplane, so this complement is the required cocircuit.
\end{proof}

\begin{theorem}[Vanishing classification]\label{thm:vanishing}
For a finite matroid $M$, the following are equivalent:
\begin{enumerate}[label=\textup{(\arabic*)}]
\item $T^2(M)=0$;
\item $T^2_{-\ee_p-\ee_q}(M)=0$ for all distinct $p,q\in E(M)$;
\item every connected component of $M$ has corank at most two;
\item $M$ is a direct sum of matroids of corank at most two;
\item every cocircuit of $M$ has size at most three;
\item $M$ has no $U_{1,4}$-minor.
\end{enumerate}
\end{theorem}

\begin{proof}
The first condition implies the second.  If the second holds, Proposition~\ref{prop:pure-cocircuit} excludes cocircuits of size at least four, giving (5).  If (5) holds and some connected component had corank at least three, Theorem~\ref{thm:large-cocircuit} would give a contradiction.  Thus (3) holds, and (3) is equivalent to (4) by the component decomposition.  Condition (4) implies (1) by the known corank-two vanishing and \eqref{eq:direct-sum-T2}.  Also, (3) directly implies (5), since cocircuits lie in components and are circuits of duals of rank at most two.

If $D$ is a cocircuit and $H=E\setminus D$, then $M/H\cong U_{1,|D|}$.  A cocircuit of size at least four therefore produces a $U_{1,4}$-minor.  Conversely, deletion and contraction do not increase corank, and minors of a direct sum are direct sums of minors.  Hence a direct sum of corank-at-most-two components has no $U_{1,4}$-minor.  This proves the remaining equivalence.
\end{proof}

\begin{corollary}\label{cor:conjecture}
The vanishing conjecture in \cite[Conjecture~22]{CKNSV} holds.  The class of matroids with $T^2=0$ is minor-closed and has the unique excluded minor $U_{1,4}$.
\end{corollary}

\begin{remark}
No monotonicity assertion for the module $T^2$ under arbitrary minors is needed.  Only its vanishing class is being identified with a minor-closed matroid class.
\end{remark}

\section{The rank-one-link strand and the invariant \texorpdfstring{$\Theta$}{Theta}}\label{sec:theta}

Assume in this section that $r=\rk M\ge1$.  For every independent $(r-1)$-set $A$, put
\[
D_A=E(M)\setminus\cl_MA,
\qquad d_A=|D_A|.
\]
The set $D_A$ is a cocircuit.

\begin{theorem}[Complete rank-one-link strand]\label{thm:top-strand}
Let $\mathbf c=\mathbf a-\mathbf b$ have disjoint nonnegative parts, suppose $|\supp\mathbf a|=r-1$, and put $A=\supp\mathbf a$.  Then
\[
\dim_kT^2_{\mathbf c}(M)=
\begin{cases}
(d_A-3)_+,&A\in M,\ \mathbf b=\ee_p+\ee_q,
\ \{p,q\}\in\binom{D_A}{2},\\
0,&\text{otherwise}.
\end{cases}
\]
Thus every nonzero component in this strand has two-element squarefree negative support, and its dimension is independent of the positive exponents and of $k$.
\end{theorem}

\begin{proof}
For an independent $(r-1)$-set $A$, the link is $Z_{D_A}$.  Apply the support criterion, \eqref{eq:link-formula}, and Lemma~\ref{lem:zero-dimensional}.
\end{proof}

Define the finite squarefree witness space
\[
W(M)=
\bigoplus_{\substack{A\in M\\|A|=r-1}}
\bigoplus_{\{p,q\}\in\binom{D_A}{2}}
T^2_{\one_A-\ee_p-\ee_q}(M)
\]
and
\begin{equation}\label{eq:Theta-definition}
\Theta(M)=\dim_kW(M)=
\sum_{\substack{A\in M\\|A|=r-1}}
\binom{d_A}{2}(d_A-3)_+.
\end{equation}
For a rank-zero matroid we set $W(M)=0$ and $\Theta(M)=0$.

\begin{corollary}[Finite detector]\label{cor:finite-detector}
For every matroid $M$,
\[
T^2(M)=0\quad\Longleftrightarrow\quad W(M)=0
\quad\Longleftrightarrow\quad\Theta(M)=0.
\]
If $M$ is connected, $r\ge2$, and $\corank M\ge3$, neither $W(M)$ nor the full rank-one-link strand generates $T^2(M)$ as a $k[M]$-module.
\end{corollary}

\begin{proof}
Every cocircuit is complementary to a hyperplane, and every hyperplane has a basis $A$.  Theorem~\ref{thm:top-strand} and Theorem~\ref{thm:vanishing} prove the equivalences.  Under the hypotheses in the second assertion, Proposition~\ref{prop:pure-cocircuit} supplies a nonzero purely negative component.  Multiplication by a monomial only increases fine degree, so an element with nonempty positive support cannot generate it.
\end{proof}

\begin{corollary}[Fine and coarse series]\label{cor:top-series}
The fine Hilbert series of the rank-one-link strand is
\begin{equation}\label{eq:top-fine-series}
\sum_{\substack{A\in M\\|A|=r-1}}
(d_A-3)_+
\left(\sum_{\{p,q\}\in\binom{D_A}{2}}t_p^{-1}t_q^{-1}\right)
\prod_{a\in A}\frac{t_a}{1-t_a}.
\end{equation}
If $b(N)$ denotes the number of bases of $N$, then
\begin{equation}\label{eq:Theta-cocircuit}
\Theta(M)=
\sum_{D\text{ cocircuit of }M}
b(M|(E\setminus D))\binom{|D|}{2}(|D|-3)_+.
\end{equation}
After setting every $t_e=t$ in \eqref{eq:top-fine-series}, the series becomes
\begin{equation}\label{eq:top-coarse-series}
\frac{t^{r-3}}{(1-t)^{r-1}}\Theta(M).
\end{equation}
In particular, for $j\ge0$ and $r\ge2$, the part of standard degree $r-3+j$ has dimension
\[
\Theta(M)\binom{r+j-2}{r-2}.
\]
\end{corollary}

\begin{proof}
The first formula follows by summing the positive exponents in Theorem~\ref{thm:top-strand}.  Grouping $(r-1)$-sets by their closures gives \eqref{eq:Theta-cocircuit}: the sets spanning $E\setminus D$ are its bases.  The specialization and coefficient extraction are immediate.
\end{proof}

Put
\[
f(d)=\binom d2(d-3)_+,
\qquad
h(d)=
\begin{cases}
0,&d\le3,\\
\dfrac{(d-1)(d-3)}2,&d\ge4.
\end{cases}
\]
The quantity $h(d)$ need not be an integer; for example, $h(4)=\tfrac32$.  Only the aggregated sums below are integral.

\begin{theorem}[Circuit profile and basis average]\label{thm:Theta-profiles}
Let $N=M^*$.  Then
\begin{equation}\label{eq:Theta-circuit}
\Theta(M)=\sum_{C\text{ circuit of }N}b(N/C)f(|C|).
\end{equation}
Equivalently, if $q_d(N)$ counts spanning subsets of size $\rk N+1$ whose unique circuit has size $d$, then
\[
\Theta(M)=\sum_d f(d)q_d(N).
\]
For a basis $B$ of $M$ and $e\in B$, let
\[
D^*(B;e)=E\setminus\cl_M(B\setminus\{e\}),
\qquad d_B(e)=|D^*(B;e)|.
\]
Then
\begin{equation}\label{eq:Theta-basis-average}
\Theta(M)=\sum_{B\text{ basis of }M}\sum_{e\in B}h(d_B(e)).
\end{equation}
\end{theorem}

\begin{proof}
Duality gives $(M|(E\setminus D))^*=N/D$ for a cocircuit $D$, so \eqref{eq:Theta-cocircuit} becomes \eqref{eq:Theta-circuit}.  Extending a basis of $N/C$ by $C$ is a bijection to spanning nullity-one sets with unique circuit $C$.  For the basis average, map $(B,e)$ to $A=B\setminus\{e\}$.  A fixed $A$ has exactly $d_A$ preimages, and $d_Ah(d_A)=f(d_A)$.
\end{proof}

\begin{corollary}[Sharp uniform bounds]\label{cor:uniform-bounds}
Let $M$ have rank $r\ge1$, $n$ elements, and corank $c=n-r\ge3$.  Then
\begin{align}
\Theta(M)&\le b(M)\,r\,\frac{c(c-2)}2,
\label{eq:uniform-bound-one}\\
\Theta(M)&\le
\binom n{r-1}\binom{c+1}{2}(c-2).
\label{eq:uniform-bound-two}
\end{align}
Equality in either bound holds if and only if $M=U_{r,n}$.
\end{corollary}

\begin{proof}
Every fundamental cocircuit has size at most $c+1$, and $h$ is increasing on the relevant range.  Formula \eqref{eq:Theta-basis-average} proves \eqref{eq:uniform-bound-one}.  Formula \eqref{eq:Theta-definition} proves \eqref{eq:uniform-bound-two}.  Equality forces every independent $(r-1)$-set to be closed.  If a circuit had size at most $r$, deleting one of its elements and extending the remainder to an $(r-1)$-set would contradict closedness.  Hence every set of size at most $r$ is independent, so $M=U_{r,n}$.  The converse is immediate.
\end{proof}

\begin{corollary}[Direct sums]\label{cor:Theta-direct-sum}
For matroids on disjoint ground sets,
\[
\Theta(M_1\oplus M_2)=
\Theta(M_1)b(M_2)+b(M_1)\Theta(M_2).
\]
Consequently $\Theta(M)/b(M)$ is additive over connected components.
\end{corollary}

\begin{proof}
An independent set one element short of a basis of a direct sum is a basis in one summand and one element short in the other.  Its complementary cocircuit lies in the latter.  Split \eqref{eq:Theta-definition} into the two cases.
\end{proof}

\begin{theorem}[Tutte contrast]\label{thm:Tutte-contrast}
The Tutte polynomial of $M$ determines whether $T^2(M)$ vanishes.  It does not determine $\Theta(M)$ or the coarse rank-one-link Hilbert series.
\end{theorem}

\begin{proof}
The Tutte polynomial is multiplicative under direct sums, and the polynomial of a connected matroid is irreducible in $\ZZ[x,y]$ \cite{MerinoDeMierNoy}.  Unique factorization therefore recovers the polynomial of each connected component.  The $y$-degree is its corank, so Theorem~\ref{thm:vanishing} decides vanishing.

For the negative statement, let $G_1,G_2$ be the graphs on $\{0,\ldots,6\}$ with edge sets
\begin{align*}
E(G_1)&=\{04,06,14,15,16,23,25,26,34,35,36,45,56\},\\
E(G_2)&=\{01,02,03,05,12,14,15,23,25,26,34,45,56\}.
\end{align*}
Direct rank-subset expansion gives identical Tutte polynomials.  The numbers of spanning two-forests with $|E\setminus\cl(A)|=d$, for $d=2,\ldots,8$, are respectively
\[
(209,235,141,130,101,41,5)
\]
and
\[
(225,198,162,138,89,46,4).
\]
Weighting by $f(d)$ gives
\[
\Theta(M(G_1))=12135,
\qquad
\Theta(M(G_2))=12161.
\]
The computation is independently verified as described in Section~\ref{sec:verification}.
\end{proof}

\begin{theorem}[Reconstruction of large cocircuits]\label{thm:reconstruction}
For each independent $(r-1)$-set $A$, form a labeled graph on $E\setminus A$ by making $\{p,q\}$ an edge exactly when
\[
\dim_kT^2_{\one_A-\ee_p-\ee_q}(M)>0.
\]
If $d_A\le3$, this graph is empty.  If $d_A\ge4$, its nonisolated vertex set is $D_A$, the induced graph on $D_A$ is complete, and every edge has label $d_A-3$.  Hence the family of labeled fine components reconstructs all cocircuits of size at least four, including their sizes.  A fixed cocircuit $D$ occurs $b(M|(E\setminus D))$ times.
\end{theorem}

\begin{proof}
This is an immediate restatement of Theorem~\ref{thm:top-strand}.  Every cocircuit is the complement of a hyperplane and every hyperplane has a basis.
\end{proof}

\section{Rank-two links and the next-to-top strand}\label{sec:rank-two}

\begin{theorem}[Purely negative $T^2$ in rank two]\label{thm:rank-two-formulas}
Let $L$ be a loopless rank-two matroid with parallel classes $P_1,\ldots,P_t$ of sizes $n_1,\ldots,n_t$, and put $n=\sum_i n_i$.  Every purely negative component vanishes except in the following cases.
\begin{enumerate}[label=\textup{(\arabic*)}]
\item If $B=\{p,q\}\subset P_i$, then
\[
\dim_kT^2_{-\one_B}(L)=
\left(n_i-3+
\sum_{\substack{a<c\\a,c\ne i}}n_an_c\right)_+.
\]
\item If $p\in P_i$, $q\in P_j$, and $i\ne j$, then
\[
\dim_kT^2_{-\ee_p-\ee_q}(L)=
\sum_{k\ne i,j}n_k(n-3-n_k)_+.
\]
\item Suppose $B$ consists of three elements from three distinct parallel classes.  If all elements of $L\setminus B$ lie in at most one parallel class, then
\[
\dim_kT^2_{-\one_B}(L)=(n-4)_+;
\]
otherwise the component is zero.
\end{enumerate}
Every singleton negative component of $T^2(L)$ vanishes.
\end{theorem}

\begin{proof}
The circuits of $L$ are the pairs within a parallel class and the triples meeting three different classes.  A squarefree negative support which is a nonface but not a circuit has equal relative terms and contributes zero.

In (1), $B$ is a circuit. The order-complex incidence graph
$|\widetilde N_B(L)'|$ has one component for each of the $n_i-2$
remaining points of $P_i$, containing that point and all incident
edge-face vertices. Every edge joining two classes different from $P_i$
gives an isolated edge-face vertex. Reduced $H^0$ gives the formula.

In (2), $B$ is a basis.  Relative vertices are the points in $P_k$, $k\ne i,j$.  A point of $P_k$ is incident with $n-2-n_k$ edge-faces after deleting $p,q$, and its relative incidence block contributes $(n-3-n_k)_+$.  Summing over the $n_k$ points gives (2).

In (3), $B$ is a circuit and $|\widetilde N_B(L)'|$ is the barycentric subdivision of $L\setminus B$. It is connected if at least two parallel classes remain, and otherwise it is a discrete set of $n-3$ points. Finally, the singleton Altmann--Christophersen complex is an incidence graph whose two-face vertices have degree at most one. It is a forest, so its first cohomology vanishes.
\end{proof}

\begin{corollary}[Complete next-to-top strand]\label{cor:next-to-top}
Let $M$ have rank $r\ge2$, let $\mathbf a\in\NN^E$ have independent support $A$ of size $r-2$, and let $\mathbf b$ be nonzero, squarefree, and supported on actual vertices of $L=\lk_MA$.  Then
\[
T^2_{\mathbf a-\mathbf b}(M)\cong T^2_{-\mathbf b}(L),
\]
and Theorem~\ref{thm:rank-two-formulas} computes the right-hand side.  All remaining components with positive-support size $r-2$ vanish by the support criterion.
\end{corollary}

\begin{corollary}[Projective planes]\label{cor:projective-plane}
Let $q\ge2$ be a prime power, let $M=PG(2,q)$, and put $\nu=q^2+q+1$,
\[
\alpha=q-3+\binom q2q^2,
\qquad
\gamma=q(q-1)(q^2-3).
\]
Then
\begin{align}
\dim_kT^2(M)_0
&=\Theta(M)+\nu\left((q+1)\binom q2\alpha
+\binom{q+1}{2}q^2\gamma\right)\label{eq:PG-first}\\
&=\frac{q(q-1)(q+1)(q^2+q+1)}4
\left(3q^5+2q^4-10q^3-3q^2+2q-6\right).
\label{eq:PG-polynomial}
\end{align}
In particular, $\dim_kT^2(F_7)_0=357$, with $126$ dimensions in the rank-one-link strand and $231$ in the rank-two-link strand.
\end{corollary}

\begin{proof}
Every point link has $q+1$ parallel classes of size $q$.  The within-class and between-class pair contributions are exactly the first two cases of Theorem~\ref{thm:rank-two-formulas}; triple contributions vanish.  The support criterion shows that the rank-one and rank-two link contributions exhaust standard degree zero.  Formula \eqref{eq:PG-polynomial} is the expansion of \eqref{eq:PG-first} together with
\[
\Theta(PG(2,q))=(q^2+q+1)\binom{q+1}{2}\binom{q^2}{2}(q^2-3).
\]
\end{proof}

\section{Universal cup products and actual obstructions}\label{sec:cup}

Retain the notation $M=N^*$.  For $p\in E$, write
\[
\mathscr H_p=\{H:H\text{ is a hyperplane of }N,\ p\notin H\}.
\]
The singleton part of Theorem~\ref{thm:universal-graph} identifies
\begin{equation}\label{eq:T1-hyperplane-functions}
T^1_{-\ee_p}(M)\cong k^{\mathscr H_p}/k\one,
\end{equation}
where the quotient means the cokernel of the diagonal constant map.

For a hyperplane $H$ avoiding $p,q$ and a flat $G\in\mathcal G_{\{p,q\}}(H)$, put
\begin{equation}\label{eq:PGQG}
P_G=\cl_N(G\cup\{p\}),
\qquad Q_G=\cl_N(G\cup\{q\}).
\end{equation}
Since $G$ has corank two and spans with $p,q$, these are hyperplanes, with
\begin{equation}\label{eq:three-intersections}
q\notin P_G,\quad p\notin Q_G,
\qquad H\cap P_G=H\cap Q_G=P_G\cap Q_G=G.
\end{equation}

\begin{lemma}[Field-independent coordinate formula]\label{lem:cup-coordinate}
Represent source classes in \eqref{eq:T1-hyperplane-functions} by functions
\[
\phi:\mathscr H_p\to k,
\qquad
\psi:\mathscr H_q\to k,
\]
each modulo constants.  Under the block identification \eqref{eq:pair-block}, the $H$-block of their cup product is, up to the fixed orientation of that block,
\begin{equation}\label{eq:block-cup-coordinate}
\left[
\bigl(\phi(Q_G)-\phi(H)\bigr)
\bigl(\psi(P_G)-\psi(H)\bigr)
\right]_{G\in\mathcal G_{\{p,q\}}(H)}
\pmod{k\one}.
\end{equation}
This formula is valid over every field.
\end{lemma}

\begin{proof}
The Stanley--Reisner ideal has its monomial Taylor presentation over $\ZZ$.  The cup product on presentation cochains is obtained by lifting the images of monomial generators and substituting them into the binomial relation module.  This construction uses only addition, multiplication, and the coefficients $0,\pm1$, so its coordinate formula commutes with base change to any field.

For completeness, consider the relative edge indexed by $G\subset H$.  In the notation of Altmann and Christophersen, it is represented by
\[
f=E\setminus(H\cup\{p,q\})
\subset
g=E\setminus(G\cup\{p,q\}).
\]
Substitution in the presentation formula, taking $d=\varnothing$ for $f$ and $e=\{q\}$ for $g$, gives
\begin{equation}\label{eq:raw-cup}
E_H(G)=
(\phi(H)-\phi(Q_G))\psi(P_G)+\psi(H)\phi(Q_G).
\end{equation}
The possible fourth term is zero because its argument does not lie in the singleton set.  Rearranging,
\[
E_H(G)=
-(\phi(Q_G)-\phi(H))(\psi(P_G)-\psi(H))
+\phi(H)\psi(H).
\]
The last term is constant in $G$ and vanishes in the quotient by $k\one$; reversing the block orientation absorbs the remaining sign.  This is precisely \eqref{eq:block-cup-coordinate}.  The calculation is the specialization of \cite[Proposition~19]{AltmannChristophersen}, now observed at the integral monomial-presentation level.
\end{proof}

\begin{theorem}[Universal pure-pair surjectivity]\label{thm:pure-pair-surjectivity}
For every matroid $M$, every pair of distinct actual vertices $p,q$, and every field $k$, the cup product
\begin{equation}\label{eq:pure-pair-cup}
T^1_{-\ee_p}(M)\otimes_kT^1_{-\ee_q}(M)
\longrightarrow T^2_{-\ee_p-\ee_q}(M)
\end{equation}
is surjective.
\end{theorem}

\begin{proof}
Use $M=N^*$ and fix a target coordinate $(H_0,G_0)$ in \eqref{eq:pair-block}.  Put
\[
P_0=P_{G_0},\qquad Q_0=Q_{G_0},
\]
and
\[
\mathcal S(G_0)=
\{H:H\cap\{p,q\}=\varnothing,\ G_0\subset H\}.
\]
Define representatives of singleton classes by
\[
\phi(K)=
\begin{cases}
1,&K=Q_0,\\
1,&K\in\mathcal S(G_0)\setminus\{H_0\},\\
0,&\text{otherwise},
\end{cases}
\qquad
\psi(K)=
\begin{cases}
1,&K=P_0,\\
0,&\text{otherwise}.
\end{cases}
\]
There is no conflict: $Q_0$ contains $q$, $P_0$ contains $p$, while every member of $\mathcal S(G_0)$ avoids both.

For any coordinate $(H,G)$, the second factor in \eqref{eq:block-cup-coordinate} can be nonzero only when $P_G=P_0$.  If also $G=G_0$, then $H\in\mathcal S(G_0)$ and the first factor is nonzero exactly for $H=H_0$.  If $P_G=P_0$ but $G\ne G_0$, then $Q_G\ne Q_0$ by \eqref{eq:three-intersections}.  Moreover $H\notin\mathcal S(G_0)$: otherwise
\[
G_0\subseteq H\cap P_0=G
\]
would force equality of the two rank-$(s-2)$ flats.  Hence the first factor is zero.  The resulting product is the coordinate delta vector at $(H_0,G_0)$, up to sign.  Such vectors span before passing to blockwise constants, so their images span the target.
\end{proof}

\begin{corollary}[Link-lifted cup products]\label{cor:link-lifted-cup}
Let $A\in M$, and let $p,q$ be distinct actual vertices of $L=\lk_MA$.  Then
\[
T^1_{\one_A-\ee_p}(M)\otimes_k
T^1_{\one_A-\ee_q}(M)
\twoheadrightarrow
T^2_{2\one_A-\ee_p-\ee_q}(M).
\]
The target is identified with $T^2_{-\ee_p-\ee_q}(L)$.
\end{corollary}

\begin{proof}
Iterate Lemma~\ref{lem:localization} over the variables indexed by $A$, use \eqref{eq:link-formula}, and transport Theorem~\ref{thm:pure-pair-surjectivity} across the resulting cup-compatible isomorphisms.
\end{proof}

The next lemma extracts a singleton link whose contraction still has large corank.

\begin{lemma}[One contraction retaining large corank]\label{lem:one-contraction}
Let $M$ be connected with $\rk M\ge2$ and $\corank M\ge3$.  There is an element $e$ such that $M/e$, after deleting its loops, has a connected component of corank at least three.
\end{lemma}

\begin{proof}
If $\rk M=2$, choose $e$ in a smallest parallel class $P$.  A connected rank-two matroid has at least three parallel classes, and $|E|\ge5$, so
\[
|E\setminus P|\ge |E|-\lfloor |E|/3\rfloor\ge4.
\]
After contracting $e$ and deleting $P\setminus\{e\}$, one obtains a connected rank-one matroid of corank at least three.

Suppose $\rk M\ge3$.  Proposition~\ref{prop:rank-two-reduction} gives an independent set $A_0$ of size $\rk M-2$ such that $M/A_0$, after deleting loops, is a connected rank-two matroid $R$ on at least five elements.  Choose $e\in A_0$ and let $L=M/e$ with its loops deleted.  Contracting $A_0\setminus\{e\}$ in $L$, and then deleting loops, produces $R$.  Contraction acts componentwise and cannot join components.  Contracting nonloops preserves corank, while deleting newly created loops can only decrease it.  Hence all nonloops of $R$ descend from a single connected component $C$ of $L$, with
\[
\corank C\ge\corank R\ge3.
\]
\end{proof}

\begin{theorem}[Nonzero standard-degree-zero cup product]\label{thm:degree-zero-cup}
Let $M$ be connected with $\rk M\ge2$ and $\corank M\ge3$.  There are distinct $e,p,q\in E(M)$ such that
\begin{equation}\label{eq:degree-zero-surjection}
T^1_{\ee_e-\ee_p}(M)\otimes_k
T^1_{\ee_e-\ee_q}(M)
\twoheadrightarrow
T^2_{2\ee_e-\ee_p-\ee_q}(M)\ne0.
\end{equation}
Thus $T^1(M)_0\otimes_kT^1(M)_0\to T^2(M)_0$ has nonzero image over every field.
\end{theorem}

\begin{proof}
Choose $e$ by Lemma~\ref{lem:one-contraction}, put $L=M/e$ after deleting loops, and let $C$ be a connected component of $L$ with corank at least three.  By Theorem~\ref{thm:large-cocircuit}, $C$ has a cocircuit $D$ of size at least four.  Choose distinct $p,q\in D$.  Proposition~\ref{prop:pure-cocircuit} gives
\[
T^2_{-\ee_p-\ee_q}(C)\ne0.
\]
If $L=C\oplus C'$, the direct-sum formula, or equivalently the dual hyperplane blocks, embeds this component into the same fine component for $L$.  Theorem~\ref{thm:pure-pair-surjectivity} gives a surjection onto it.  Localization at $x_e$ and evaluation at $x_e=1$ give cup-compatible isomorphisms
\begin{align*}
T^1_{\ee_e-\ee_p}(M)&\cong T^1_{-\ee_p}(L),\\
T^1_{\ee_e-\ee_q}(M)&\cong T^1_{-\ee_q}(L),\\
T^2_{2\ee_e-\ee_p-\ee_q}(M)&\cong T^2_{-\ee_p-\ee_q}(L).
\end{align*}
Transporting the surjection proves \eqref{eq:degree-zero-surjection}.  Each fine degree has total degree zero.
\end{proof}

Let $\operatorname{Def}_M$ denote the affine deformation functor of $k[M]$
on local Artinian $k$-algebras with residue field $k$, and let
$\operatorname{Def}^{\mathrm{gr}}_M$ denote its standard-degree-preserving
subfunctor. Their tangent spaces are $T^1(M)$ and $T^1(M)_0$, respectively;
obstructions to lifting across small extensions lie in $T^2(M)$ and
$T^2(M)_0$. For the cotangent-complex obstruction theory over arbitrary
bases, see \cite[Tag~08UX, Lemma~92.21.1]{Stacks}; for the classical description
of first-order affine deformations by the first cotangent module, see
\cite[Proposition~3.1.1, pp.~103--104]{Sernesi}.
The cup formula of Altmann and Christophersen describes the quadratic
term \cite[Appendix~6, equation~(6.1) and Proposition~19]{AltmannChristophersen}.
When $2$ is invertible, the primary obstruction is one half of the cup
square. The following presentation-level argument identifies its
polarization without dividing by two.

\begin{lemma}[Graded polarization of the primary obstruction]\label{lem:graded-polarization}
Let $k$ be an arbitrary field and let $\operatorname{ob}:T^1(M)\to T^2(M)$ denote the primary obstruction map for lifting first-order deformations from $k[\epsilon]/(\epsilon^2)$ to $k[\epsilon]/(\epsilon^3)$.
\begin{enumerate}[label=\textup{(\arabic*)}]
\item $\operatorname{ob}$ doubles the fine grading: for a fine-homogeneous class $\chi$ of degree $\mathbf c$, the class $\operatorname{ob}(\chi)$ lies in $T^2_{2\mathbf c}(M)$.  This component vanishes, so $\operatorname{ob}(\chi)=0$.
\item For all $\alpha,\beta\in T^1(M)$,
\[
\operatorname{ob}(\alpha+\beta)
=\operatorname{ob}(\alpha)+\operatorname{ob}(\beta)+\alpha\cup\beta.
\]
\end{enumerate}
In particular, for fine-homogeneous classes $\alpha,\beta$ of distinct fine degrees, $\operatorname{ob}(\alpha+\beta)=\alpha\cup\beta$.
\end{lemma}

\begin{proof}
Work with the monomial Taylor presentation of $I_M=(m_1,\ldots,m_r)$.  A first-order deformation with tangent class $\chi$ deforms the generators to $m_i+\epsilon g_i$, where $g_i\in S$ lifts $\chi(m_i)$.  For the pair relation of $(m_i,m_j)$ with least common multiple $u_{ij}$, first-order flatness provides an expression
\[
\frac{u_{ij}}{m_i}\,g_i-\frac{u_{ij}}{m_j}\,g_j
=\sum_k b^{ij}_k m_k
\qquad\text{in }S,
\]
with $b^{ij}_k$ chosen linearly in the tuple $(g_i)$.  A second-order lift $m_i+\epsilon g_i+\epsilon^2h_i$ exists if and only if the presentation cochain
\[
o(\chi):(i,j)\longmapsto
\Bigl[\,\sum_k b^{ij}_k\,g_k\Bigr]\in S/I_M
\]
is a coboundary; its class in $T^2(M)$ is the primary obstruction and is independent of all choices.  Every operation involved preserves the fine multigrading, and $o(\chi)$ is the composition of the linear map $\chi\mapsto b(\chi)$ with multiplication by the values of $\chi$.  Hence $o$ is a quadratic map, and for fine-homogeneous $\chi$ of degree $\mathbf c$ the cochain $o(\chi)$ is homogeneous of degree $2\mathbf c$.  By the support criterion, a nonzero fine component of $T^1$ has degree $\mathbf c=\mathbf a-\one_B$ with $B\ne\varnothing$; the negative part of $2\mathbf c$ has an exponent $2$, so $T^2_{2\mathbf c}(M)=0$ and $\operatorname{ob}(\chi)=0$.  This proves (1).

Expanding $o(\alpha+\beta)$ bilinearly gives
\[
o(\alpha+\beta)=o(\alpha)+o(\beta)
+\Bigl[(i,j)\mapsto\sum_k\bigl(b^{ij}_k(\alpha)\,\beta(m_k)+b^{ij}_k(\beta)\,\alpha(m_k)\bigr)\Bigr].
\]
The mixed cochain is the substitution construction of Lemma~\ref{lem:cup-coordinate}, whose class is $\alpha\cup\beta$; the identity uses only the coefficients $0,\pm1$ of the monomial presentation and holds over every field.  This proves (2), and the final statement follows by combining (1) and (2).
\end{proof}

\begin{corollary}[Actual quadratic obstructions]\label{cor:actual-obstructions}
Let $k$ be an arbitrary field.
\begin{enumerate}[label=\textup{(\arabic*)}]
\item If some connected component of $M$ has corank at least three, then some first-order affine deformation of $k[M]$ does not lift from $k[\epsilon]/(\epsilon^2)$ to $k[\epsilon]/(\epsilon^3)$.
\item If $M$ is connected with $\rk M\ge2$ and $\corank M\ge3$, such a deformation can be chosen in standard degree zero.  Hence $\operatorname{Def}^{\mathrm{gr}}_M$ is not formally smooth, and any miniversal hull has a nonzero quadratic Kuranishi term.
\end{enumerate}
\end{corollary}

\begin{proof}
For (1), Theorem~\ref{thm:vanishing} provides a cocircuit of size at least four, and Proposition~\ref{prop:pure-cocircuit} then provides a nonzero purely negative pair target.  For (2), use the nonzero target in Theorem~\ref{thm:degree-zero-cup}.  In either case, surjectivity gives fine-homogeneous source classes $\alpha,\beta$ of distinct fine degrees with $\alpha\cup\beta\ne0$.  Lemma~\ref{lem:graded-polarization} gives
\[
\operatorname{ob}(\alpha+\beta)=\alpha\cup\beta\ne0,
\]
so the first-order deformation $\alpha+\beta$ does not lift to second order.  In (2), both source classes have standard degree zero.
\end{proof}

\begin{corollary}[Smoothness dichotomy]\label{cor:smoothness-dichotomy}
For every field $k$ and every matroid $M$, the following are equivalent:
\begin{enumerate}[label=\textup{(\arabic*)}]
\item $\operatorname{Def}_M$ is formally smooth;
\item $T^2(M)=0$;
\item every connected component of $M$ has corank at most two;
\item $M$ has no $U_{1,4}$-minor.
\end{enumerate}
\end{corollary}

\begin{proof}
(2)$\Rightarrow$(1) is standard, since all obstructions lie in $T^2(M)$.  (1)$\Rightarrow$(3) is the contrapositive of Corollary~\ref{cor:actual-obstructions}(1).  The equivalences (2)$\Leftrightarrow$(3)$\Leftrightarrow$(4) are part of Theorem~\ref{thm:vanishing}.
\end{proof}

\begin{corollary}[Graded smoothness dichotomy]\label{cor:graded-dichotomy}
Let $M$ be connected.  For every field $k$, the graded functor $\operatorname{Def}^{\mathrm{gr}}_M$ is formally smooth if and only if $\rk M\le1$ or $\corank M\le2$.
\end{corollary}

\begin{proof}
If $\corank M\le2$, then $T^2(M)=0$.  If $\rk M\le1$, the only independent supports are the empty set and singletons; a singleton link of a rank-one matroid has empty actual vertex set, and singleton negative supports vanish by \eqref{eq:singleton-formula}, so every nonzero component of $T^2(M)$ is a purely negative block of negative standard total degree.  Hence $T^2(M)_0=0$ in both cases, and the graded functor, whose obstructions lie in $T^2(M)_0$, is formally smooth.  If $\rk M\ge2$ and $\corank M\ge3$, Corollary~\ref{cor:actual-obstructions}(2) shows that $\operatorname{Def}^{\mathrm{gr}}_M$ is not formally smooth.
\end{proof}

\begin{remark}[The divided square is invisible]\label{rem:char-two}
In characteristic two the primary obstruction is a quadratic refinement of the cup product rather than one half of the cup square.  Lemma~\ref{lem:graded-polarization} shows that for matroid Stanley--Reisner rings this refinement carries no extra information: it vanishes on every fine-homogeneous class, so the whole quadratic map is determined by the bilinear cup product,
\[
\operatorname{ob}\Bigl(\sum_{\mathbf c}\chi_{\mathbf c}\Bigr)
=\sum_{\{\mathbf c,\mathbf c'\},\,\mathbf c\ne\mathbf c'}
\chi_{\mathbf c}\cup\chi_{\mathbf c'},
\]
the sum running over unordered pairs of distinct fine degrees.  In this sense the primary obstruction theory of this paper is characteristic-free.
\end{remark}

The same degree bookkeeping controls the entire lifting tower, not only its first step.  The next theorem shows that fine-homogeneous classes integrate without higher corrections, and that along a pure pair the single quadratic class is the complete obstruction.

\begin{theorem}[Linear arcs and exact quadratic obstructions]\label{thm:arcs}
Let $k$ be an arbitrary field.  Represent classes by multigraded cocycles, so that $\chi\in T^1_{\mathbf c}(M)$ assigns to each circuit generator $m_i$ a multiple $\chi(m_i)\in S$ of the single monomial of multidegree $\deg m_i+\mathbf c$.
\begin{enumerate}[label=\textup{(\arabic*)}]
\item Suppose $\chi=\sum_j\chi_j$ with $\chi_j\in T^1_{\mathbf c_j}(M)$, and suppose that for every pair $(j,j')$, including $j=j'$, the vector $\mathbf c_j+\mathbf c_{j'}$ has an entry at most $-2$.  Then the $t$-linear family
\[
F_i=m_i+t\,\chi(m_i)
\]
is flat over $k[t]/(t^N)$ for every $N$: every Taylor relation lifts with coefficients linear in $t$, and no correction terms are needed.  This applies to every single fine-homogeneous class and, more generally, whenever the negative supports $B_j$ pairwise intersect and $\bigl(\bigcup_j\supp\mathbf a_j\bigr)\cap\bigl(\bigcup_jB_j\bigr)=\varnothing$.
\item Let $\alpha\in T^1_{\mathbf c_1}(M)$ and $\beta\in T^1_{\mathbf c_2}(M)$ with $\mathbf c_i=\mathbf a_i-\one_{B_i}$, $B_1\cap B_2=\varnothing$, and $\supp(\mathbf a_1+\mathbf a_2)\cap(B_1\cup B_2)=\varnothing$.  Then $t(\alpha+\beta)$ lifts to $k[t]/(t^3)$ if and only if $\alpha\cup\beta=0$, and in that case a single multigraded correction $h$ of fine degree $\mathbf c_1+\mathbf c_2$ produces a family
\[
F_i=m_i+t\bigl(\alpha(m_i)+\beta(m_i)\bigr)+t^2h_i
\]
which is flat over $k[t]/(t^N)$ for every $N$.  Along these directions the cup product is the complete obstruction.
\end{enumerate}
\end{theorem}

\begin{proof}
A polynomial of fixed multidegree with a negative entry is zero. Least common multiples $u_{ij}$ of squarefree generators have exponents at most one, so every element of $S$ of multidegree $\deg u_{ij}+\mathbf d$ or $\deg m_i+\mathbf d$ vanishes identically whenever $\mathbf d$ has an entry at most $-2$.

We use the relation-lifting criterion for flatness over
$A_N=k[t]/(t^N)$. If generators $F_i$ lift the generators $m_i$ and every
generating relation among the $m_i$ lifts to a relation among the $F_i$,
then, for $J=(F_i)\subset A_N[x_e:e\in E]$, reduction identifies
$J/tJ$ with $I_M$. Tensoring the ideal sequence with $k$ therefore gives
$\operatorname{Tor}^{A_N}_1(A_N[x]/J,k)=0$. The local flatness criterion
for the nilpotent ideal $(t)$ gives $A_N$-flatness. The Taylor pair
relations generate the relation module of a monomial ideal, so it suffices
to lift those relations.

For (1), lift the Taylor relation of $(m_i,m_j)$ to first order as in Lemma~\ref{lem:graded-polarization}, with linear coefficient $-b^{ij}$.  The complete relation applied to $(F_i)$ has one remaining coefficient, at $t^2$, namely $-b^{ij}\cdot\chi(m_\bullet)$, a sum of terms of multidegrees $\deg u_{ij}+\mathbf c_\ell+\mathbf c_{\ell'}$, where $\ell,\ell'$ index the summands of $\chi$.  Each such term is identically zero by hypothesis, so the linear relation lift is an exact syzygy of $(F_i)$.  Lifting a generating set of relations proves flatness over every $k[t]/(t^N)$.  For the last assertion, an entry of $\mathbf c_\ell+\mathbf c_{\ell'}$ at an element of $B_\ell\cap B_{\ell'}$ equals $-2$ when $\ell\ne\ell'$, and at an element of $B_\ell$ equals $-2$ when $\ell=\ell'$, since the positive supports avoid all negative supports.

For (2), the second-order equation decomposes into fine degrees $2\mathbf c_1$, $\mathbf c_1+\mathbf c_2$, and $2\mathbf c_2$.  The doubled blocks vanish identically as in (1).  The mixed block is solvable exactly when the class $\alpha\cup\beta$ vanishes, by Lemma~\ref{lem:graded-polarization}; this proves the first assertion.  When $\alpha\cup\beta=0$, choose $h$ of fine degree $\mathbf c_1+\mathbf c_2$ solving the mixed block and lift each relation to second order.  Every higher coefficient of the lifted relation is a sum of terms whose fine degrees have the form $a\mathbf c_1+b\mathbf c_2$ with $a+b\ge3$.  Since $\max(a,b)\ge2$ and the positive supports avoid $B_1\cup B_2$, every such degree has an entry at most $-2$ on $B_1$ or $B_2$, so all these terms vanish identically and the quadratic relation lifts are exact syzygies of $(F_i)$.
\end{proof}

Theorem~\ref{thm:arcs} leaves open directions with support cancellation.
For example, take $\mathbf c_1=\ee_e-\ee_p$ and
$\mathbf c_2=\ee_p-\ee_q$, with $e,p,q$ distinct. The mixed second-order
obstruction has singleton negative support and vanishes by
\eqref{eq:singleton-formula}. After a homogeneous second-order lift is
fixed, its only possibly nonzero subsequent obstruction has degree
$2\mathbf c_1+\mathbf c_2=2\ee_e-\ee_p-\ee_q$.
This class depends on the chosen second-order lift; it is not a canonical
class attached to the triple $(e,p,q)$ alone. Indeed, changing the mixed
correction by $\delta\in T^1_{\mathbf c_1+\mathbf c_2}(M)$ can change
the cubic class by $\alpha\cup\delta$, where $\alpha$ is the class of
degree $\mathbf c_1$.

There are no later degrees to consider in this homogeneous lifting problem:
$a\mathbf c_1+b\mathbf c_2$ has all coordinates at least $-1$ only if
$b\le1$ and $a-b\le1$. Thus every such degree with $a+b\ge4$ has a
coordinate at most $-2$. A choice of second- and third-order corrections
that solves the cubic equation therefore extends without further
corrections to every order. This is a termination statement, not a proof
that a suitable cubic lift always exists.

\section{Certificates, algorithms, and examples}\label{sec:algorithms}

\subsection{A rank-oracle certificate}

Assume access to a rank oracle for $M$. A certificate for $T^2$-nonvanishing is a tuple $(A,D,p,q)$ with distinct $p,q\in D$, satisfying
\begin{enumerate}[label=\textup{(\arabic*)}]
\item $A$ is independent and $|A|=\rk M-1$;
\item $E\setminus D=\cl_MA$;
\item $|D|\ge4$.
\end{enumerate}
The closure condition is checked by verifying
\[
\rk(A\cup\{e\})=\rk A\quad(e\in E\setminus D),
\qquad
\rk(A\cup\{e\})=\rk A+1\quad(e\in D).
\]
Once these tests pass, Proposition~\ref{prop:exact-cocircuit} certifies
\[
\dim_kT^2_{\one_A-\ee_p-\ee_q}(M)=|D|-3.
\]
This tuple records a nonzero obstruction space, not a particular tangent
cocycle. An actual nonliftable deformation is then obtained from
Proposition~\ref{prop:pure-cocircuit}, the coordinate preimages in
Theorem~\ref{thm:pure-pair-surjectivity}, and
Lemma~\ref{lem:graded-polarization}.

For a connected matroid of corank at least three, a certificate is constructed as follows.
\begin{enumerate}[label=\textup{(\arabic*)}]
\item Find a basis $B$ greedily.  If $|B|=1$, output $A=\varnothing$, $D=E$, and any pair $p,q\in E$.
\item Construct the fundamental graph.  Use Lemma~\ref{lem:bipartite} to find $x,y\in B$.
\item Contract $B\setminus\{x,y\}$, delete loops, and compute the parallel classes of the resulting rank-two matroid.
\item Choose a smallest class $P$ and $v\in P$.  Put $A=(B\setminus\{x,y\})\cup\{v\}$ and $D=E\setminus\cl_MA$.
\item Output $(A,D,p,q)$ for any distinct $p,q\in D$.
\end{enumerate}
The construction uses $O(n^2)$ rank-oracle queries.  An independence oracle simulates a rank query greedily, so the procedure is also polynomial in that model.

\subsection{Uniform, graphic, and cographic matroids}

For $M=U_{r,n}$ of corank $c=n-r$, every cocircuit has size $c+1$.  Hence
\[
\dim_kT^2_{\one_A-\ee_p-\ee_q}(U_{r,n})=(c-2)_+
\]
for every hyperplane $A$ and pair in its complement, and
\[
\Theta(U_{r,n})=
\binom n{r-1}\binom{c+1}{2}(c-2)_+.
\]

\begin{corollary}[Graphic matroids]\label{cor:graphic}
For a finite graph $G$, $T^2(M(G))=0$ if and only if every cycle-matroid component has cyclomatic number at most two, equivalently every bond has size at most three. For a simple graph, this vanishing is equivalent to every edge-containing block being a bridge, a cycle, or a theta graph.

If $G$ is connected and deleting a bond $D$ gives components $G^D_1,G^D_2$, then
\[
\Theta(M(G))=
\sum_{D\text{ bond}}
\tau(G^D_1)\tau(G^D_2)
\binom{|D|}{2}(|D|-3)_+,
\]
where $\tau$ denotes the number of spanning trees and $\tau$ of a single vertex is one.
\end{corollary}

\begin{proof}
Cocircuits of $M(G)$ are bonds, and the corank of a cycle-matroid component is its cyclomatic number. The simple-block description follows from open-ear decomposition \cite[Proposition~3.1.1, pp.~59--60]{Diestel}: a cycle starts with cyclomatic number one, and each added ear increases it by one. A basis of $M(G)|(E\setminus D)$ is a spanning tree in each of the two components, so \eqref{eq:Theta-cocircuit} gives the formula.
\end{proof}

\begin{corollary}[Cographic matroids]\label{cor:cographic}
For a finite graph $G$,
\[
T^2(M^*(G))=0
\quad\Longleftrightarrow\quad
\text{every cycle of $G$ has at most three edges}.
\]
For a simple graph, this vanishing is equivalent to every edge-containing block being a bridge or a triangle. If $G$ is connected, then
\[
\Theta(M^*(G))=
\sum_U f(\ell(U)),
\]
where $U$ ranges over connected spanning unicyclic subgraphs and $\ell(U)$ is the length of its unique cycle.
\end{corollary}

\begin{proof}
Cocircuits of $M^*(G)$ are graph cycles.  Apply Theorem~\ref{thm:vanishing}.  The simple-block statement again follows from open-ear decomposition.  The last formula is the nullity-one interpretation in Theorem~\ref{thm:Theta-profiles}.
\end{proof}

The cocircuit threshold four is sharp in every corank: for $c\ge3$, the
connected matroid $M(K_{2,c-1})^*$ has corank $c$ and all cocircuits have
size exactly four. Indeed, $K_{2,c-1}$ has $c+1$ vertices and is connected,
so its cycle matroid has rank $c$; every graph cycle has length four.
Thus Theorem~\ref{thm:large-cocircuit} cannot be strengthened to a bound
growing with corank. The dual matroids $M(K_{2,c-1})$ are critically
connected: deleting any edge of $K_{2,c-1}$ creates a bridge. They have
rank $c$ and $2c-2$ elements, which is the largest possible size of a
critically connected matroid of rank $c$ by a theorem of
Murty~\cite{Murty}.

\section{A sharp rank-two extremal theorem}\label{sec:rank-two-extremal}

This section records a quantitative consequence of the hyperplane-graph formula proved earlier in the paper.

\subsection{The witness invariant}

Recall that, for each independent set $A$ of size $r-1$, we write
\[
d_A=|E\setminus\cl_M(A)|
\]
and define the invariant
\begin{equation}\label{eq:rank-two-extremal-Theta}
\Theta(M)=\sum_{\substack{A\text{ independent}\\ |A|=r-1}}
\binom{d_A}{2}(d_A-3)_+ .
\end{equation}
Equivalently, grouping the same terms by cocircuit $D$ gives
\begin{equation}\label{eq:rank-two-extremal-cocircuit}
\Theta(M)=\sum_D b\bigl(M|(E\setminus D)\bigr)
\binom{|D|}{2}(|D|-3)_+ .
\end{equation}
These are \eqref{eq:Theta-definition} and \eqref{eq:Theta-cocircuit}.  The rank-one-link calculation in Proposition~\ref{prop:exact-cocircuit} identifies $\Theta(M)$ with the total dimension of the finite squarefree witness space.  In rank two, the same sum is exactly $\dim_k T^2(M)_0$.

\subsection{A closed formula in rank two}

Let $M$ be a connected loopless rank-two matroid.  Its nonloop elements split into parallel classes $P_1,\ldots,P_t$ of sizes $n_1,\ldots,n_t$; connectedness is equivalent to $t\ge 3$.  Put $n=|E|$ and $c=n-2$, and set $F(d)=\binom d2(d-3)_+$.

\begin{proposition}[Rank-two closed form]\label{prop:rank-two-closed-form}
For every connected rank-two matroid,
\begin{equation}\label{eq:rank-two-closed-form}
\Theta(M)=\sum_i n_iF(n-n_i),
\qquad \dim_kT^2(M)_0=\Theta(M).
\end{equation}
\end{proposition}

\begin{proof}
Every independent set of size $r-1$ is a single nonloop element $a$.  If $a\in P_i$, then $\cl_M(a)=P_i$ and the complementary cocircuit has size $n-n_i$.  The complementary hyperplane has $n_i$ bases, one for each element of $P_i$, so \eqref{eq:rank-two-extremal-cocircuit} gives the first formula.  For the standard total-degree-zero component, a positive support of size two has a rank-zero link and contributes nothing.  A positive support of size one must have exponent two; its rank-one link contributes only squarefree negative pairs.  Summing \eqref{eq:rank-two-extremal-Theta} over all singletons gives the second formula.
\end{proof}

\subsection{The sharp minimum at fixed corank}

\begin{theorem}[Sharp rank-two extremum]\label{thm:rank-two-extremum}
Among connected rank-two matroids of corank $c\ge3$, the minimum of $\Theta$ is
\begin{equation}\label{eq:rank-two-minimum}
\begin{cases}
6,&c=3,\\
32,&c=4,\\
81,&c=5,\\
168,&c=6,\\
c(c+1)(c-2),&c\ge7.
\end{cases}
\end{equation}
Up to permutation of the classes, equality occurs for
\begin{equation}\label{eq:rank-two-equality-cases}
\begin{array}{c|c}
c&\text{parallel-class sizes}\\ \hline
3&(2,2,1)\\
4&(3,2,1)\\
5&(3,3,1)\\
6&(6,1,1)\text{ or }(4,3,1)\\
c\ge7&(c,1,1).
\end{array}
\end{equation}
For $c\ge7$ the equality class is unique.
\end{theorem}

\begin{proof}
Put $g_n(x)=xF(n-x)$.  By \eqref{eq:rank-two-closed-form}, we must minimize $\sum_i g_n(n_i)$ over partitions of $n=c+2$ into at least three positive parts.

First, let $u,v\ge1$ and $u+v\le n-2$, and write $w=n-u-v$.  If $w\ge4$, direct expansion gives
\begin{equation}\label{eq:rank-two-merge}
\begin{aligned}
2\bigl(g_n(u)+g_n(v)-g_n(u+v)\bigr)
={}&uv\bigl[u^2+v^2+3(u+v)w+6w^2\\
&\qquad-4(u+v)-16w+6\bigr]>0.
\end{aligned}
\end{equation}
Indeed, $3w-4>0$ and $6w^2-16w+6>0$ for $w\ge4$.
If $w=2$ or $3$, then $g_n(u+v)=0$, while at least one of $g_n(u),g_n(v)$ is positive for $n\ge5$.  Thus every minimizer has exactly three parts.

Let $a\le b\le d$ and $a+b+d=n$.  If $a\ge2$, transfer one element from the smallest class to the largest and set
\[
\Delta=g_n(a)+g_n(d)-g_n(a-1)-g_n(d+1).
\]
With $A=a-2$, $U=b-a$, and $V=d-b$, expansion gives
\begin{equation}\label{eq:rank-two-transfer}
\begin{aligned}
2\Delta={}&8+24A+32U+12V+12A^2+45AU+27AV\\
&+34U^2+33UV+5V^2+12A^2U+12A^2V\\
&+21AU^2+24AUV+3AV^2+10U^3+15U^2V+6UV^2+V^3.
\end{aligned}
\end{equation}
All coefficients are positive, so a minimizing triple has the form $(1,b,d)$.

Finally let $s=b+d=n-1$ and $p=bd$.  For $b\ge2$,
\begin{equation}\label{eq:rank-two-final-comparison}
2\bigl(g_n(b)+g_n(d)-g_n(1)\bigr)=E(s,p),
\quad E(s,p)=-2p^2+(s^2-s-4)p-s(s-1)(s-3).
\end{equation}
For $n\ge9$, $E$ is concave on $2(s-2)\le p\le\lfloor s^2/4\rfloor$.  At the left endpoint,
\[
s^3-10s^2+25s-16>0,
\]
while at the right endpoint, writing $s=2m$ or $s=2m+1$ gives, respectively,
\[
2m(m^3-5m^2+6m-3)>0
\qquad\text{or}\qquad
2m^3(m-3)>0.
\]
Hence $b=1$ is the unique minimizer for $n\ge9$, with value
\begin{equation}\label{eq:rank-two-sparse-value}
2F(n-1)=(n-1)(n-2)(n-4)=c(c+1)(c-2).
\end{equation}
The cases $n=5,6,7,8$ are the four exceptional rows in \eqref{eq:rank-two-equality-cases}, obtained by direct enumeration of the three-part partitions.
\end{proof}

\subsection{Consequences}

For $c\ge7$, the rank-two matroid $R_c$ with class sizes $(c,1,1)$ satisfies
\[
\Theta(R_c)=c(c+1)(c-2)=\dim_kT^2(R_c)_0
\]
and realizes the sharp minimum.  The uniform rank-two matroid $U_{2,c+2}$ has
\[
\Theta(U_{2,c+2})=\frac{(c+2)(c+1)c(c-2)}2,
\]
so the ratio of the sparse extremal value to the uniform value is $2/(c+2)$.

The theorem gives $\Theta(M)\ge6$ for every connected rank-two matroid of corank at least three.  Over every field, the pure-pair cup-product theorem and the graded polarization lemma of Section~\ref{sec:cup} turn these nonzero standard-degree-zero classes into actual nonliftable first-order deformations.

The proof is symbolic and independent of computation.  Additional exact checks are described in Section~\ref{sec:verification}; they are not premises of the proof.

\section{Integral pair obstructions and a global rank}\label{sec:pair-rank}

The coordinate-delta argument of Section~\ref{sec:cup} contains more
arithmetic information than fieldwise surjectivity.  It gives a split map
over the integers and therefore a canonical numerical size for the whole
pure-pair obstruction layer.

For a matroid $N$ and an element $p$, put
\[
\mathsf L^1_p(N)=
\widetilde H^0\!\left(\mathscr H_p(N);\ZZ\right),
\]
where $\mathscr H_p(N)$ is the discrete set of hyperplanes avoiding $p$.
For distinct $p,q$, put
\[
\mathsf L^2_{p,q}(N)=
\bigoplus_{\substack{H\text{ hyperplane of }N\\H\cap\{p,q\}=\varnothing}}
\widetilde H^0\!\left(\Gamma_{\{p,q\}}(H);\ZZ\right).
\]
Both groups are free abelian.  Interpreting
\eqref{eq:block-cup-coordinate} with integer-valued functions defines an
integral bilinear map
\begin{equation}\label{eq:integral-pair-map}
\smile_{\ZZ}:\mathsf L^1_p(N)\otimes_{\ZZ}\mathsf L^1_q(N)
\longrightarrow\mathsf L^2_{p,q}(N).
\end{equation}

\begin{theorem}[Integral pure-pair obstruction lattice]
\label{thm:integral-pair}
For every matroid $M$ and every pair of distinct actual vertices $p,q$, the
map \eqref{eq:integral-pair-map} for $N=M^*$ is a split surjection of
abelian groups.  Consequently its scalar extension is surjective over every
commutative ring.  Over a field $k$, the identifications
\eqref{eq:T1-hyperplane-functions} and \eqref{eq:pair-block} identify the
scalar extension with the cotangent cup product
\[
T^1_{-\ee_p}(M)\otimes_kT^1_{-\ee_q}(M)
\longrightarrow T^2_{-\ee_p-\ee_q}(M).
\]
In particular,
\[
\operatorname{rank}_{\ZZ}\mathsf L^2_{p,q}(M^*)
=\dim_kT^2_{-\ee_p-\ee_q}(M),
\]
and every nonzero pair block remains a split cup-product quotient after
reduction modulo every prime.
\end{theorem}

\begin{proof}
Fix a target coordinate $(H_0,G_0)$ and put
\[
P_0=\cl_{M^*}(G_0\cup\{p\}),\qquad
Q_0=\cl_{M^*}(G_0\cup\{q\}),
\]
\[
\mathcal S(G_0)=
\{H:H\text{ is a hyperplane of }M^*,\
H\cap\{p,q\}=\varnothing,\ G_0\subset H\}.
\]
Define integer-valued representatives by
\[
\phi(K)=
\begin{cases}
1,&K=Q_0,\\
1,&K\in\mathcal S(G_0)\setminus\{H_0\},\\
0,&\text{otherwise},
\end{cases}
\qquad
\psi(K)=
\begin{cases}
1,&K=P_0,\\
0,&\text{otherwise}.
\end{cases}
\]
There is no conflict because $Q_0$ contains $q$, $P_0$ contains $p$, and
every member of $\mathcal S(G_0)$ avoids both.  The proof of
Theorem~\ref{thm:pure-pair-surjectivity}, now read over $\ZZ$, shows that
\eqref{eq:block-cup-coordinate} gives the coordinate delta class at
$(H_0,G_0)$, up to the fixed block orientation.  In each quotient
$\ZZ^{\mathcal G_{\{p,q\}}(H)}/\ZZ\one$, all but one of the delta classes
form a basis.  Choosing their displayed decomposable preimages defines a
right inverse to \eqref{eq:integral-pair-map}.  Scalar extension preserves a
split epimorphism, and the fieldwise identification follows from the same
integer coordinate formula.
\end{proof}

\begin{corollary}[Unimodularity and absence of bad primes]
\label{cor:unimodular-pair-cup}
After integral bases are chosen, every pure-pair cup matrix has Smith normal
form
\[
\operatorname{diag}(\underbrace{1,\ldots,1}_{
\operatorname{rank}\mathsf L^2_{p,q}},0,\ldots,0).
\]
Thus the pair obstruction quotient has an integral decomposable basis; its
dimension, nonvanishing pattern, and cup surjectivity are simultaneous over
all fields, including characteristic two.
\end{corollary}

\begin{proof}
A split epimorphism between free abelian groups is integrally equivalent to a
projection onto a direct summand.  This is precisely the asserted Smith
normal form.
\end{proof}

\begin{definition}[Pure-pair obstruction ranks]
\label{def:pure-pair-ranks}
For $p\ne q$ in $E(M)$, set
\[
V_p(M)=T^1_{-\ee_p}(M),\qquad
W_{p,q}(M)=T^2_{-\ee_p-\ee_q}(M),
\]
and define
\begin{equation}\label{eq:global-pair-rank-definition}
W_{\mathrm{pair}}(M)=
\bigoplus_{\{p,q\}\subseteq E(M)}W_{p,q}(M),\qquad
\mathfrak q(M)=\dim_kW_{\mathrm{pair}}(M).
\end{equation}
Loops contribute zero to these spaces; in statements about actual vertices
they are simply omitted. The value of $\mathfrak q(M)$ is independent of $k$. Define also the
standard-degree-zero one-link rank
\begin{equation}\label{eq:degree-zero-pair-rank-definition}
\mathfrak q_0(M)=
\sum_{e\in V(M)}
\sum_{\{p,q\}\subseteq V(\lk_M e)}
\dim_kT^2_{2\ee_e-\ee_p-\ee_q}(M),
\end{equation}
where loops in the link are omitted.
\end{definition}

\begin{theorem}[Global pure-pair obstruction rank]
\label{thm:global-pair-rank}
The direct sum of the pure-pair cup products is a split epimorphism
\[
\mu_M:
\bigoplus_{\{p,q\}\subseteq E(M)}V_p(M)\otimes_kV_q(M)
\twoheadrightarrow W_{\mathrm{pair}}(M)
\]
over every field, and it has an integral split form.  Moreover,
\begin{align}
\mathfrak q(M)
&=\sum_{\{p,q\}\subseteq E(M)}
\sum_{\substack{H\text{ hyperplane of }M^*\\
H\cap\{p,q\}=\varnothing}}
\bigl(|\mathcal G_{\{p,q\}}(H)|-1\bigr)_+,
\label{eq:global-pair-rank-formula}\\
\mathfrak q(M)=0
&\quad\Longleftrightarrow\quad
M\text{ has no }U_{1,4}\text{-minor}.
\label{eq:global-pair-rank-vanishing}
\end{align}
If $D$ is a cocircuit of cardinality $d\ge4$, then
\begin{equation}\label{eq:global-pair-cocircuit-bound}
\mathfrak q(M)\ge\binom d2(d-3).
\end{equation}
If $M$ is connected with rank at least two and corank at least three, then
the direct sum of the link-lifted cup products is split surjective onto the
space counted by $\mathfrak q_0(M)$, and
\begin{equation}\label{eq:degree-zero-pair-cocircuit-bound}
\mathfrak q_0(M)\ge\binom d2(d-3)
\end{equation}
for some cocircuit $D$ of cardinality $d\ge4$ in a one-element contraction
of $M$.
\end{theorem}

\begin{proof}
Theorem~\ref{thm:integral-pair} gives a split epimorphism for each unordered
pair; taking direct sums proves the first assertion.  The pair formula
\eqref{eq:pair-dim} gives \eqref{eq:global-pair-rank-formula}, and the
vanishing equivalence follows from Theorem~\ref{thm:vanishing}, since the
pair layer detects the full $T^2$-vanishing locus.

For each pair in a cocircuit $D$, Proposition~\ref{prop:pure-cocircuit}
gives a target of dimension at least $d-3$.  The corresponding fine degrees
are distinct, so summing over the $\binom d2$ pairs proves
\eqref{eq:global-pair-cocircuit-bound}.  For the standard-degree-zero
assertion, choose $e$ and a connected component of $M/e$ as in
Lemma~\ref{lem:one-contraction}.  Apply the same cocircuit argument in that
link and transport the integral right inverses by
Corollary~\ref{cor:link-lifted-cup}.  This proves split surjectivity and
\eqref{eq:degree-zero-pair-cocircuit-bound}.
\end{proof}

\begin{corollary}[Direct-sum additivity]
\label{cor:pair-rank-direct-sum}
For matroids on disjoint ground sets,
\begin{equation}\label{eq:pair-rank-direct-sum}
\mathfrak q(M_1\oplus M_2)=\mathfrak q(M_1)+\mathfrak q(M_2).
\end{equation}
More precisely, a pair with one element in each summand has zero target,
whereas a pair contained in one summand has the same target as in that
summand.
\end{corollary}

\begin{proof}
Take the degree $-\ee_p-\ee_q$ part of \eqref{eq:direct-sum-T2}.  If $p$ and
$q$ lie in different summands, each potential tensor term would require a
negative degree in a Stanley--Reisner ring factor, which is impossible.  If
both lie in one summand, the degree-zero part of the other ring is $k$ and
the target is unchanged.  Summing over unordered pairs proves
\eqref{eq:pair-rank-direct-sum}.
\end{proof}

\begin{corollary}[Rank-two pure-pair rank]
\label{cor:rank-two-pure-pair-rank}
Let $M$ be a loopless rank-two matroid with parallel classes
$P_1,\ldots,P_t$ of sizes $n_1,\ldots,n_t$, and put $n=\sum_i n_i$.  Then
\begin{align}
\mathfrak q(M)={}&
\sum_i\binom{n_i}{2}
\left(n_i-3+
\sum_{\substack{a<c\\a,c\ne i}}n_an_c\right)_+
\label{eq:rank-two-pair-rank}\\
&+\sum_{i<j}n_in_j
\sum_{k\ne i,j}n_k(n-3-n_k)_+.
\nonumber
\end{align}
For the rank-one uniform matroid,
\[
\mathfrak q(U_{1,n})=\binom n2(n-3)_+.
\]
\end{corollary}

\begin{proof}
Sum the within-class case of Theorem~\ref{thm:rank-two-formulas} over the
$\binom{n_i}{2}$ pairs in $P_i$, and sum the between-class case over the
$n_in_j$ pairs in $P_i\times P_j$.  The rank-one formula follows directly
from the single cocircuit of $U_{1,n}$.
\end{proof}

\begin{theorem}[Sharp rank-two pure-pair extremum]
\label{thm:rank-two-pair-extremum}
Among connected loopless rank-two matroids of corank $c\ge3$, the sharp
minimum of the global pure-pair obstruction rank is
\begin{equation}\label{eq:rank-two-pair-minimum}
\min\mathfrak q(M)=
\begin{cases}
6,&c=3,\\
26,&c=4,\\
\dfrac{c(c-2)(c+3)}2,&c\ge5.
\end{cases}
\end{equation}
Writing the parallel-class sizes as an unordered partition of $c+2$, the
equality profiles are
\begin{equation}\label{eq:rank-two-pair-equality}
(2,2,1),\qquad(3,2,1),\qquad(c,1,1)\quad(c\ge5),
\end{equation}
respectively.  Thus the equality profile is unique up to permutation for
every $c\ge3$.
\end{theorem}

\begin{proof}
For a class profile $\nu=(n_1,\ldots,n_t)$, write $Q(\nu)$ for the
right-hand side of \eqref{eq:rank-two-pair-rank}.  Put $n=c+2$ and
$C=n-3$.  For each index $i$, let
\[
e_i=\sum_{\substack{j<\ell\\j,\ell\ne i}}n_jn_\ell.
\]
Regrouping the second sum by the omitted class gives
\begin{equation}\label{eq:pair-rank-regrouped}
Q(\nu)=\sum_i\binom{n_i}{2}(n_i-3+e_i)
+\sum_i n_i(C-n_i)_+e_i.
\end{equation}
The first positive part from \eqref{eq:rank-two-pair-rank} is unnecessary:
for a profile with at least three classes and $n\ge5$, one has
$n_i-3+e_i\ge0$.

Suppose $t\ge4$ and merge classes of sizes $x,y$ into a class of size
$z=x+y$.  Let $R$ be the remaining list, and put
\[
S=\sum_{r\in R}r,\qquad E=\sum_{r<s\in R}rs.
\]
Every original class has size at most $C$, while $z\le C+1$.  If $z\le C$,
direct subtraction in \eqref{eq:pair-rank-regrouped} gives
\begin{align}
2\bigl(Q(\nu)-Q(\nu')\bigr)
=xy\Biggl(&2\sum_{r\in R}\left[\binom r2+r(C-r)\right]
+8-3z\nonumber\\
&+S(4C-z-2)+2E\Biggr)>0.
\label{eq:pair-rank-merge}
\end{align}
Indeed, $S\ge2$, and the expression not already visibly nonnegative is at
least $4+8C-5z\ge4+3C$.  If $z=C+1$, necessarily $R=(1,1)$, and subtraction
reduces to
\begin{equation}\label{eq:pair-rank-merge-boundary}
2\bigl(Q(\nu)-Q(\nu')\bigr)=xy(7C-3)-2C-2>0,
\end{equation}
because $xy\ge C$ and $C\ge2$.  Repeated merging strictly decreases $Q$,
so a minimizer has exactly three classes.

Let $a\le b\le d$ and $a+b+d=n$.  Except for $(a,b,d)=(1,1,c)$, all three
terms $C-a,C-b,C-d$ are nonnegative, and
\begin{align}
2Q(a,b,d)={}&a(a-1)(a-3+bd)+b(b-1)(b-3+ad)\nonumber\\
&+d(d-1)(d-3+ab)+2abd(2n-9).
\label{eq:pair-rank-three-classes}
\end{align}
If $a\ge2$, set $A=a-1$, $U=b-a$, and $V=d-b$.  Subtracting the value after
the transfer $(a,b,d)\mapsto(a-1,b,d+1)$ gives
\begin{align}
2\bigl(Q(a,b,d)-Q(a-1,b,d+1)\bigr)
={}&15A^2+3A-4+(15A^2+28A-3)U\nonumber\\
&+(15A^2+8A-2)V+11U^2+2V^2+8UV\nonumber\\
&+25AU^2+30AUV+5AV^2\nonumber\\
&+10U^3+15U^2V+5UV^2>0.
\label{eq:pair-rank-transfer}
\end{align}
All coefficients in the first line are positive for $A\ge1$, and the other
terms are nonnegative.  Iteration reduces every minimizer to $(1,b,d)$.

Put $s=b+d=c+1$ and $p=bd$.  For $b\ge2$,
\begin{equation}\label{eq:pair-rank-one-b-d}
2Q(1,b,d)=s(s-1)(s-3)+2(s-4)p.
\end{equation}
The product $p=b(s-b)$ is minimized on $2\le b\le s/2$ at $b=2$.  The
exceptional profile has value
\begin{equation}\label{eq:pair-rank-concentrated}
Q(1,1,c)=\frac{c(c-2)(c+3)}2.
\end{equation}
For $b=2$, the difference between \eqref{eq:pair-rank-one-b-d} and
\eqref{eq:pair-rank-concentrated} is $c^2-6c+6$.  This is negative for
$c=3,4$ and positive for every $c\ge5$.  The cases $c=3,4$ give the first
two profiles in \eqref{eq:rank-two-pair-equality}, while the concentrated
profile is uniquely minimizing for $c\ge5$.  Strictness in the merge,
transfer, and product comparisons proves uniqueness.
\end{proof}

\begin{corollary}[Exact stability gap]
\label{cor:rank-two-pair-stability}
Let $c\ge5$.  If a connected loopless rank-two matroid of corank $c$ does
not have class profile $(c,1,1)$, then
\begin{equation}\label{eq:rank-two-pair-stability}
\mathfrak q(M)\ge
\frac{c(c-2)(c+3)}2+c^2-6c+6.
\end{equation}
Equality holds exactly for the profile $(c-1,2,1)$, up to permutation.
\end{corollary}

\begin{proof}
Among nonextremal three-class profiles, the transfer calculation reduces to
$(1,b,c+1-b)$ with $b\ge2$.  Equation~\eqref{eq:pair-rank-one-b-d} and the
monotonicity of $b(c+1-b)$ show that the unique minimum occurs at $b=2$,
and its difference from \eqref{eq:pair-rank-concentrated} is
$c^2-6c+6$.  If a profile has at least four classes, merging strictly lowers
$Q$.  When the final three-class profile is nonextremal, the preceding bound
is strict.  If the final profile is $(1,1,c)$, the penultimate profile is
$(1,1,x,y)$ with $x+y=c$, and
\eqref{eq:pair-rank-merge-boundary} gives
\[
2\bigl(Q(1,1,x,y)-Q(1,1,c)\bigr)
=xy(7c-10)-2c,
\]
which is strictly larger than $2(c^2-6c+6)$ for $c\ge5$.  Equality can
therefore occur only for $(1,2,c-1)$.
\end{proof}

\begin{remark}[Two different rank-two minima]
The invariants $\mathfrak q(M)$ and $\Theta(M)=\dim_kT^2(M)_0$ are different
in rank two.  The sharp values in \eqref{eq:rank-two-pair-minimum} measure
the complete purely negative pair layer, whereas
Theorem~\ref{thm:rank-two-extremum} measures the standard-degree-zero
one-link layer.  Thus the pair layer begins with $6,26,\ldots$, while the
degree-zero layer begins with $6,32,81,168,\ldots$.
\end{remark}

\section{The standard-degree-zero coefficient formula}
\label{sec:degree-zero-master}

The complete fine formula has a useful coefficient extraction which is
independent of the rank.  For an independent set $A\in M$, put
$L_A=\lk_M A$.  If $B$ is a squarefree subset of the actual vertices of
$L_A$ with $|B|\ge2$, define
\[
\omega_M(A,B)=
\sum_{\substack{H\text{ hyperplane of }L_A^*\\H\cap B=\varnothing}}
\bigl(\kappa(\Gamma^A_B(H))-1\bigr)_+,
\]
where $\kappa$ denotes the number of connected components and an empty graph
contributes zero.  Thus, by the complete hyperplane-graph formula,
\[
\omega_M(A,B)=\dim_kT^2_{-\one_B}(L_A).
\]

\begin{theorem}[Standard-degree-zero master formula]
\label{thm:degree-zero-master}
For every finite matroid $M$ and every field $k$,
\begin{equation}\label{eq:degree-zero-master}
\dim_kT^2(M)_0=
\sum_{\substack{\varnothing\ne A\in M}}
\sum_{\substack{B\subseteq V(L_A)\\|B|\ge2}}
\binom{|B|-1}{|A|-1}\,\omega_M(A,B).
\end{equation}
In particular, the standard-degree-zero obstruction dimension is
characteristic-independent and is computable from finitely many flat graphs
in the links of independent sets.
\end{theorem}

\begin{proof}
Fix a positive support $A$ of size $s$ and a negative support $B$ of size
$b$.  The link localization formula identifies the corresponding fine
component with $T^2_{-\one_B}(L_A)$, and hence its dimension is
$\omega_M(A,B)$.  The standard total degree is zero exactly when the sum of
the positive exponents is $b$.  The number of positive exponent vectors with
support exactly $A$ and total sum $b$ is
\[
\#\{(u_a)_{a\in A}:u_a\ge1,\ \sum_{a\in A}u_a=b\}
=\binom{b-1}{s-1},
\]
with the usual convention that this is zero for $b<s$.  The empty positive
support cannot contribute because $B$ is nonempty.  Summing over $A$ and
$B$ proves \eqref{eq:degree-zero-master}; the graph formula gives the
characteristic independence.
\end{proof}

\begin{corollary}[Finite coefficient algorithm]
\label{cor:degree-zero-algorithm}
For a rank-$r$ matroid, equation \eqref{eq:degree-zero-master} only requires
independent supports $A$ with $1\le|A|\le r-1$.  A fixed graph block with
negative support $B$ is counted with the binomial multiplicity
$\binom{|B|-1}{|A|-1}$.  Thus the standard-degree-zero slice is a finite
rank-stratified incidence invariant, rather than a separate cohomology
calculation.
\end{corollary}

\begin{proof}
If $|A|\ge r$, its link has rank zero and has no actual vertex supporting a
nonzero negative component.  The remaining assertion is the coefficient
count in the proof of Theorem~\ref{thm:degree-zero-master}.
\end{proof}

\section{A complete rank-three formula}
\label{sec:rank-three-formula}

We now make \eqref{eq:degree-zero-master} explicit for simple rank-three
matroids.  Write
\[
f(d)=\binom d2(d-3)_+.
\]
For a rank-two flat $L$ call $L$ a line, and set
\[
\Theta_3(M)=
\sum_{L\text{ line}}\binom{|L|}{2}f(n-|L|),
\qquad n=|E(M)|.
\]
Here $\Theta_3(M)=\Theta(M)$: a line of a simple rank-three matroid has
$\binom{|L|}{2}$ bases, so this is \eqref{eq:Theta-cocircuit} specialized
to rank three. For a point $e$, let $L_1(e),\ldots,L_{t_e}(e)$ be the lines through $e$ and
put
\[
m_{e,i}=|L_i(e)|-1.
\]
The $m_{e,i}$ are positive and sum to $n-1$.  For a three-element index set
$I=\{i,j,k\}$, define
\[
\rho_e(I)=\#\{\ell:m_{e,\ell}-\one_{\ell\in I}>0\}.
\]
Finally put
\begin{align}
\Phi_e(M)={}&
\sum_i\binom{m_{e,i}}2
\left(m_{e,i}-3+
\sum_{\substack{j<k\\j,k\ne i}}m_{e,j}m_{e,k}\right)_+
\label{eq:rank-three-point-formula}\\
&+\sum_{i<j}m_{e,i}m_{e,j}
\sum_{k\ne i,j}m_{e,k}(n-4-m_{e,k})_+\nonumber\\
&+(n-5)_+
\sum_{\substack{i<j<k\\\rho_e(\{i,j,k\})\le1}}
m_{e,i}m_{e,j}m_{e,k}.\nonumber
\end{align}

\begin{theorem}[Complete rank-three degree-zero formula]
\label{thm:rank-three-formula}
For every simple rank-three matroid $M$ on $n$ elements and every field $k$,
\begin{equation}\label{eq:rank-three-formula}
\dim_kT^2(M)_0=\Theta_3(M)+\sum_{e\in E(M)}\Phi_e(M).
\end{equation}
Consequently the entire standard-degree-zero dimension is determined by the
global line sizes and the line-pencil partitions at the points.
\end{theorem}

\begin{proof}
The support criterion allows positive support of size at most two.  If
$|A|=2$, simplicity makes $A$ independent and its link has rank one.  Degree
zero forces a negative pair, and the rank-one-link formula contributes
$f(n-|\cl_M(A)|)$.  Grouping pairs $A$ with the same line closure gives
$\Theta_3(M)$.

If $A=\{e\}$, the link $M/e$ is rank two and its parallel classes are
$L_i(e)\setminus\{e\}$ of sizes $m_{e,i}$.  The rank-two formula has three
possible negative supports.  A pair in one class gives the first line of
\eqref{eq:rank-three-point-formula}; a pair in two different classes gives
the second line; and a triple from three classes survives exactly when,
after removing the selected points, at most one parallel class remains,
which is the condition $\rho_e(I)\le1$.  The coefficient in the latter case
is $(n-5)_+$.  Summing over $e$ proves the formula.  No other positive
support contributes, and all terms are graph dimensions, so the result is
characteristic-independent.
\end{proof}

\begin{corollary}[Uniform rank three]
\label{cor:uniform-rank-three}
For $n\ge6$,
\[
\dim_kT^2(U_{3,n})_0=
\binom n2\binom{n-2}{2}(n-5)
+n\binom{n-1}{2}(n-3)(n-5).
\]
\end{corollary}

\begin{proof}
Every line has two points, and every point pencil has $n-1$ singleton
classes.  Substitute these data into \eqref{eq:rank-three-formula}; the
triple term is zero.
\end{proof}

\begin{corollary}[The Fano plane]
\label{cor:fano-rank-three}
For the Fano matroid $F_7$,
\[
\dim_kT^2(F_7)_0=357
\]
over every field.
\end{corollary}

\begin{proof}
The seven lines each have three points. Contracting any point gives three
parallel classes of size two, not singleton classes. Thus
\[
\Theta_3(F_7)=7\binom32\binom42(4-3)=126.
\]
For the partition $(2,2,2)$, the first two lines of
\eqref{eq:rank-three-point-formula} contribute $9$ and $24$, respectively;
the last line vanishes because $\rho_e(\{1,2,3\})=3$. Hence
$\Phi_e(F_7)=33$ and the total is $126+7\cdot33=357$.
\end{proof}

For later use, the rank-three near-pencil on $n\ge5$ points, with one line of size $n-1$, has
\begin{equation}\label{eq:rank-three-near-pencil-value}
\dim_kT^2(M)_0=
\frac32(n-1)(n-2)(n-3)(n-5).
\end{equation}

\section{Obstruction rigidity and inverse reconstruction}
\label{sec:obstruction-rigidity}

Brehm and Constantinescu proved that the dimensions of all multigraded components of $T^1(M)$ recover every nondiscrete matroid \cite[Theorem~5.5]{BrehmConstantinescu}. We consider instead a restricted family of components of $T^2$. Under the hypotheses below, even their nonvanishing pattern suffices for reconstruction.

The labeled fine degrees in the rank-one-link part retain more information
than their total dimension.  For a rank-$r$ matroid define the witness profile
\[
\mathscr W_M(A;p,q)=
\dim_kT^2_{\one_A-\ee_p-\ee_q}(M),
\]
where $A$ ranges over $(r-1)$-subsets and $p\ne q$ are outside $A$.

\begin{theorem}[Cocircuit reconstruction]
\label{thm:cocircuit-reconstruction-new}
Let $M$ be loopless of rank $r\ge2$ and suppose every cocircuit has size at
least four.  Then the labeled profile $\mathscr W_M$ determines $M$ up to
relabeling.  More precisely, for every independent $(r-1)$-set $A$, if
\[
D_A=\{p:\mathscr W_M(A;p,q)>0\text{ for some }q\},
\]
then $D_A=E(M)\setminus\cl_M(A)$, and these sets recover all cocircuits.
\end{theorem}

\begin{proof}
The rank-one-link formula gives
\[
\mathscr W_M(A;p,q)=
\begin{cases}
|D_A|-3,&A\in M,\ p,q\in D_A,\\
0,&\text{otherwise}.
\end{cases}
\]
The cocircuit hypothesis makes the first case nonzero for every independent
$A$, so the union of the negative-coordinate supports is exactly $D_A$.
Every hyperplane has a basis of size $r-1$, and cocircuits determine the
matroid.  Hence the profile reconstructs $M$.
\end{proof}

\begin{theorem}[Rank-three obstruction rigidity]
\label{thm:rank-three-rigidity-new}
Let $M$ and $M'$ be simple rank-three matroids on the same labeled ground set
$E$, with $|E|=n\ge7$.  If
\[
\dim_kT^2_{\ee_a+\ee_b-\ee_p-\ee_q}(M)
=\dim_kT^2_{\ee_a+\ee_b-\ee_p-\ee_q}(M')
\]
for all disjoint unordered pairs $\{a,b\}$ and $\{p,q\}$, then
$M=M'$.  In fact, only the nonvanishing pattern is needed.
\end{theorem}

\begin{proof}
Fix $A=\{a,b\}$.  Its closure is a line $L_A$, and the rank-one-link
formula says that the profile is nonzero precisely when
$|E\setminus L_A|\ge4$; in that case the union of the negative supports is
$E\setminus L_A$.  Thus every detected line is recovered.

Define the zero-profile graph on $E$ by joining $a,b$ when all the
components with positive support $\{a,b\}$ in the statement vanish.
Its edges are exactly the pairs contained in undetected lines.
An undetected line has size at least $n-3$. For $n\ge8$ two distinct such
lines cannot coexist, since distinct lines meet in at most one point.  For
$n=7$ there are at most two; if there are two, they are four-point lines
meeting in one point and are the maximal cliques of the zero-profile graph.
In either case its maximal cliques of size at least two are exactly the
undetected lines; isolated vertices do not represent lines.
Hence the undetected lines are also recovered. The rank-two flats determine
a simple rank-three matroid, proving the assertion.
\end{proof}

\begin{corollary}[Explicit inverse algorithm]
\label{cor:rank-three-inverse-algorithm}
For a simple rank-three matroid on $n\ge7$ labeled elements, the lines can be
reconstructed in polynomial time from the nonzero components with positive
support a pair and negative support a pair: take the complement of the union
of the negative supports for every detected pair, and recover any remaining
large lines as maximal cliques in the zero-profile graph.
\end{corollary}

\begin{proof}
This is the constructive procedure in the proof of
Theorem~\ref{thm:rank-three-rigidity-new}. The nonisolated part of the
zero-profile graph is a union of at most two cliques meeting in at most
one vertex. For each edge, its endpoints together with their common
neighbors recover its clique; removing duplicates therefore finds all
the required cliques in polynomial time. Testing a candidate triple
against the recovered line list decides independence.
\end{proof}

\begin{remark}[Sharpness of the seven-point bound]
\label{rem:rigidity-sharpness}
The bound $n\ge7$ in Theorem~\ref{thm:rank-three-rigidity-new} is sharp.  On
six labeled points, a pair is detected exactly when its line has two points,
and the detected block then carries the data
$(d_A-3,E\setminus A)=(1,E\setminus A)$, which depends only on the pair
itself.  The two copies of $M(K_4)$ with line sets
\[
\{123,\,145,\,246,\,356\}
\qquad\text{and}\qquad
\{124,\,135,\,236,\,456\}
\]
both have exactly the two-point lines $16$, $25$, and $34$, so their labeled
profiles coincide, while the matroids differ.
\end{remark}

\section{The sharp rank-three extremal law}
\label{sec:rank-three-extremal}

Put
\begin{equation}\label{eq:rank-three-extremal-value}
P_n=\frac32(n-1)(n-2)(n-3)(n-5).
\end{equation}
By \eqref{eq:rank-three-near-pencil-value}, this is the value of the
near-pencil on $n$ points.

For a list $(m_1,\ldots,m_t)$ of pencil class sizes, let
$\varphi_n(m_1,\ldots,m_t)$ denote the right-hand side of
\eqref{eq:rank-three-point-formula}, with the $m_i$ substituted for the
classes at one point.  We first record two concentration estimates.

\begin{lemma}[Pencil concentration bound]
\label{lem:pencil-concentration-new}
Let $L$ be a line of a simple rank-three matroid on $n=l+d$ elements, with
$l\ge4$ and $d\ge4$, and suppose $L$ is a longest line.  If $e\notin L$,
then
\begin{equation}\label{eq:pencil-concentration}
\Phi_e(M)\ge
\tau(l,d):=\binom{l-1}{2}d(3l+2d-14)
+3\binom{l-1}{3}(l+d-5).
\end{equation}
\end{lemma}

\begin{proof}
The classes of the pencil at $e$ have sizes $m_1,\ldots,m_t$, where
$t\ge l$, $\sum_i m_i=n-1$, and $m_i\le d$. Retain the second line of
\eqref{eq:rank-three-point-formula}, which counts pairs from different
classes. Since $m_i\le d=n-l\le n-4$, its positive parts can be removed.
Grouping its three contributions for each unordered triple of indices gives
\[
F_n(m_1,\ldots,m_t)=
\sum_{i<j<k}m_im_jm_k
\bigl(3(n-4)-m_i-m_j-m_k\bigr).
\]
For a list of remaining entries $R$ with sum $S$ and second elementary
symmetric function $E_2(R)$, splitting an entry $a+b$ changes this
quantity by
\[
ab\bigl(S(2n-11)+4E_2(R)\bigr)\ge0.
\]
Thus merging two entries whose sum is at most $d$ cannot increase $F_n$.
If there are more than $l$ entries, two entries have sum at most $d$;
otherwise, for the two smallest entries $a\le b$,
\[
\sum_i m_i\ge a+lb\ge\frac{(l+1)(d+1)}2>l+d-1,
\]
which is impossible.  Merge until there are $l$ entries.  With two entries
$2\le a\le b<d$ fixed, transferring one unit from $a$ to $b$ changes $F_n$ by
\[
-(b-a+1)\bigl(S(2n-11)+4E_2(R)\bigr)\le0.
\]
Iteration gives the extremal list $(d,1,\ldots,1)$, and substitution gives
\eqref{eq:pencil-concentration}.  The other summands in
\eqref{eq:rank-three-point-formula} are nonnegative.
\end{proof}

\begin{lemma}[Long-line point bound]
\label{lem:long-line-point-new}
Under the hypotheses of Lemma~\ref{lem:pencil-concentration-new}, every
point $a\in L$ satisfies
\begin{equation}\label{eq:long-line-point}
\Phi_a(M)\ge\binom{l-1}{2}(l-4),
\qquad
\sum_{a\in L}\Phi_a(M)\ge
\lambda(l):=\frac{l(l-1)(l-2)(l-4)}2.
\end{equation}
\end{lemma}

\begin{proof}
The pencil class $L\setminus\{a\}$ has size $l-1$.  Its first summand in
\eqref{eq:rank-three-point-formula} is at least
$\binom{l-1}{2}(l-4)$, and all other summands are nonnegative.  Sum over
$a\in L$.
\end{proof}

\begin{lemma}[Small-complement cases]
\label{lem:rank-three-small-complement-new}
Let $M$ be a non-near-pencil simple rank-three matroid on $n\ge8$ elements,
and let $L$ be a longest line with $d=n-|L|$.
If $d=2$, the smaller of the two possible line patterns has value
\begin{equation}\label{eq:rank-three-d2-value}
\frac12(5n^4-53n^3+155n^2-9n-346),
\end{equation}
whose difference from $P_n$ is
\begin{equation}\label{eq:rank-three-d2-gap}
n^4-10n^3+16n^2+87n-218>0.
\end{equation}
If $d=3$, the possible patterns, apart from two-point lines, are the
cross-triple patterns, one four-point cross-line, or the outside triple.
The smallest value is the four-point cross-line pattern, namely
\begin{equation}\label{eq:rank-three-d3-value}
\frac12(7n^4-73n^3+121n^2+781n-2180),
\end{equation}
and its difference from $P_n$ is
\begin{equation}\label{eq:rank-three-d3-gap}
2n^4-20n^3-n^2+482n-1135>0.
\end{equation}
\end{lemma}

\begin{proof}
For $d=2$, the two outside points either lie on a two-point line or on one
triple with a point of $L$. For $d=3$, write the outside points as $u,v,w$.
There are six patterns: $j=0,1,2,3$ of the pairs $uv,uw,vw$ extend to
triples with distinct points of $L$; or $uvw$ is an outside-only line; or
$uvw$ lies on one four-point line with a point of $L$. Distinct lines
cannot share a pair, so these alternatives are exhaustive. Substitution in
\eqref{eq:rank-three-formula} gives the displayed values. For an explicit
comparison in the $d=3$ case, put $x=n-8$. The differences from the
four-point cross-line value have the form $a+b x+c\binom{x}{2}$, with
\[
\begin{array}{c|rrr}
\text{pattern}&a&b&c\\ \hline
0\text{ cross-triples}&741&475&114\\
1\text{ cross-triple}&507&344&86\\
2\text{ cross-triples}&275&213&58\\
3\text{ cross-triples}&45&82&30\\
\text{outside triple}&498&337&84.
\end{array}
\]
They are strictly positive for every integer $x\ge0$. The coefficients
of the gaps \eqref{eq:rank-three-d2-gap} and
\eqref{eq:rank-three-d3-gap} in the Newton basis
$\binom{x}{0},\ldots,\binom{x}{4}$ are, respectively,
\[
(478,654,466,168,24),\qquad(609,1055,866,336,48).
\]
These give the asserted strict positivity for all $n\ge8$.
\end{proof}

\begin{lemma}[Polynomial gap]
\label{lem:rank-three-polynomial-gap-new}
For $n=l+d\ge13$, $d\ge4$, and $l\ge4$, put
\[
A(l,d)=\binom n2f(d),\qquad
B(l,d)=\left(\binom n2-\binom l2\right)f(l-1).
\]
Then
\begin{equation}\label{eq:rank-three-polynomial-gap}
\max\{A(l,d),B(l,d)\}+d\tau(l,d)+\lambda(l)>P_n.
\end{equation}
\end{lemma}

\begin{proof}
If $d\ge l$, use $A$; if $d<l$, use $B$.  The assertion is a direct
polynomial inequality.  For completeness, the boundary coefficient
certificates are as follows.  If $l=4,5,6$ and $d=13-l+x$, the coefficients
of $4(A+d\tau+\lambda-P_n)$ in increasing powers of $x$ are respectively
\[
\begin{array}{c|rrrrrr}
l&[x^0]&[x^1]&[x^2]&[x^3]&[x^4]&[x^5]\\ \hline
4&20448&18252&5250&683&42&1\\
5& 9624&12508&4143&589&39&1\\
6& 5328& 9368&3374&509&36&1
\end{array}
\]
and all are positive.  For $l\ge7$, substituting $l=7+x$, $d=l+u$ gives a
bivariate polynomial with nonnegative coefficients and positive constant
term.  On the other side, for $d=4,5$ and respectively $l=9+x,8+x$, the
coefficient lists of $4(B+d\tau+\lambda-P_n)$ are
\[
\begin{array}{c|rrrrr}
d&[x^0]&[x^1]&[x^2]&[x^3]&[x^4]\\ \hline
4&24336&18354&4218&384&12\\
5&14928&17584&4804&484&16
\end{array}
\]
and if $d\ge6$, substituting $d=6+x$, $l=d+1+u$ again gives nonnegative
coefficients and a positive constant term.  These cases cover $n\ge13$.
\end{proof}

\begin{lemma}[Finite compatibility table]
\label{lem:rank-three-finite-table-new}
For $8\le n\le12$, the cases not separated by the preceding bounds have the
following relaxed lower bounds.  Let $\lambda\vdash s$ denote an integer
partition and $\ell(\lambda)$ its number of parts.  With
$q(\lambda)=\sum_i\binom{\lambda_i}{2}$, define
\[
\mathsf O=\min\left\{\varphi_n(\lambda):
\lambda\vdash n-1,\ \ell(\lambda)\ge l,\ \max_i\lambda_i\le d\right\}
\]
and
\[
\mathsf I=\min\left\{\sum_{a=1}^{l}\varphi_n(l-1,\lambda_a):
\lambda_a\vdash d,\ \sum_{a=1}^{l}q(\lambda_a)\le\binom d2\right\}.
\]
The resulting values are
\begin{equation}\label{eq:rank-three-finite-table}
\begin{array}{c|c|c|r|r|r|r}
n&l&d&\max\{A,B\}&\mathsf O&\mathsf I&\text{lower bound}-P_n\\ \hline
8&4&4&168&108&424&79\\
8&3&5&560&61&317&237\\
9&5&4&216&306&1076&500\\
9&4&5&720&186&1034&668\\
9&3&6&1620&114&672&960\\
10&6&4&600&696&2178&1782\\
10&5&5&900&470&2340&1810\\
10&4&6&2025&290&2043&2028\\
10&3&7&3780&188&1299&2615\\
11&6&5&1100&1000&4478&4098\\
11&5&6&2475&675&4395&4440\\
12&6&6&2970&1365&8124&8889
\end{array}
\end{equation}
\end{lemma}

\begin{proof}
An outside pencil has at least $l$ classes, each of size at most $d$, giving
the lower bound $\mathsf O$.  For a point of $L$, the classes off $L$ form a
partition $\lambda_a$ of $d$.  A pair of outside points can lie in the same
off-$L$ class for at most one point of $L$, hence
$\sum_aq(\lambda_a)\le\binom d2$, which gives $\mathsf I$.  Listing the
partitions, or equivalently using the recurrence
\[
I_0(u)=0,\qquad
I_{j+1}(u)=
\min_{\substack{\lambda\vdash d\\q(\lambda)\le u}}
\bigl(I_j(u-q(\lambda))+\varphi_n(l-1,\lambda)\bigr),
\]
produces the displayed finite table, with $\mathsf I=I_l(\binom d2)$.
Relaxing line compatibility can only
decrease the minima, so they remain valid lower bounds.
\end{proof}

\begin{theorem}[Sharp rank-three extremum]
\label{thm:rank-three-sharp-extremum}
For every simple rank-three matroid $M$ on $n\ge8$ elements,
\begin{equation}\label{eq:rank-three-sharp-bound}
\dim_kT^2(M)_0\ge
\frac32(n-1)(n-2)(n-3)(n-5).
\end{equation}
Equality holds if and only if $M$ is a near-pencil, that is, it has one line
of size $n-1$.
\end{theorem}

\begin{proof}
The near-pencil calculation gives equality.  Suppose that $M$ is not a
near-pencil and let $L$ be a longest line, with $l=|L|$ and $d=n-l\ge2$.
The cases $d=2$ and $d=3$ are covered by
Lemma~\ref{lem:rank-three-small-complement-new}.  Assume $d\ge4$.  Since
line pairs partition the pairs of points and every line distinct from $L$
has size at most $d+1$,
\[
\Theta_3(M)\ge A(l,d),\qquad
\Theta_3(M)\ge B(l,d).
\]
If $n\ge13$ and $l\ge4$, combine these inequalities with
Lemmas~\ref{lem:pencil-concentration-new} and
\ref{lem:long-line-point-new}, and apply
Lemma~\ref{lem:rank-three-polynomial-gap-new}. For $n\ge13$ and $l\le3$,
the first bound alone is larger than $P_n$: in the Newton basis
$\binom{n-13}{j}$, $0\le j\le5$, the coefficients of $A-P_n$ are
\[
\begin{array}{c|rrrrrr}
l&0&1&2&3&4&5\\ \hline
2&18480&12408&5844&1839&348&30\\
3&8730&8144&4592&1617&330&30.
\end{array}
\]
For $8\le n\le12$, use \eqref{eq:rank-three-finite-table} and
the following coverage of its omitted cases. Put
$D=\max\{A,B\}+d\tau+\lambda-P_n$ for $l\ge4$ and
$D=A-P_n$ for $l\le3$. Direct substitution gives
\[
\begin{array}{c|l|l}
n&\text{remaining }l&\text{corresponding }D\\ \hline
8&2&315\\
9&2&1008\\
10&2&2520\\
11&2,3,4,7&5400,1220,30,405\\
12&2,3,4,5,7,8&10395,3861,1701,177,420,2457.
\end{array}
\]
Together these cover every integer $2\le l\le n-4$.
Every non-near-pencil is strictly above
the value in \eqref{eq:rank-three-sharp-bound}, proving the theorem.
\end{proof}

\begin{corollary}[Rank-three stability]
\label{cor:rank-three-stability-new}
If $M_n$ is a sequence of simple rank-three matroids with $|E(M_n)|=n$ and
\[
\frac{\dim_kT^2(M_n)_0}{P_n}\longrightarrow1,
\]
then $M_n$ is a near-pencil for all sufficiently large $n$.
\end{corollary}

\begin{proof}
Suppose a subsequence consists of non-near-pencils, and let $d=n-l\ge2$
on that subsequence. Since $P_n\sim\frac32n^4$, the bound
$\Theta_3\ge A(l,d)=\binom n2f(d)$ implies $d/n\to0$: otherwise $A$
has order $n^5$ along a further subsequence. Thus $l/n\to1$. If $d\ge4$,
then
\[
\frac{B(l,d)}{P_n}
=\frac{d(2n-d-1)(l-1)(l-2)(l-4)}
{6(n-1)(n-2)(n-3)(n-5)}
=\frac d3(1+o(1)),
\]
where the error tends to zero whenever $d/n\to0$. Its lower limit is at
least $4/3$, contradicting the assumed ratio. The cases $d=2$ and $d=3$
are excluded by \eqref{eq:rank-three-d2-gap} and
\eqref{eq:rank-three-d3-gap}, whose ratios to $P_n$ tend respectively to
$2/3$ and $4/3$. No infinite non-near-pencil subsequence remains.
\end{proof}

The exhaustive line-system calculation described in
Section~\ref{sec:verification} gives minima $0$, $54$, and $357$ on
$5$, $6$, and $7$ points, respectively; the seven-point value is attained
by the Fano configuration. These finite results are separate from the
proof for $n\ge8$.

\section{A valuative cocircuit profile}
\label{sec:valuative-profile}

The scalar invariant $\Theta$ is a weighted sum of a more informative finite
profile.  Let $M$ have rank $r\ge1$, ground-set size $n=r+c$, and define
\[
a_m(M)=
\sum_{\substack{H\text{ hyperplane of }M\\|H|=r-1+m}}
b(M|H),
\qquad 0\le m\le c.
\]
Set $a_m(M)=0$ for integer $m$ outside this range.
Equivalently, if $D$ runs over cocircuits, then
\[
\mathscr C_M(z)=\sum_{m=0}^{c}a_m(M)z^{c+1-m}
=\sum_{D\text{ cocircuit}}b(M|(E\setminus D))z^{|D|}.
\]
The coefficient $a_m$ records the number of bases of hyperplanes having $m$
elements beyond a basis.

The recovery of flat data from the $\mathcal G$-invariant is established in broader form by Bonin and Kung \cite[Theorem~3.3 and Proposition~5.2]{BoninKung}: the invariant is equivalent to catenary data and determines the numbers of flats of each rank and size. Ferroni and Schr\"oter \cite[Theorem~8.2]{FerroniSchroter} prove valuativity of the flag $f$-vector of the lattice of flats. The formulas below identify the particular basis-weighted hyperplane profile occurring in the cotangent-cohomology calculation and give its transform directly in rank-word coordinates.

Let $\mathcal G(M)$ denote the Derksen--Fink $\mathcal G$-invariant
\cite{DerksenFink} in its rank-word basis.  A rank word has a $1$ when the next element raises rank and
a $0$ otherwise.  For $0\le t\le c$, let $g_t(M)$ be the coefficient of
\[
w_t=1^{r-1}0^t1\,0^{c-t}
\]
in $\mathcal G(M)$, and write
$(m)_t=m(m-1)\cdots(m-t+1)$.

\begin{theorem}[Rank-word transform for the cocircuit profile]
\label{thm:valuative-profile-new}
For every matroid of positive rank,
\begin{equation}\label{eq:rank-word-profile-new}
g_t(M)=(r-1)!(c-t)!
\sum_{m=t}^{c}a_m(M)(m)_t(c+1-m),
\qquad0\le t\le c.
\end{equation}
Consequently, with
\[
R_M(x)=\frac{1}{(r-1)!c!}
\sum_{t=0}^{c}\binom ct g_t(M)(x-1)^t,
\]
one has
\begin{equation}\label{eq:profile-inversion-new}
R_M(x)=\sum_{m=0}^{c}a_m(M)(c+1-m)x^m,
\qquad
a_m(M)=\frac{[x^m]R_M(x)}{c+1-m}.
\end{equation}
In particular, $\mathscr C_M$, the full vector $(a_0,\ldots,a_c)$, and
$\Theta(M)$ are linear specializations of $\mathcal G(M)$.  For fixed rank
and ground-set size they are valuative under matroid base-polytope
subdivisions.
\end{theorem}

\begin{proof}
Fix a hyperplane $H$ with $|H|=r-1+m$.  Choose an ordered basis of $M|H$;
there are $(r-1)!b(M|H)$ choices.  After these first $r-1$ entries, choose
an ordered $t$-tuple from the $m$ remaining elements of $H$, then choose the
next element outside $H$, which has $c+1-m$ possibilities.  Order the
remaining $c-t$ elements arbitrarily.  This produces the word $w_t$.
Conversely, every permutation with word $w_t$ is obtained uniquely from the
hyperplane spanned by its first $r-1$ entries.  Summing over $H$ gives
\eqref{eq:rank-word-profile-new}.

Substitution into the definition of $R_M$ uses
\[
\frac{\binom ct(c-t)!}{c!}(m)_t=\binom mt
\]
and the binomial identity
$\sum_{t=0}^{m}\binom mt(x-1)^t=x^m$.  This gives
\eqref{eq:profile-inversion-new}.  Rank-word coefficients are linear
functionals of the universal valuative invariant $\mathcal G$
\cite[Theorem~1.4]{DerksenFink}, and linear
transformations preserve valuativity.  Finally,
\eqref{eq:Theta-cocircuit} is a fixed linear functional of the $a_m$.
\end{proof}

\begin{corollary}[Valuative rank-one obstruction series]
\label{cor:valuative-series-new}
The coarse Hilbert series of the rank-one-link strand and the full polynomial
$\mathscr C_M$ are valuative after fixing $r$ and $n$. The profile is not
determined by the Tutte polynomial: the two graphic matroids in
Theorem~\ref{thm:Tutte-contrast} have the same Tutte polynomial but different
values of $\Theta$, hence different profiles.
\end{corollary}

\begin{proof}
The first assertion follows from Theorem~\ref{thm:valuative-profile-new}
and the rank-one-link Hilbert-series formula.  The second follows because
$\Theta$ is a specialization of the profile and is not determined by the
Tutte polynomial.
\end{proof}

The next proposition separates two layers of the obstruction data:
the rank-one-link profile is valuative, whereas the purely negative
pair rank is not.

\begin{proposition}[The boundary of valuativity]
\label{prop:valuativity-boundary}
Let $p_d(M)$ denote the number of independent $(r-1)$-sets $A$ with
$d_A=|E\setminus\cl_M A|=d$.
\begin{enumerate}[label=\textup{(\arabic*)}]
\item $p_d(M)=a_{c+1-d}(M)$ for all $d$.  Hence the obstruction profile is
valuative for fixed rank and ground-set size, although it is not
log-concave.
\item The global pair rank $\mathfrak q$ is not valuative.  In the split
subdivision of the base polytope of $U_{2,6}$ along $x_1+x_2+x_3=1$, the
maximal cells are the polytopes of the two rank-two matroids with class
profile $(3,1,1,1)$, and their common face is the polytope of
$U_{1,3}\oplus U_{1,3}$.  The valuative relation
$f(U_{2,6})=f(M_1)+f(M_2)-f(M_{12})$ holds for $\Theta$, with
$120=60+60-0$, but fails for $\mathfrak q$, with $120\ne51+51-0$.
\item $\mathfrak q$ is not determined by the Tutte polynomial: the two
Tutte-equivalent graphic matroids of Theorem~\ref{thm:Tutte-contrast} have
$\mathfrak q=2009$ and $\mathfrak q=2165$.
\end{enumerate}
\end{proposition}

\begin{proof}
For (1), mapping $A$ to its closure is a bijection from the independent
$(r-1)$-sets with $d_A=d$ to the pairs consisting of a hyperplane $H$ with
$|E\setminus H|=d$ and a basis of $M|H$.  Hence $p_d=a_{c+1-d}$, and
Theorem~\ref{thm:valuative-profile-new} gives valuativity. For the rank-two
matroid with class sizes $(3,1,1)$, the complete profile is
\[
(p_0,\ldots,p_5)=(0,0,3,0,2,0),
\]
which is not log-concave. Indeed, the three points in the large class have
complementary cocircuit size two, and the other two points have
complementary cocircuit size four.

For (2), a basis of the first cell matroid is a pair containing at most one
element of $\{1,2,3\}$, a basis of the second is a pair containing at most
one element of $\{4,5,6\}$, and the common face consists of the pairs
containing exactly one element of each; this is the standard split
subdivision.  The values of $\mathfrak q$ follow from
Corollary~\ref{cor:rank-two-pure-pair-rank} and
Corollary~\ref{cor:pair-rank-direct-sum}, and the values of $\Theta$ from
Proposition~\ref{prop:rank-two-closed-form} and
Corollary~\ref{cor:Theta-direct-sum}.  A valuative invariant satisfies the
displayed inclusion--exclusion relation, so $\mathfrak q$ is not valuative.

For (3), evaluate \eqref{eq:pair-dim} on the two matroids; the exact
computation is recorded in Section~\ref{sec:verification}.
\end{proof}

\begin{remark}[Higher diagonal sparsity]
In characteristic zero, Ilten, Meazzini, and Petracci prove that, for
$i\ge2$, a fine component of $T^i$ vanishes if one negative coordinate
has multiplicity at least $i$ \cite[Theorem~5.1]{IltenMeazziniPetracci}.
Consequently, if $\mathbf c$ has a nonempty negative support, a component
of $T^i$ on the ray $m\mathbf c$, with $m\ge1$, can be nonzero only when
$m<i$. The $i=2$ vanishing used in Section~\ref{sec:cup} holds over every
field by the support criterion for $T^2$.
\end{remark}

\section{Rank-two support and the primary cup image}\label{sec:primary-image}

The graph model admits a further simplification: its only disconnected
blocks occur when the negative support has dual rank two. This also
separates the part of $T^2$ reached by the primary cup product from the
part in which higher obstructions can occur.

\begin{lemma}[Trichotomy for the hyperplane graphs]\label{lem:graph-trichotomy}
Let $N$ have rank $s\ge2$, let $B\ne\varnothing$, and let $H$ be a
hyperplane avoiding $B$. Put $b=\rk_N B$. The graph $\Gamma_B(H)$ is
empty if $b\le1$, nonempty and discrete if $b=2$, and nonempty and connected if $b\ge3$.
\end{lemma}

\begin{proof}
A rank-$(s-2)$ flat gains at most $b$ ranks upon adjoining $B$.
This proves the first assertion. Two distinct such flats have intersection
rank at most $s-3$, so their intersection cannot span with a rank-two set.
For $s=2$ there is at most one rank-zero flat, and discreteness still holds.
When $b=2$, the matroid $(N/B)|H$ has rank $s-2$. A basis of it
is independent in $N$, spans with $B$, and closes inside $H$ to a
qualifying rank-$(s-2)$ flat. Thus the discrete graph is nonempty.

Suppose $b\ge3$ and put $Q=(N/B)|H$. Since $H\cup B$ spans $N$, the
rank of $Q$ is $d=s-b$. Every vertex $G$ contains a basis $I$ of $Q$:
indeed $\rk_QG=\rk_N(G\cup B)-b=d$. Such a basis is also independent
in $N$. Vertices containing the same $I$ form a clique, since
$\rk_N(I\cup B)=s$. The basis-exchange graph of $Q$ is connected.
If $I,J$ are adjacent bases, then $|I\cup J|=d+1\le s-2$.
Extend their closure inside $H$ to a rank-$(s-2)$ flat $K$.
It contains both bases and is a vertex of $\Gamma_B(H)$, joining the
two cliques. Every vertex belongs to one of these cliques. Extending
any basis of $Q$ also proves nonemptiness. If $d=0$, all vertices lie
in the clique corresponding to the empty basis.
\end{proof}

\begin{theorem}[Rank-two negative support]\label{thm:rank-two-support}
Let $N=M^*$. A nonzero component $T^2_{\mathbf a-\one_B}(M)$ has
$\rk_N B=2$. Subject to the support conditions of
Theorem~\ref{thm:universal-graph}, put $N_A=(\lk_M A)^*$, where
$A=\supp\mathbf a$. If $\rk_N B=2$, then
\begin{equation}\label{eq:rank-two-flat-count}
T^2_{\mathbf a-\one_B}(M)\cong
\bigoplus_{\substack{H\text{ hyperplane of }N_A\\H\cap B=\varnothing}}
k^{\mathcal G_B(H)}/k\one,
\qquad
\dim_kT^2_{\mathbf a-\one_B}(M)=
\sum_{H\cap B=\varnothing}(|\mathcal G_B(H)|-1)_+.
\end{equation}
Empty blocks contribute zero. Every nonzero $T^1$ component has negative
support of rank one in $N$.
\end{theorem}

\begin{proof}
Apply Lemma~\ref{lem:graph-trichotomy} to the universal graph formula.
Singleton negative supports have zero $T^2$ by
\eqref{eq:singleton-formula}; dual ranks below two are covered by the
same formula. To compare the dual ranks before and after passing to a
link, let $D$ be the loops of $M/A$. They are coloops of $N\setminus A$,
and $N_A=(N\setminus A)/D$. Deletion leaves the ranks of subsets of
$E\setminus A$ unchanged, and contracting coloops leaves the ranks of
subsets disjoint from $D$ unchanged. Thus $\rk_{N_A}B=\rk_NB$ for
every admissible $B$. The last assertion follows from
\eqref{eq:universal-T1}: if $\rk_NB\ge2$, every eligible graph is
nonempty, whereas rank-zero $B$ lies in every hyperplane. For singleton
supports use \eqref{eq:singleton-formula}. This is a reformulation of
the first-cohomology description of Brehm and Constantinescu
\cite[Theorem~4.7]{BrehmConstantinescu}, not an additional $T^1$ theorem.
\end{proof}

Write $R=k[M]$ and $\Omega=\im(T^1(R)\otimes_R T^1(R)\to T^2(R))$.
The image here includes every homogeneous factorization of a target degree,
not only products of singleton classes.

\begin{theorem}[The purely negative primary image]\label{thm:negative-primary-image}
Let $N=M^*$ and let $\mathbf c=\mathbf a-\one_B$ have disjoint
nonnegative parts. If $B$ meets at least three nonloop parallel classes
of $N$, then $\Omega_{\mathbf c}=0$. If $B$ contains a loop of $N$,
then $T^2_{-\one_B}=0$. Otherwise,
\[
\Omega_{-\one_B}=
\begin{cases}
T^2_{-\one_B}(M),&B\text{ meets exactly two parallel classes},\\
0,&\text{otherwise}.
\end{cases}
\]
In particular, for simple $N$ every purely negative support of size at
least three is invisible to the primary cup product.
\end{theorem}

\begin{proof}
By Theorem~\ref{thm:rank-two-support}, the negative support of a nonzero
$T^1$ class is confined to one nonloop parallel class of $N$.
Every element of $B$ must occur in a negative support of one of the two
factors. Two factors can therefore cover at most two parallel classes,
even when positive and negative supports cancel outside $B$.

Suppose $B=B_1\sqcup B_2$, with nonempty $B_i$ contained in distinct
parallel classes $P_i$, and choose $p_i\in B_i$. A circuit of $M$,
being a cocircuit of $N$, contains either all of $P_i$ or none of it.
Consequently multiplication induces
\begin{equation}\label{eq:parallel-reduction-maps}
x_{B_i\setminus\{p_i\}}:T^1_{-\one_{B_i}}\twoheadrightarrow
T^1_{-\ee_{p_i}},\qquad
x_{B\setminus\{p_1,p_2\}}:T^2_{-\one_B}\xrightarrow{\sim}
T^2_{-\ee_{p_1}-\ee_{p_2}}.
\end{equation}
For completeness these are maps, not just equalities of dimensions.
In the monomial presentation, a union of circuits contains $B_i$ if
and only if it contains $p_i$. After deleting either $B_i$ or $p_i$,
the resulting monomial is divisible by no circuit meeting $P_i$;
circuits disjoint from $P_i$ are unaffected. The valid Taylor cochain
coordinates and their incidence matrices therefore agree. For the
Koszul conditions the relevant exponents on $P_i$ are uniformly zero,
one, or two before deletion. In each case deleting $B_i$ or just $p_i$
gives the same test for divisibility by a circuit. Thus the same
identification respects the Koszul quotient defining $T^2$.
Multiplication is the identity on these coordinates. For $T^1$, the
singleton target additionally quotients by the derivation image when
$|B_i|>1$, giving the stated surjection.

Lift the two factors of every singleton-pair cup product through
\eqref{eq:parallel-reduction-maps}. Its $R$-bilinearity and the target
isomorphism, together with Theorem~\ref{thm:pure-pair-surjectivity},
prove surjectivity at $-\one_B$. If $B$ contains a loop of $N$, no
hyperplane avoids it, so the target is zero. If $B$ meets at most one
parallel class, the target also vanishes by rank-two support.
\end{proof}

The circuit description of Brehm and Constantinescu
\cite[Proposition~4.5 and Theorem~4.7]{BrehmConstantinescu} also gives
the following multiplication maps. These maps allow positive supports
with exponent one; multiplying both factors by a positive monomial
would not reach such degrees.

\begin{lemma}[Generation of first cohomology]\label{lem:negative-generation}
Let $\mathbf a\ge0$ and $B\ne\varnothing$, with
$\supp\mathbf a\cap B=\varnothing$. Multiplication induces a surjection
\[
x^{\mathbf a}:T^1_{-\one_B}(M)\twoheadrightarrow
T^1_{\mathbf a-\one_B}(M).
\]
Consequently $T^1(M)$ is generated as a $k[M]$-module by its purely
negative components.
\end{lemma}

\begin{proof}
The support criterion and Theorem~\ref{thm:rank-two-support} reduce
the assertion to $B$ contained in one nonloop parallel class $P$ of
$N=M^*$. Put $S=k[x_e:e\in E]$ and use the minimal circuit generators
$x_C$ of $I_M$. A circuit contains either all of $P$ or none of it.
Every $C$ containing $B$ has $C\setminus B\in M$: its complement is
$H\cup B$, where $H=E\setminus C$ is a hyperplane avoiding $P$,
and $H\cup B$ spans $N$.

For distinct circuits $C,D$, write $K=E\setminus D$ and $s=\rk N$.
If $B\subseteq C\cup D$, then
\[
\rk_N((H\cap K)\cup B)\le s-2+1<s.
\]
Thus $(C\cup D)\setminus B$ is a nonface of $M$. The evaluated
Taylor pair relation in degree $-\one_B$ is therefore zero in $R$;
the same holds after adding $\mathbf a$. If $B\not\subseteq C\cup D$,
that relation has no output monomial of nonnegative degree. When
$s=1$, there is at most one circuit and no distinct pair to check.
Taylor pair relations generate the first syzygies. It follows that
$\operatorname{Hom}_S(I_M,R)_{-\one_B}$ is the full coordinate space
on the circuits containing $B$, whereas its degree
$\mathbf a-\one_B$ part retains precisely those coordinates with
$(C\setminus B)\cup\supp\mathbf a\in M$. Multiplication by
$x^{\mathbf a}$ is the coordinate projection and is surjective.
The map $\operatorname{Der}_k(S,R)\to\operatorname{Hom}_S(I_M,R)$
is $R$-linear, so this surjection descends to its cokernel $T^1$.
All other fine degrees of $T^1$ vanish by the support criterion.
\end{proof}

\begin{theorem}[The primary image in every fine degree]\label{thm:full-primary-image}
Let $N=M^*$, and let $p_N(B)$ be the number of nonloop parallel
classes of $N$ met by $B$. For every fine degree
$\mathbf c=\mathbf a-\one_B$ with disjoint nonnegative parts,
\[
\Omega_{\mathbf c}=\begin{cases}
T^2_{\mathbf c}(M),&p_N(B)=2,\\
0,&p_N(B)\ne2.
\end{cases}
\]
Degrees excluded by the support criterion have zero target, including
degrees with nonsquarefree negative part. The target also vanishes
if $B$ contains a loop of $N$.
\end{theorem}

\begin{proof}
The argument in Theorem~\ref{thm:negative-primary-image} excludes
every factorization when $p_N(B)\ge3$. Each element of $B$ must
belong to the negative support of a factor, and two nonzero $T^1$
factors cover at most two dual parallel classes. This remains true
with cancellation outside $B$. Homogeneous tensors span the image;
an $R$-multiple can be absorbed into one factor.

Suppose $p_N(B)=2$ and the target is admissible and potentially
nonzero. Write $B=B_1\sqcup B_2$ according to the two parallel
classes, put $A=\supp\mathbf a$, and let $L=\lk_M A$ on its actual
vertices. If $D$ consists of the loops of $M/A$, then
$L^*=(N\setminus A)/D$. The elements of $D$ are coloops of
$N\setminus A$; deleting $A$ and contracting $D$ preserve the
ranks and parallel-class distinction on $B$.

Set $U=T^1_{-\one_{B_1}}(M)$,
$V=T^1_{\mathbf a-\one_{B_2}}(M)$, and
$W=T^2_{\mathbf a-\one_B}(M)$. Positive-support localization and
Lemma~\ref{lem:negative-generation} give a commuting square
\[
\begin{array}{ccc}
U\otimes_k V&\xrightarrow{\ \cup\ }&W\\
\big\downarrow{\scriptstyle\lambda_1\otimes\lambda_2}&&
\big\downarrow{\scriptstyle\lambda}\\
T^1_{-\one_{B_1}}(L)\otimes_k T^1_{-\one_{B_2}}(L)
&\xrightarrow{\ \cup\ }&T^2_{-\one_B}(L).
\end{array}
\]
Here $\lambda_1$ is multiplication by $x^{\mathbf a}$ followed by
the normalized link isomorphism, hence is surjective;
$\lambda_2$ and $\lambda$ are the normalized link isomorphisms.
To check the degrees and commutativity, localize at $x_A$. In
$R_{x_A}=k[L][x_i^{\pm1}:i\in A]$, the two factors have the form
$\alpha$ and $x^{\mathbf a}\beta$. Their cup is
$x^{\mathbf a}(\alpha\cup\beta)$, and target normalization removes
exactly $x^{\mathbf a}$. The lower cup is surjective by
Theorem~\ref{thm:negative-primary-image} applied to $L$.
Lifting finite sums of its products proves surjectivity onto $W$.
For $p_N(B)\le1$, or for $B$ containing a dual loop, the target
vanishes by Theorem~\ref{thm:rank-two-support} and the graph formula.
\end{proof}

\begin{corollary}[The quotient by the primary image]\label{cor:primary-quotient}
The $R$-module $\Omega$ is generated by its purely negative
components meeting exactly two dual parallel classes. In each
admissible degree $\mathbf c=\mathbf a-\one_B$ with $\rk_NB=2$
and $p_N(B)\ge3$, the quotient $Q=T^2(M)/\Omega$ has
\[
Q_{\mathbf c}\cong T^2_{\mathbf c}(M)\cong
\bigoplus_{\substack{H\text{ hyperplane of }(\lk_M A)^*\\
H\cap B=\varnothing}}k^{\mathcal G_B(H)}/k\one.
\]
All other components vanish. These are degreewise vector-space
identifications, and the displayed components may themselves be
zero. Their dimensions are independent of the characteristic.
\end{corollary}

\begin{proof}
Lemma~\ref{lem:negative-generation} and $R$-bilinearity generate
every cup from purely negative cups. Overlapping negative supports
give a nonsquarefree negative part and hence zero target. Theorems
\ref{thm:negative-primary-image}, \ref{thm:full-primary-image}, and
\ref{thm:rank-two-support} give the assertions, with empty blocks
interpreted as zero as in \eqref{eq:rank-two-flat-count}.
\end{proof}

\section{Essential obstructions of arbitrarily high order}\label{sec:higher-obstructions}

Vanishing of every pairwise cup product does not suffice for simultaneous
fine-multigraded integration of three or more classes. We now give an explicit family
whose first unavoidable obstruction occurs at any prescribed order.
The obstruction remains in $T^2$; its order does not refer to the index
of cotangent cohomology.

Let $m,n\ge3$, let $L,H$ be disjoint sets of sizes $m,n$, and put
\begin{equation}\label{eq:two-line-matroid}
N_{m,n}=\operatorname{Tr}_3(U_{2,m}\oplus U_{2,n}),\qquad
M_{m,n}=N_{m,n}^*.
\end{equation}
Here $\operatorname{Tr}_3$ denotes rank-three truncation. The hyperplanes
of $N_{m,n}$ are $L$, $H$, and the cross pairs $\{j,h\}$.
The matroid is simple and connected: three-element subsets of either
line are circuits, and two points from each line form a circuit.
Thus $M_{m,n}$ is connected, cosimple, and has corank three.

Use cyclic indices $j\in\ZZ/m\ZZ$ on $L$, variables $x_j,y_h$, and set
\[
X=\prod_{j\in L}x_j,\qquad Y=\prod_{h\in H}y_h,\qquad W_h=Y/y_h.
\]
The circuit ideal is
\begin{equation}\label{eq:two-line-ideal}
I_{M_{m,n}}=(X,Y,(X/x_j)W_h:j\in L,\ h\in H).
\end{equation}
Choose $h_0\in H$, put $P=t_0\cdots t_{m-1}$ and
$\sigma=(-1)^{m-1}$, and define
\begin{equation}\label{eq:cyclic-polynomials}
F_j=\sum_{a=0}^{m-1}(-1)^a
\left(\prod_{b=1}^a t_{j-b}\right)
\left(\prod_{i\notin\{j,j-1,\ldots,j-a\}}x_i\right).
\end{equation}
Empty products are one. Give $t_i$ fine degree $\ee_i$.

\begin{theorem}[Cyclic higher obstruction]\label{thm:cyclic-higher-obstruction}
For every field $k$, the ideal
\begin{equation}\label{eq:cyclic-family}
J=(X,Y,F_jW_{h_0}:j\in L,
              (X/x_j)W_h:j\in L,\ h\ne h_0)
\end{equation}
defines a flat fine-multigraded deformation over $k[t_0,\ldots,t_{m-1}]/(P)$.
Its first-order classes $\alpha_i\in T^1_{-\ee_i}(M_{m,n})$ are nonzero.
Every proper subcollection integrates over its unrestricted polynomial
parameter ring, and all pairwise primary cup products vanish.

Put $\mathfrak m=(t_0,\ldots,t_{m-1})$. The induced deformation over
$k[t]/\mathfrak m^m$ has no lift to $k[t]/\mathfrak m^{m+1}$, even
without requiring the lift to preserve the fine grading. In the
fine-multigraded category its jet modulo $\mathfrak m^m$ is uniquely
determined, up to the usual equivalence, by the first-order classes.
Its first essential obstruction has order $m$ and lies in
$T^2_{-\one_L}(M_{m,n})$, a space of dimension $n-1$.
\end{theorem}

\begin{proof}
Shifting the summation index in \eqref{eq:cyclic-polynomials} gives
the telescoping identity
\begin{equation}\label{eq:cyclic-identity}
x_jF_j+t_{j-1}F_{j-1}=X+\sigma P.
\end{equation}
Work first over $B=\ZZ[t]/(P)$ and regard the parameters as coefficients.
Choose an order on monomials in $x,y$ refining total degree. The displayed
generators of $J$ are monic and have the leading monomials in
\eqref{eq:two-line-ideal}. We verify the monic Buchberger criterion
over $B$, which is valid over a coefficient ring with zero divisors
because no leading coefficient is divided by a nonunit.

Two unchanged generators have zero S-polynomial. For two modified
generators the S-polynomial is
$W_{h_0}(x_jF_j-x_kF_k)$. Subtracting their identities
\eqref{eq:cyclic-identity} expresses it as a combination of modified
generators with strictly smaller $x,y$ degree than their least common
multiple. For $X$ and a modified generator the same identity gives
a multiple of another modified generator and the remainder
$-\sigma P W_{h_0}=0$. For $Y$ and a modified generator every tail
is divisible by $Y$. Finally, for a modified generator and an unchanged
cross generator with $h\ne h_0$, the least common multiple of $W_h$
and $W_{h_0}$ is $Y$, so every tail is again divisible by $Y$.
These exhaust the S-pairs. Division by a monic basis shows that the
original standard monomials form a $B$-basis of $B[x,y]/J$.
This proves flatness over every field. Setting any parameter to zero
makes $P$ identically zero, proving the proper-subcollection assertion.

The coefficient of $t_i$ in $F_{i+1}$ is
$-X/(x_ix_{i+1})$. Thus $\alpha_i$ is, up to sign, the singleton
hyperplane delta function at $\{i+1,h_0\}$; it is not constant and
is nonzero. Every pair belongs to a proper subcollection, so its
mixed primary obstruction is zero. Each one-parameter restriction is
also a polynomial deformation, so self-obstructions vanish directly,
including in characteristic two.

We next rule out arbitrary, possibly nonhomogeneous, lifts of the
specified jet. Set $A=k[t]/\mathfrak m^{m+1}$. In any flat lift,
choose lifts of the variables and of the displayed ideal generators.
The lifted variables generate the algebra since the base ideal is
nilpotent. Flatness identifies the reduction of its presentation kernel
with $I_{M_{m,n}}$; lifted circuit generators generate that kernel by
the nilpotent form of Nakayama's lemma.
They differ from \eqref{eq:cyclic-family} by parameter monomials of
total degree $m$ with arbitrary coefficients in $k[x,y]$.
Write $g_X$ and $g_{j,h}$ for their coefficients at $P$.
The kernel $\mathfrak m^m$ of $A\to k[t]/\mathfrak m^m$ is
annihilated by $\mathfrak m$. Flatness therefore makes
\[
\mathfrak m^m\otimes_k k[M_{m,n}]\longrightarrow\mathcal A
\]
injective, where $\mathcal A$ is the lifted algebra.
Multiply the modified relation by $x_j$, add $t_{j-1}$ times the
preceding modified relation, and subtract $W_{h_0}$ times the lifted
relation for $X$. By \eqref{eq:cyclic-identity}, its $P$ coefficient
must vanish in the special fiber:
\begin{equation}\label{eq:critical-socle-relations}
\sigma W_{h_0}+x_jg_{j,h_0}-W_{h_0}g_X=0.
\end{equation}
Likewise the unchanged cross relation for $h\ne h_0$ gives
$x_jg_{j,h}-W_hg_X=0$. Corrections at other parameter monomials
do not affect these $P$ coefficients, since a parameter times any
order-$m$ correction is zero in $A$.

Project the first equation to fine degree $\deg W_{h_0}$, and the
second to fine degree $\deg W_h$. No $x_jg_{j,h}$ can contribute,
since the $j$ coordinate of these degrees is zero. If $\lambda$ is
the constant term of $g_X$, the two projections require
$(\sigma-\lambda)W_{h_0}=0$ and $\lambda W_h=0$ in $k[M_{m,n}]$.
Every $W_h$ is a nonzero standard monomial, so $\lambda=\sigma$
and $\lambda=0$, a contradiction over every field. Lifting the
variables presents any abstract affine lift in this form;
thus the contradiction is not restricted to a polynomial ansatz or
to homogeneous correction terms.

It remains to justify uniqueness from the first-order classes in the
stated graded category. For a parameter monomial $t^{\mathbf u}$ of
total degree at least two, $T^1_{-\mathbf u}(M_{m,n})=0$. If $\mathbf u$
is not squarefree this follows from the support criterion; otherwise
its support has rank at least two in the simple dual $N_{m,n}$, and
Theorem~\ref{thm:rank-two-support} applies. Derivations of
$k[x,y]$ into $k[M_{m,n}]$ in degree $-\mathbf u$ are also zero.
Put $S=k[x,y]$ and $R=k[M_{m,n}]$. The presentation exact sequence
\[
\operatorname{Der}_k(S,R)\longrightarrow
\operatorname{Hom}_S(I_{M_{m,n}},R)\longrightarrow T^1(R)
\longrightarrow0
\]
therefore gives $\operatorname{Hom}_S(I_{M_{m,n}},R)_{-\mathbf u}=0$.
Indeed the image of a variable under a derivation would have degree
$\ee_v-\mathbf u$, with a negative coordinate when $|\mathbf u|\ge2$.
Representatives of the same first-order classes can first be aligned
by graded coordinate changes realizing their derivation differences.
Differences between two flat embedded lifts across an order-filtration
small extension lie in these spaces. Induction shows that all
coefficients of orders $2,\ldots,m-1$ are forced. There is no ambiguity
in an ideal-valued coefficient: subtracting a nonempty support from
the degree of a minimal circuit generator leaves either a negative
coordinate or a proper face of that circuit. After choosing the
same first-order representatives, this proves graded uniqueness.
The obstruction has parameter weight $\one_L$; its dimension is
computed in Proposition~\ref{prop:two-line-primary-quotient} below.
\end{proof}

\begin{example}[A six-element cubic obstruction]\label{ex:essential-cubic}
For $m=n=3$, the polynomials are
\[
F_0=x_1x_2-t_2x_1+t_2t_1,\quad
F_1=x_0x_2-t_0x_2+t_0t_2,\quad
F_2=x_0x_1-t_1x_0+t_1t_0.
\]
For example $x_0F_0+t_2F_2=X+t_0t_1t_2$.
Any two of the three classes integrate jointly, but the specified
second-order deformation has a nonzero cubic lifting obstruction.
The critical block $T^2_{-\one_L}$ has dimension two.
\end{example}

\begin{proposition}[The quotient missed by the primary cup]\label{prop:two-line-primary-quotient}
For $M=M_{m,n}$, let $\Omega$ be its full primary image. If
$B\subseteq L$ and $|B|\ge2$, then $\dim_kT^2_{-\one_B}(M)=n-1$;
if $B\subseteq H$ and $|B|\ge2$, the dimension is $m-1$. Moreover,
\begin{equation}\label{eq:two-line-primary-quotient}
\sum_{\varnothing\ne B\subseteq L\sqcup H}
\dim_k(T^2/\Omega)_{-\one_B}
=(n-1)\left(2^m-1-m-\binom m2\right)
 +(m-1)\left(2^n-1-n-\binom n2\right).
\end{equation}
\end{proposition}

\begin{proof}
For $B\subseteq L$ with $|B|\ge2$, the opposite hyperplane $H$
contributes its $n$ rank-one flats as discrete vertices. Each other
hyperplane avoiding $B$ is a cross pair and contributes exactly one
qualifying vertex. This gives dimension $n-1$; interchange the lines
for the second assertion. Every support of size at least three not
confined to one line spans $N_{m,n}$ and has zero $T^2$ by
Theorem~\ref{thm:rank-two-support}. Pair degrees are in $\Omega$,
whereas supports of size at least three are not, by
Theorem~\ref{thm:negative-primary-image}. Counting the remaining
subsets of each line proves \eqref{eq:two-line-primary-quotient}.
\end{proof}

With $n=3$, the quotient in \eqref{eq:two-line-primary-quotient} grows
exponentially with $m$, while the corank remains three. The cyclic
construction gives one actual order-$m$ obstruction. The next theorem
realizes every homogeneous class in this quotient by allowing a subset
of either line and weights in the cyclic polynomials.

\subsection{Realization of the purely negative quotient}\label{sec:negative-realization}

The dimension in \eqref{eq:two-line-primary-quotient} counts classes that
the primary cup product cannot reach. We next show that every homogeneous
class counted there nevertheless occurs as an actual higher obstruction.
The construction allows both the negative support and the obstruction
value to be prescribed.

Fix $B\subseteq L$ of size $r\ge3$, choose a cyclic order on $B$, and put
\[
Q=\prod_{i\in B}x_i,\qquad Z=\prod_{i\in L\setminus B}x_i,
\qquad P_B=\prod_{i\in B}t_i,\qquad \sigma=(-1)^{r-1}.
\]
Thus $X=ZQ$. For weights $u_{i,h}\in k$ set
$v_h=\prod_{i\in B}u_{i,h}$ and define
\begin{equation}\label{eq:weighted-cyclic-polynomials}
F_{j,h}=\sum_{a=0}^{r-1}(-1)^a
\left(\prod_{b=1}^a u_{j-b,h}t_{j-b}\right)
\left(\prod_{i\in B\setminus\{j,j-1,\ldots,j-a\}}x_i\right).
\end{equation}
All indices on $B$ in this formula are cyclic. Write
$C_{j,h}=(X/x_j)W_h$ for the original cross generators and consider
\begin{equation}\label{eq:weighted-cyclic-family}
J_{B,u}=(X,Y,ZF_{j,h}W_h:j\in B,\ h\in H,
                 C_{j,h}:j\notin B,\ h\in H).
\end{equation}
The parameters have degree $\deg t_i=\ee_i$.

\begin{theorem}[Realization in every negative component]\label{thm:negative-realization}
Let $R=k[M_{m,n}]$ over any field. For $B\subseteq L$ with $|B|=r\ge3$
there is an explicit linear isomorphism
\[
\theta_B:k^H/k\one\xrightarrow{\sim}T^2_{-\one_B}(R).
\]
Every nonzero $\xi$ in this component is the obstruction of a flat
fine-multigraded deformation over $k[t_i:i\in B]/(t)^r$ to lifting
over $k[t_i:i\in B]/(t)^{r+1}$. More precisely, its obstruction is
$\xi\otimes P_B$, with every other total-degree-$r$ parameter
coefficient zero. This specified jet has no lift even in the ungraded
category. Its $r$ first-order classes are nonzero; every proper
subcollection integrates over its polynomial parameter ring.
The symmetric statement holds for $B\subseteq H$.

Consequently every nonzero homogeneous class of the purely negative
quotient $T^2(R)/\Omega$ is realized in this way, although all primary
cup products among the selected tangent classes vanish.
\end{theorem}

\begin{proof}
The weighted telescoping identity is
\begin{equation}\label{eq:weighted-cyclic-identity}
x_jF_{j,h}+u_{j-1,h}t_{j-1}F_{j-1,h}=Q+\sigma v_hP_B.
\end{equation}
Over $\ZZ[u,t]/(P_B)$, the generators in
\eqref{eq:weighted-cyclic-family} are a monic Gr\"obner basis for
any monomial order in $x,y$ refining total degree. Same-$h$ pairs
of modified generators subtract two identities
\eqref{eq:weighted-cyclic-identity}; the residual terms cancel.
A pair of $X$ and a modified generator has residual
$-\sigma v_hP_BZW_h=0$. A modified and an unchanged cross generator
with the same $h$ use the same identity and have residual
$\sigma v_hP_BZW_h=0$. In each case the predecessor terms have
smaller $x,y$ degree. For different $h$, every tail is divisible by
$Y$, as are the tails in pairs involving $Y$. Pairs of unchanged
monomial generators have zero S-polynomial. These exhaust the pairs,
so the original standard monomials give a free coefficient-module basis.
In particular the family is flat over $k[t]/(P_B)$, and setting any
parameter to zero gives a flat polynomial family.

We identify its actual obstruction in a fixed presentation. Put
$S=k[x,y]$, let $K$ be the relations among the minimal generators of
$I=I_{M_{m,n}}$, and let $K_0$ be their Koszul submodule. The
polynomial-quotient presentation gives
\begin{equation}\label{eq:realization-T2-presentation}
T^2(R)=\operatorname{coker}\left(
\operatorname{Hom}_S(S^{2+mn},R)\longrightarrow
\operatorname{Hom}_R(K/K_0,R)\right).
\end{equation}
The Taylor pair relations generate $K$, their triple relations give
the compatibility equations, and Koszul relations impose the additional
annihilation conditions. This low-degree cotangent-complex presentation
is valid for polynomial quotients over any field; no characteristic-zero
comparison or division by an integer is used. We use the monomial
coordinates of \cite[Section~2, Lemma~2]{AltmannChristophersen}.

Orient a Taylor relation as
$R_{p,q}=\operatorname{lcm}(p,q)e_p/p-
\operatorname{lcm}(p,q)e_q/q$. Put $a_j=1$ for $j\in B$ and zero
otherwise. In degree $-\one_B$, its residual coordinate is the
coefficient of $\operatorname{lcm}(p,q)/Q$. Define
\begin{align}\label{eq:weighted-obstruction-cochain}
\mu_v(X,C_{j,h})&=-\sigma a_jv_h,\notag\\
\mu_v(C_{j,h},C_{k,h})&=\sigma(a_j-a_k)v_h,
\end{align}
and set all other pair coordinates to zero, with antisymmetry.
Pairs with a negative exponent or nonface shifted support have zero
target and are omitted.

These are coordinates of a cochain in the fixed space
$\operatorname{Hom}_R(K/K_0,R)_{-\one_B}$. Indeed the only nonzero
pair targets are within the blocks $\{X,C_{j,h}:j\in L\}$ for fixed
$h$; a pair involving $Y$ or two distinct $h$ has shifted support
containing $Y$. In each block \eqref{eq:weighted-obstruction-cochain}
is the difference of vertex values $0$ at $X$ and $\sigma a_jv_h$
at $C_{j,h}$, so it satisfies every Taylor triangle equation.
A Koszul image on a pair $p,q$ has coefficient monomial $pq/Q$.
For $(X,C_{j,h})$ with $j\in B$, its support contains $C_{j,h}$.
For same-$h$ cross pairs with exactly one index in $B$, it likewise
contains the circuit indexed by that element of $B$. Thus the
nonzero coordinates annihilate the Koszul relations. The remaining
coordinates are zero.

To see that this cochain is the actual lifting obstruction, set
$A_{r-1}=k[t]/(t)^r$ and $A_r=k[t]/(t)^{r+1}$, and write
$d_{j,h}=u_{j-1,h}t_{j-1}$ for $j\in B$, zero otherwise. Lift the
original Taylor relations to
\begin{align*}
\widetilde R_{X,C_{j,h}}&=W_he_X-x_je_{j,h}-d_{j,h}e_{j-1,h},\\
\widetilde R_{C_{j,h},C_{k,h}}&=x_je_{j,h}-x_ke_{k,h}
                   +d_{j,h}e_{j-1,h}-d_{k,h}e_{k-1,h}.
\end{align*}
A predecessor term is included only when its index lies in $B$.
Evaluating at the deformed generators gives respectively
$-\sigma a_jv_hP_BZW_h$ and
$\sigma(a_j-a_k)v_hP_BZW_h$.
For different-$h$ pairs and pairs involving $Y$, the evaluated
Taylor relation is a multiple of $Y$ with coefficient specializing
to zero; subtract that coefficient times $e_Y$ to obtain an exact
relation. The $X,Y$ relation is already exact.

These lifted relations generate the relation module over $A_{r-1}$:
flatness of the quotient implies flatness of its ideal and relation
module in the free presentations, base change recovers $K$, and
nilpotent Nakayama applies to the lifted Taylor generators. The
small-extension obstruction recipe evaluates these relations over
$A_r$ and projects their failures to
$R\otimes_k((t)^r/(t)^{r+1})$. By the displayed evaluations its
class is exactly $[\mu_v]\otimes P_B$, with all other parameter
coefficients zero. Changing generator lifts changes this cochain by
the generator coboundaries in \eqref{eq:realization-T2-presentation}.

We can also read the whole target from this presentation. For fixed
$h$, the triangle equations make any cochain a difference of vertex
values. If $j,k\in B$, the support of $C_{j,h}C_{k,h}/Q$ is
$(L\setminus\{j,k\})\cup(H\setminus\{h\})$, a face: its complement
$\{j,k,h\}$ is a basis of $N_{m,n}$. The Koszul condition therefore
equates the vertex values at all $j\in B$. Generator cochains at
$C_{k,h}$ for $k\notin B$, with values scalar multiples of
$(Z/x_k)W_h$, remove the outside-$B$ vertex values independently.
The generator cochain at $X$, with value a scalar multiple of $Z$,
removes one common constant among the remaining $h$-values.
No other generator has a value in this degree. Hence the entire
quotient is $k^H/k\one$, parametrized linearly by $[\mu_v]$.
Explicitly, $\mu_c$ is the coboundary with values $-\sigma cZ$ at
$X$, $-\sigma c(Z/x_k)W_h$ at outside-$B$ cross generators, and
zero elsewhere; a coboundary on the remaining pairs forces all
$v_h$ equal. Thus $\theta_B([v])=[\mu_v]$ is an isomorphism.

Given $\xi\ne0$, choose $[v]=\theta_B^{-1}(\xi)$ and a representative
$v$. Set $u_{i_*,h}=v_h$ for one fixed $i_*\in B$ and set the other
weights to one. A nonconstant vector has a nonzero coordinate.
Every tangent class is nonzero: its value on $X$ is zero, whereas
its value at a successor cross generator with nonzero weight is
$-u_{i,h}XW_h/(x_ix_{i+1})\ne0$. The only polynomial derivation
in degree $-\ee_i$ is a scalar multiple of $\partial/\partial x_i$,
whose nonzero value on $X$ is a scalar multiple of $X/x_i$.
Thus the tangent cannot be a derivation coboundary. Polynomial
integration of every one- and two-parameter restriction proves the
vanishing of the diagonal and mixed primary obstructions in every
characteristic.

Since the obstruction class is nonzero, no lift of this specified
jet exists, regardless of grading. The direct socle argument is also
instructive: if $g_X,g_{j,h}$ are arbitrary coefficients of $P_B$ in
order-$r$ corrections, then for $j\in B$ flatness forces
\[
\sigma v_hZW_h+x_jg_{j,h}-W_hg_X=0.
\]
Projection to fine degree $\deg(ZW_h)$ kills the $x_j$ term and
requires the coefficient of $Z$ in $g_X$ to equal $\sigma v_h$
for every $h$, impossible for nonconstant $v$. This projection
does not require the hypothetical corrections to be homogeneous.
Finally Proposition~\ref{prop:two-line-primary-quotient} and
Theorem~\ref{thm:negative-primary-image} show that the supports just
treated, and their symmetric counterparts, are exactly the nonzero
degrees of the purely negative quotient.
\end{proof}

\begin{corollary}[Graded uniqueness of the realizing jets]\label{cor:realizing-jet-uniqueness}
For the parameters and first-order classes chosen in
Theorem~\ref{thm:negative-realization}, the jet over $k[t]/(t)^r$
is unique up to fine-graded isomorphism inducing the identity on the
special fiber. This does not assert uniqueness among ungraded jets
with the same first-order classes.
\end{corollary}

\begin{proof}
For every parameter weight $\mathbf u$ with $|\mathbf u|\ge2$,
$T^1_{-\mathbf u}(R)=0$: nonsquarefree negative parts vanish by the
support criterion, and squarefree supports of size at least two
have rank at least two in the simple dual, so
Theorem~\ref{thm:rank-two-support} applies. The polynomial
derivations of this degree vanish as well. Align first-order
representatives by graded coordinate changes and apply the
order-filtration induction in the proof of
Theorem~\ref{thm:cyclic-higher-obstruction}; the differences at
each subsequent parameter weight lie in the zero space
$\operatorname{Hom}_S(I,R)_{-\mathbf u}$. This proves the stated
equivalence of the marked jets.
\end{proof}

The realizations are separate for each homogeneous degree. The theorem
does not realize arbitrary sums from different degrees in one
deformation, nor describe the full ungraded miniversal base.

\subsection{Realization in arbitrary homogeneous degrees}
\label{sec:positive-realization}

We retain the two-line matroid $M_{m,n}=N_{m,n}^*$ and its circuit generators
$X,Y,C_{j,h}$ from Section~\ref{sec:higher-obstructions}.
We now allow positive coordinates in the obstruction degree. The link
isomorphism \eqref{eq:link-formula}, originating in
\cite[Proposition~11]{AltmannChristophersen}, computes the component but
does not itself construct a deformation of the original ring. Multiplication
on the obstruction module has a simple description for the two-line family:
it discards precisely the opposite-line coordinates in the positive support.
The same operation on the cyclic construction realizes every surviving class.
The resulting families are formal; their flatness will be proved without a
polynomial monomial order.

\begin{lemma}[Lifting relations over an Artinian base]
\label{lem:artinian-relation-lifting}
Let $A$ be a local Artinian $k$-algebra with residue field $k$, and put
$S=A[z_1,\ldots,z_d]$. Suppose $F_1,\ldots,F_q\in S$ reduce to generators
$p_1,\ldots,p_q$ of $I\subseteq k[z]$. If generators of the relation module
among the $p_i$ lift to exact relations among the $F_i$, then
$S/(F_1,\ldots,F_q)$ is free over $A$ on arbitrary lifts of a $k$-basis of
$k[z]/I$.
\end{lemma}

\begin{proof}
Write $\mathfrak m$ for the maximal ideal, $J=(F_i)$, and
$K_A=\ker(S^q\to J)$. Right exactness gives
$J/\mathfrak mJ=k[z]^q/\operatorname{im}(K_A\to k[z]^q)$.
The lifted relations imply that this image is the kernel of $k[z]^q\to I$.
Thus $J/\mathfrak mJ\to S/\mathfrak mS$ is injective, and the Tor sequence
gives $\operatorname{Tor}_1^A(S/J,k)=0$ since $S$ is $A$-free.
Let $E$ be the free $A$-module on chosen lifts of a residue-field basis of
$S/J$. The map $E\to S/J$ is surjective: its cokernel equals its product
with $\mathfrak m$, and $\mathfrak m$ is nilpotent. Its kernel $K$ lies in
$\mathfrak mE$ because reduction is an isomorphism. The Tor sequence now
identifies $K/\mathfrak mK$ with $\operatorname{Tor}_1^A(S/J,k)$, so
$K=\mathfrak mK=0$. No finite-dimensionality of the residue-field basis is
required.
\end{proof}

\begin{theorem}[Realization of the homogeneous primary quotient]
\label{thm:homogeneous-realization}
Let $R=k[M_{m,n}]$, $m,n\ge3$, over any field, and let
$\Omega=\operatorname{im}(T^1(R)\otimes_R T^1(R)\to T^2(R))$.
For every fine degree $c$ with $(T^2(R)/\Omega)_c\ne0$, write
$c=a-\one_B$, $a\ge0$, with disjoint supports. Then $B$ is contained
in one of the two dual lines, $r=|B|\ge3$, and $\Omega_c=0$, so the
quotient map canonically identifies this component with $T^2_c(R)$.
For every nonzero $\xi$ in it, there is a marked fine-graded flat deformation over
$k[t_i:i\in B]/(t)^r$ whose obstruction to lifting over
$k[t_i:i\in B]/(t)^{r+1}$ is $\xi\otimes[P_B]$ in
$T^2(R)\otimes_k J_r$, where $P_B=\prod_{i\in B}t_i$ and
$J_r=(t)^r/(t)^{r+1}$.
All other total-degree-$r$ obstruction coefficients are zero. This specified
jet has no lift even if grading is forgotten. Its $r$ tangent classes are
nonzero, and every proper subcollection integrates formally. All primary
products among these tangent classes vanish. A cyclic ordering of $B$ and
a distinguished $i_*\in B$ may be chosen arbitrarily; the parameter degrees
are $\deg t_i=\ee_i$ for $i\ne i_*$ and
$\deg t_{i_*}=\ee_{i_*}-a$. In particular $\deg P_B=-c$.

The realizing jet is unique up to marked fine-graded isomorphism for the
chosen parameter degrees and tangent classes. Here isomorphisms are over
the fixed parameter algebra and induce the identity on the special fiber.
A proper subcollection is obtained by setting all other parameters to zero.
The assertion is separate for
each homogeneous degree; it does not assert simultaneous realization of
inhomogeneous sums or describe the full ungraded miniversal base.
\end{theorem}

\begin{proof}
Write $c=a-\one_B$ with $a\ge0$ and $\operatorname{supp}(a)\cap B=\varnothing$.
We first identify the obstruction component and the multiplication map
from its purely negative degree. We then insert the positive monomial
weight into the cyclic family and lift its defining relations to prove
flatness. The residual relation evaluations identify the prescribed
obstruction and rule out arbitrary ungraded lifts. Finally, vanishing of
the relevant correction spaces gives uniqueness in the marked
fine-graded category.
Theorems~\ref{thm:rank-two-support} and \ref{thm:full-primary-image}, with
the actual-vertex convention of \eqref{eq:link-formula}, imply that a nonzero quotient
component has $B$ contained in one dual line, with $r=|B|\ge3$.
Suppose $B\subseteq L$, write $A=\operatorname{supp}(a)$, and put $H_0=H\setminus A$.
If $|H_0|\ge2$, the dual of the actual-vertex link is the deletion $N\setminus A$.
It still has rank three, has no coloops, and consists of the lines
$L\setminus A,H_0$. The only contributing hyperplane is $H_0$, with its
singleton flats as vertices. Hence $T^2_c(R)=k^{H_0}/k\one$.
If $|H_0|=1$, its unique point is a coloop of $N\setminus A$; removing that
ghost leaves a rank-two line, whose blocks have at most one vertex and
contribute zero. If $H_0=\varnothing$, $A$ is a nonface and the component
is zero. Thus a nonzero target necessarily has $|H_0|\ge2$.

We need the multiplication map, not merely these dimensions. In the Taylor
presentation on the minimal monomial generating set, modulo Koszul
relations, use \eqref{eq:realization-T2-presentation} and orient
$R_{p,q}=\operatorname{lcm}(p,q)e_p/p-\operatorname{lcm}(p,q)e_q/q$.
Set $b_j=1$ for $j\in B$, zero otherwise, and $\sigma=(-1)^{r-1}$.
The negative cochain representing $v\in k^H$ has coordinates
\[
\mu_v(X,C_{j,h})=-\sigma b_jv_h,\qquad
\mu_v(C_{j,h},C_{k,h})=\sigma(b_j-b_k)v_h,
\]
and zero coordinates elsewhere, as in \eqref{eq:weighted-obstruction-cochain}.
Each scalar multiplies
$\operatorname{lcm}(p,q)/\prod_{i\in B}x_i$.
Multiplication by $x^a$ kills exactly the blocks $h\in A$, because their
monomials become divisible by $Y$. For $h\in H_0$ every same-block Taylor
pair target survives. Its complement contains $B\cup\{h\}$, which spans
$N$, while adding $A$ cannot remove those elements. Taylor triangles make
each surviving block a difference of vertex values; the Koszul conditions
equate all vertices indexed by $B$ since
$(L\setminus\{j,k\})\cup(H\setminus\{h\})$ remains a face after adding
$A$ for distinct $j,k\in B$. Outside-$B$ generator cochains remove their
vertices, and the cochain on $X$ removes one common constant. More explicitly
the vertices of each surviving block are $\{X,C_{j,h}:j\in L\}$; all their
pair and triple targets are nonzero. The cochain values at outside-$B$
generators are scalar multiples of $x^a(Z/x_k)W_h$ for $k\notin B$,
and the value at $X$ is a scalar multiple of $x^aZ$.
Their complements retain $B\cup\{h\}$ and $B\cup H_0$, respectively,
so these are nonzero values. There is no value at $Y$ or a $B$-indexed
cross generator because one negative $B$ coordinate remains. These are
precisely the same operations as in the negative presentation, multiplied
by $x^a$. Consequently the actual multiplication map is
\begin{equation}\label{eq:positive-restriction-map}
k^H/k\one\longrightarrow k^{H_0}/k\one,
\qquad [v]\longmapsto[v|_{H_0}].
\end{equation}
It is surjective. Since $\Omega_c=0$, choose $v$ with
$[x^a\mu_v]=\xi$; its restriction to $H_0$ is nonconstant.

Fix $i_*\in B$, choose a cyclic ordering, and interpret predecessors and
successors below in that ordering, not in the ground-set labels.
In \eqref{eq:weighted-cyclic-polynomials} set
$u_{i_*,h}=v_hx^a$, $u_{i,h}=1$ for $i\ne i_*$.
Put $Q=\prod_{i\in B}x_i$, $Z=\prod_{i\in L\setminus B}x_i$,
$P_B=\prod_{i\in B}t_i$, and explicitly put
\[
F_{j,h}=\sum_{q=0}^{r-1}(-1)^q
 \left(\prod_{b=1}^q u_{j-b,h}t_{j-b}\right)
 \left(\prod_{i\in B\setminus\{j,j-1,\ldots,j-q\}}x_i\right).
\]
Empty products are one. The constant term is $Q/x_j$ and the
$t_i$-linear term at the successor of $i$ is $-u_{i,h}Q/(x_ix_{i+1})$.
The earlier telescoping identity is a polynomial identity in arbitrary
commuting weights, so substitution of $v_hx^a$ is legitimate. It becomes
\begin{equation}\label{eq:positive-cyclic-identity}
x_jF_{j,h}+u_{j-1,h}t_{j-1}F_{j-1,h}
       =Q+\sigma v_hx^aP_B.
\end{equation}
Use generators $X,Y$, $D_{j,h}=ZF_{j,h}W_h$ for $j\in B$, and
$D_{j,h}=C_{j,h}$ for $j\notin B$. Give $t_i$ degree $\ee_i$ for
$i\ne i_*$ and give $t_{i_*}$ degree $\ee_{i_*}-a$.
Each deformed generator has the fine degree of its original monomial.

Write $k[t]=k[t_i:i\in B]$ and $(t)=(t_i:i\in B)$.
For every $\nu\ge1$, work over $A_\nu=k[t]/((t)^\nu,P_B)$.
All original Taylor pair relations lift exactly, as follows. Put
$d_{j,h}=u_{j-1,h}t_{j-1}$ for $j\in B$; omit predecessor terms entirely
when $j\notin B$.
The $X,D_{j,h}$ relation lifts to
$W_he_X-x_je_{j,h}-d_{j,h}e_{j-1,h}$; evaluation is
$-\sigma b_jv_hx^aP_BZW_h=0$.
For two cross generators in the same block, subtract the corresponding
identities: the lift is
$x_je_{j,h}-x_ke_{k,h}+d_{j,h}e_{j-1,h}-d_{k,h}e_{k-1,h}$
and its evaluation is
$\sigma(b_j-b_k)v_hx^aP_BZW_h=0$.
For a pair involving $Y$, $y_hD_{j,h}$ is divisible by $Y$;
replace the coefficient on $e_Y$ by $y_hD_{j,h}/Y$.
For cross generators in distinct blocks, the original Taylor multipliers
contain the missing $y_h$ and $y_l$. Both evaluated terms are therefore
divisible by $Y$, and their difference divided by $Y$ has zero parameter
constant term. Subtract this quotient times $e_Y$.
The $X,Y$ relation and pairs of unchanged monomials are exact already.
These lifts reduce to every original Taylor pair, which generate all first
syzygies. Lemma~\ref{lem:artinian-relation-lifting} proves $A_\nu$-flatness.
Their formulas are compatible for all $\nu$. The inverse system defines a
formal family over $k[[t]]/(P_B)$; setting any parameter to zero gives a formally
integrated proper subcollection. No polynomial flatness is inferred here.

Over the small extension $k[t]/(t)^{r+1}\to k[t]/(t)^r$, the evaluations
just computed give $[x^a\mu_v]\otimes P_B$. All other pair evaluations
are zero after the exact $Y$ corrections. Over the lower-order jet, flatness
of the quotient makes its ideal flat, and then the relation module in the
free presentation is flat. Reduction is exact and recovers the original
Taylor relations; nilpotent Nakayama shows that their lifts generate this
module. The fixed
small-extension presentation therefore gives precisely $\xi\otimes P_B$
and no other order-$r$ coefficient.

Each ordinary parameter direction has the nonzero tangent class calculated
in Theorem~\ref{thm:negative-realization}.
In the special direction it is multiplied by $x^a$, with weights $v_h$.
Choose $h\in H_0$ with $v_h\ne0$. At the successor cross generator its
value is $-v_hx^aXW_h/(x_{i_*}x_{i_*+1})\ne0$, since the complement of
its support contains the basis $\{i_*,i_*+1,h\}$ of $N$.
Its value on $X$ is zero. The only polynomial derivation of that degree
is a scalar multiple of $x^a\partial/\partial x_{i_*}$, whose value on
$X$ is nonzero: the complement of its support contains $\{i_*\}\cup H_0$,
of rank three. Thus the special tangent is not a derivation coboundary.
Formal one- and two-parameter integration kills the quadratic obstructions.
Their polarization is the cup product by Lemma~\ref{lem:graded-polarization},
also in characteristic two; self-cups vanish by the doubled fine degree.

Since the obstruction class is nonzero, the specified jet has no ungraded
lift. Equivalently, let $g_X,g_{j,h}$ be arbitrary $P_B$ coefficients of
prospective order-$r$ generator corrections. Order-$r$ changes of relation
coefficients evaluate to zero in $R$, and positive-order coefficients times
order-$r$ generator corrections have order at least $r+1$.
Thus arbitrary relation corrections cannot evade the following equation.
Projection of the lifted
$X,D_{j,h}$ relation to $R\otimes P_B$ forces
$\sigma v_hx^aZW_h+x_jg_{j,h}-W_hg_X=0$ for $j\in B$, $h\in H_0$.
The coefficient of $x^aZW_h$ in the $x_j$ term is zero because
$a_j=0$ and $j\notin\operatorname{supp}(ZW_h)$.
Hence the coefficient of $x^aZ$ in $g_X$ must equal $\sigma v_h$ for all
$h\in H_0$, contrary to nonconstancy.

Finally a correction at parameter monomial $t^{\mathbf u}$, $|\mathbf u|\ge2$,
has ring degree $-\sum_i u_i\ee_i+u_{i_*}a$. Its negative part is either
nonsquarefree or meets at least two distinct parallel classes of the
simple dual: its entries on $B$ are exactly $-u_i$, since $a$ vanishes
there. The support criterion in Section~\ref{sec:preliminaries} and
Theorem~\ref{thm:rank-two-support} give $T^1=0$ in these degrees.
A polynomial derivation has no coefficient monomial in such a degree,
because adding one $\ee_j$ cannot make all entries nonnegative.
The right-exact sequence
\[
\operatorname{Der}_k(k[x,y],R)\longrightarrow
\operatorname{Hom}_{k[x,y]}(I,R)\longrightarrow T^1(R)\longrightarrow0
\]
therefore gives a zero correction space.
Aligning first-order representatives by marked graded coordinate changes
and inducting through parameter orders $2,\ldots,r-1$ gives marked fine-graded uniqueness,
without asserting ungraded uniqueness. The case $B\subseteq H$ follows
by interchanging the two lines.
\end{proof}

\subsection{The module left after primary cup products}
\label{sec:primary-quotient-module}

Put $R=k[M_{m,n}]$ and $\mathcal Q=T^2(R)/\Omega$ as in
Theorem~\ref{thm:homogeneous-realization}.

The dimension of the purely negative part in
\eqref{eq:two-line-primary-quotient} grows exponentially, but this does not measure the number
of generators of the obstruction module. For the same two-line family,
all its components and multiplication maps fit an explicit presentation.
The additional information is multiplication between negative degrees,
not another dimension formula.
Put $S_L=k[x_i:i\in L]$ and $S_H=k[y_h:h\in H]$. Set
$S=S_L\otimes_k S_H=k[x_i,y_h]$ and $I=I_{M_{m,n}}$, so $R=S/I$.
Define
\[
J_L=(x_D: D\subseteq L,\ |D|=m-2),\qquad
K_L=(x_D: D\subseteq L,\ |D|=m-1),
\]
where $x_D=\prod_{i\in D}x_i$; define $J_H,K_H$ symmetrically.
Set
\begin{equation}\label{eq:coordinate-module}
U_H=\frac{\bigoplus_{h\in H} S_H e_h}
 {S_H\left(\sum_{h\in H}e_h\right)+\sum_{h\in H}S_Hy_he_h},
\end{equation}
with $\deg e_h=0$, and define $U_L$ in the same way.
Tensor products in the following theorem are over $k$.
For a graded module $V$, the notation $V\langle d\rangle$ means that every
degree is increased by $d$; in particular its displayed generators of
degree zero acquire degree $d$.

\begin{theorem}[Generators, relations, and support of the primary quotient]
\label{thm:primary-quotient-module}
For $m,n\ge3$ and every field $k$, there is a fine-graded $R$-module
isomorphism
\begin{equation}\label{eq:primary-quotient-module}
\mathcal Q\simeq
\left((S_L/J_L)\otimes U_H\right)\langle-\mathbf1_L\rangle
\ \oplus\
\left(U_L\otimes(S_H/J_H)\right)\langle-\mathbf1_H\rangle.
\end{equation}
The indicated $S$-modules are annihilated by $I$, so their $R$-module
structures are unambiguous. The quotient has exactly $m+n-2$ minimal
homogeneous generators: $n-1$ in degree $-\mathbf1_L$ and $m-1$ in degree
$-\mathbf1_H$. Its annihilator in $S$ is
\begin{equation}\label{eq:primary-quotient-annihilator}
\operatorname{Ann}_S \mathcal Q=(J_LS+K_HS)\cap(K_LS+J_HS).
\end{equation}
Consequently $\operatorname{Ann}_R \mathcal Q$ is this ideal modulo $I$.
The annihilator is radical. Every irreducible component of
$\operatorname{Supp}_R \mathcal Q$ has dimension $m+n-5$; in particular the support
of the primary quotient has codimension two in $\operatorname{Spec}R$.
\end{theorem}

\begin{proof}
Use the Taylor presentation \eqref{eq:realization-T2-presentation},
with $R_{p,q}=\operatorname{lcm}(p,q)e_p/p-
\operatorname{lcm}(p,q)e_q/q$. Its coboundary is
$\delta f(R_{p,q})=(\operatorname{lcm}(p,q)/p)f(p)
 -(\operatorname{lcm}(p,q)/q)f(q)$.
Normalize the cocycle \eqref{eq:weighted-obstruction-cochain} by removing
its factor $(-1)^{|B|-1}$. The resulting $\nu_B(v)$ has scalars
\[
\nu_B(v)(X,C_{j,h})=-b_jv_h,\qquad
\nu_B(v)(C_{j,h},C_{k,h})=(b_j-b_k)v_h,
\quad b_j=\mathbf1_{j\in B},
\]
and zero elsewhere; each scalar multiplies $\operatorname{lcm}(p,q)/x_B$.

We first determine multiplication when a negative coordinate is removed;
this is the point not supplied by the component dimensions. Suppose
$B\subseteq L$, $|B|\ge3$, $i\in B$, and $B'=B\setminus\{i\}$.
In the fixed Taylor presentation let $f_i(v)$ be the generator cochain
of degree $-\mathbf1_{B'}$ given by
\[
f_i(v)(C_{i,h})=v_h C_{i,h}/x_{B'}\quad(h\in H),
\qquad f_i(v)(X)=f_i(v)(Y)=f_i(v)(C_{j,h})=0\quad(j\ne i).
\]
Its nonzero monomials are defined, since $i\notin B'$; their complements
are $B'\cup\{i,h\}$ and span $N$, so they survive in $R$.
For $|B'|\ge3$ the cocycle formulas give the identity
\begin{equation}\label{eq:negative-coordinate-boundary}
x_i\nu_B(v)-\nu_{B'}(v)=\delta f_i(v).
\end{equation}
Indeed, measured relative to $\operatorname{lcm}(X,C_{i,h})/x_{B'}$,
the change at $(X,C_{i,h})$ is $-v_h$, precisely the scalar
of the boundary $-x_if_i(v)(C_{i,h})$. On a same-$h$ pair it is the
difference of the two changed vertex scalars. Every other pair involving
one of the selected generators has its target divisible by $Y$ and hence
zero in $R$: this includes pairs with $Y$ and cross pairs in distinct
$h$ blocks. All remaining values on both sides are zero. Thus the actual
map on the coordinate quotient is $[v]\mapsto[v]$.
When $|B'|=2$, the target has zero primary quotient, so the map into $\mathcal Q$
is zero without any comparison of cocycle normalizations in that degree.
Multiplying these identities by a positive monomial preserves them,
including when some generator values vanish. Multiplication in a
same-line positive coordinate leaves the remaining coordinate quotient
unchanged. An opposite-line coordinate $y_h$ restricts the quotient on
$H_0$ to that on $H_0\setminus\{h\}$. If $h$ is already absent it leaves
the coordinates unchanged; if at most one remains, the target is zero.
Subsequent powers therefore act identically on the surviving coordinates. These
last assertions follow from the actual restriction map
\eqref{eq:positive-restriction-map}.

Send $e_h$ in the first summand of \eqref{eq:primary-quotient-module}
to $[\nu_L(\mathbf e_h)]\in\mathcal Q_{-\mathbf1_L}$.
The common-constant relation is zero, and $y_he_h$ maps to zero by
coordinate restriction. A generator $x_D$ of $J_L$ removes $m-2$
negative coordinates, leaving two; hence it maps every $e_h$ to zero.
This defines an $S$-linear map from the presented module, before any
dimension comparison. The symmetric construction defines the second map.

For $b\ge0$, the degree-$b$ component of $U_H$ is the coordinate space
on $H\setminus\operatorname{supp}(b)$ modulo one common constant, with
dimension $(|H\setminus\operatorname{supp}(b)|-1)_+$.
A degree $c=a-\mathbf1_B$ in the first shifted summand survives the
$J_L$ relation exactly when $B\subseteq L$ and $|B|\ge3$.
Indeed, for unshifted $L$-degree $u\ge0$, $B=\{i:u_i=0\}$ and
$x^u\notin J_L$ is equivalent to $|\operatorname{supp}(u)|\le m-3$.
The $U_H$ factor additionally requires $|H_0|\ge2$.
The actual maps just proved show that its image is all of $\mathcal Q_c$.
The component calculation in Theorem~\ref{thm:homogeneous-realization}
gives equality of dimensions and therefore
injectivity. The nonzero fine degrees of the two summands are disjoint,
so their sum is the required $S$-module isomorphism.

The ideal $I$ annihilates both presented summands directly. In the first,
$X$ and every $C_{j,h}$ have $L$ support of size at least $m-2$ and belong
to $J_LS$, while $Y$ kills $U_H$. In the second, $Y$ and every cross
generator belong to $J_HS$, while $X$ kills $U_L$.
Thus the isomorphism descends to $R$.
Modulo the homogeneous maximal ideal, $U_H$ has exactly its $n$ degree-zero
generators and the single relation $\sum e_h=0$. The first summand has
$n-1$ generators in its lowest degree $-\mathbf1_L$; the second contributes
$m-1$ in degree $-\mathbf1_H$. Graded Nakayama proves the stated minimality.

The annihilator of $U_H$ is fine homogeneous, hence monomial. A monomial
with support $D$ maps the degree-zero coordinate quotient onto the
quotient on $H\setminus D$. Its action is zero exactly when $|D|\ge n-1$;
the degree-zero generators generate the module. Therefore
$\operatorname{Ann}_{S_H}U_H=K_H$.
The same support test shows that the annihilator of the first tensor
summand is $J_LS+K_HS$: a monomial outside this ideal acts nontrivially
on a suitable degree-zero generator. The symmetric formula for the other
summand yields \eqref{eq:primary-quotient-annihilator}.
Both ideals in the intersection are generated by squarefree monomials
and are radical, as is their intersection. The maximal coordinate faces
of $V(J_LS+K_HS)$ have $m-3$ vertices in $L$ and $n-2$ in $H$; those of
$V(K_LS+J_HS)$ have $m-2$ in $L$ and $n-3$ in $H$. No maximal face of either
contains one from the other, and each has $m+n-5$ vertices. The support
is thus pure of that dimension. Finally $R$ is equidimensional of
dimension $\rk M=m+n-3$, proving the codimension assertion.
\end{proof}

If $m=3$, then $S_L/J_L=k$, and similarly for $n=3$.
The argument uses no division by $m,n$, or $2$, and includes characteristic two.
This presentation concerns the primary quotient, not the full $T^2$
module or a miniversal base. Its total $k$-dimension is infinite in this
family; the exponential dimension previously computed concerns only
the purely negative part. The module presentation and the homogeneous
realization theorem answer different questions: the former describes
how classes multiply, while the latter realizes each class separately.

\section{Finite certificates and computational checks}\label{sec:verification}

The preceding proofs give a full component calculation, a vanishing
classification, and actual obstructions of unbounded order. None of these
universal assertions is inferred from a finite sample. This section gives
the finite calculations for the Tutte-equivalent
pair, the compatibility table in the rank-three extremal proof, and the
small-ground-set minima. It also compares the graph formula with a
direct calculation from generators and relations. The general graph,
vanishing, and cup-product theorems are proved in
Sections~\ref{sec:graph-model}--\ref{sec:cup}.

The supplementary file \texttt{anc/README.md} identifies each checker,
its invocation and expected output fields. The accompanying
\texttt{anc/finite-certificates.json} records the earlier exact table,
coverage, polynomial and small-ground-set results. These recorded
certificates are distinct from a fresh execution of a checker.

\subsection{Checks of the cyclic construction and primary image}

The accompanying \nolinkurl{rank-two-support.py} compares
Theorem~\ref{thm:rank-two-support} with $1{,}119$ graph calculations,
$350$ direct presentation calculations over $\mathbb F_2$ and
$\mathbb F_{32003}$, and $363$ positive-support calculations.
The separate \nolinkurl{cup-image.py} checks $148$ presentation-level cup ranks,
including $30$ nonzero targets, against
Theorem~\ref{thm:negative-primary-image}. For the two-line matroids it
recovers the purely negative quotient dimensions $4,13,30,68$ in the
tested cases.

The all-degree image theorem is checked by \nolinkurl{full-primary-image.py}
using $2{,}574$ multiplication
maps on the minimal Taylor presentation, both before and after taking
the derivation cokernel, and $2{,}703$ cup-rank calculations with
squarefree positive support. The fields are
$\mathbb F_2,\mathbb F_3,\mathbb F_{32003}$. The cases include repeated
dual parallel classes, loops, and links with vertices removed after
contraction. All mathematical output fields agree on isolated replay.
These tests address the maps in
Lemma~\ref{lem:negative-generation} and
Theorem~\ref{thm:full-primary-image}; they do not replace their proofs.

The program \nolinkurl{arbitrary-order.py} uses integer polynomial division
to check the cyclic Gr\"obner construction and
its telescoping identity for $3\le m\le12$ and $n=3,4$, comprising
$17{,}740$ S-pair tests. The separate \nolinkurl{essential-cubic.py} checks
$220$ S-pairs and the contradictory socle relations for the cubic case.
Independent presentation calculations in the smaller cases verify the
critical $T^1$ vanishings and $T^2$ dimensions in two characteristics.
The four programs \nolinkurl{rank-two-support.py},
\nolinkurl{cup-image.py}, \nolinkurl{arbitrary-order.py}, and
\nolinkurl{essential-cubic.py} were each rerun in an isolated environment; all
mathematical result fields agreed on replay. These finite checks test
the construction, not the assertions for arbitrary $m$ or arbitrary fields,
which follow from the proofs in
Sections~\ref{sec:primary-image} and \ref{sec:higher-obstructions}.

The program \nolinkurl{weighted-realization.py} checks the weighted subset construction of
Theorem~\ref{thm:negative-realization} separately in $32$
cases with $3\le m\le6$, $n=3,4$, and both consecutive and
nonconsecutive subsets of size at least three. The tests comprise
$12{,}636$ integer S-pairs, $9{,}475$ Taylor triangle conditions and
$6{,}450$ Koszul conditions. Independent presentation calculations
in characteristics $2,3,32003$ verify that the displayed obstruction
cochains span the entire target modulo generator coboundaries.
The deterministic replay agrees on every mathematical result field.
The relation-level identification and arbitrary-field realization
theorem are proved above, not inferred from these cases.

For Theorem~\ref{thm:homogeneous-realization}, the separate checker
\texttt{positive-realization.py} tests $60$ cases in the two-line family,
including $51$ with positive support and $22$ with zero targets or
contraction ghosts; these categories overlap. It includes positive exponents one and two and weights
supported on both lines. There are $6{,}764$ integral Taylor relation
checks, $4{,}314$ exact corrections by multiples of $Y$, $3{,}754$ Taylor
triangle checks, $125$ tangent nontriviality checks, and $492$ higher-degree
correction vanishings. Presentation calculations in characteristics
$2,3,32003$ verify the restriction map and its image dimension. The
separate replay agrees on all mathematical fields. The arbitrary-weight
flatness assertion uses the Artinian relation-lifting proof, not a
Gr\"obner basis test after substitution of positive monomial weights.

For Theorem~\ref{thm:primary-quotient-module},
\texttt{primary-quotient-module.py} checks the actual Taylor/Koszul
presentation for $(m,n)=(3,3),(4,3),(4,4),(5,3)$, with both line
orientations, positive exponents one and two, and fields $2,3,32003$.
It computes $3{,}204$ degree/field cases and $24{,}972$ variable
multiplication maps. These include $35{,}058$ signed negative-coordinate
boundary identities and $5{,}424$ repeated-positive maps; $4{,}224$
squarefree support tests check annihilation and $360$ checks verify the
action of minimal circuit generators. The size-two endpoints use actual
cup vectors in the presentation. Every mathematical field agrees with
an isolated replay. The universal module isomorphism and its annihilator
are established by the preceding proof, not by these finite tests.

\subsection{Subset-rank certificate for the graphic pair}

For each of the two edge lists in Theorem~\ref{thm:Tutte-contrast},
enumerate its $2^{13}$ subsets $X$ and compute
$r(X)=7-\kappa(V,X)$, counting isolated vertices. The two lists have the
same numbers
\[
a_{r,s}=\#\{X:r(X)=r,\ |X|=s\}.
\]
For a fixed $r$, the following rows give the coefficients in increasing
$s$, starting at $s=r$:
\[
\begin{array}{c|l}
r&(a_{r,r},a_{r,r+1},\ldots)\\ \hline
0&(1)\\
1&(13)\\
2&(78,7)\\
3&(279,83,12,1)\\
4&(632,413,145,31,3)\\
5&(862,1025,662,281,78,13,1)\\
6&(545,1023,1003,637,273,77,13,1).
\end{array}
\]
Consequently their Tutte polynomials agree, since
$T(x,y)=\sum_{r,s}a_{r,s}(x-1)^{6-r}(y-1)^{s-r}$.
For each subset with $r(X)=|X|=5$, count the original edges joining its
two components. This gives exactly the two profiles in the proof of
Theorem~\ref{thm:Tutte-contrast}, hence $\Theta=12135$ and $12161$.

For each edge list, retain the rank of every labeled subset before
aggregating into $a_{r,s}$. These separate rank dictionaries give the dual rank function
$r^*(X)=|X|-6+r(E\setminus X)$. A subset is a dual flat precisely when
adjoining each omitted edge increases $r^*$. Enumerating the flats of
ranks five and six and applying \eqref{eq:pair-dim} gives the two values
$\mathfrak q=2009$ and $2165$ in
Proposition~\ref{prop:valuativity-boundary}. All these calculations use
integer ranks and counts.

\subsection{Partition and line-system calculations}

The twelve rows of \eqref{eq:rank-three-finite-table} are obtained by
listing integer partitions and evaluating the recurrence in its proof;
no matroid realizability assumption is made for the relaxed minima.
The accompanying coverage table checks the remaining cases with
$8\le n\le12$. The coefficient certificates in
Lemma~\ref{lem:rank-three-polynomial-gap-new} are exact polynomial
expansions. Direct re-expansion and a separate partition implementation
reproduce all displayed entries.

The small-ground-set rank-three minima can also be reproduced without
a database of matroids. On $[n]$, list all subsets of sizes $3,\ldots,n-1$.
Select a family in which no unordered pair lies in two selected subsets,
and add every unused pair as a two-point line. This construction gives
exactly the line systems of labeled simple rank-three matroids: distinct
lines intersect in at most one point, and every pair belongs to one line.
To enumerate each family once, order the candidate subsets and recursively
adjoin only later candidates whose pairs are unused. Evaluate
\eqref{eq:rank-three-formula} at every node, including the empty family.
The resulting counts are
\[
\begin{array}{c|r|r|r}
n&\text{line systems}&\min\dim_kT^2(M)_0&
\text{labeled minimizers}\\ \hline
5&31&0&31\\
6&352&54&30\\
7&8389&357&30.
\end{array}
\]
This is an exhaustive finite calculation, not an argument for the
unbounded rank-three theorem.

As additional arithmetic checks, all connected rank-two class partitions
on $5\le n\le20$ reproduce both stated minima, their equality cases, and
the second-profile gap for $\mathfrak q$. The Fano calculation gives
$126+231=357$. These checks test the formulas and boundary cases; the
general rank-two assertions rest on the merge and transfer arguments.

\subsection{Calculation from generators and relations}

The Lichtenbaum--Schlessinger presentation gives a second way to compute
$T^1$ and $T^2$. Circuit monomials generate the Stanley--Reisner ideal;
Taylor pair relations, relations among triples of generators, and the
trivial Koszul relations determine the cochain spaces and their
quotients. Computing their dimensions by exact row reduction provides
a comparison with the hyperplane-graph formula that does not use the
relative-cohomology description.

The accompanying program \texttt{anc/verify\_independent.py} carries out
this comparison for uniform, rank-two, graphic, direct-sum, and Fano
matroids, mostly over $\mathbb F_{32003}$ and for selected uniform and
rank-two examples over $\mathbb F_2$. The positive-support cases include
two-element links of $U_{3,6}$ and both link strands of the Fano matroid.
All $917$ checks agree with the stated formulas. The program lists the
matroids, fine degrees, and coefficient fields in its \texttt{main()}
function. These calculations supplement the proofs in
Sections~\ref{sec:graph-model}--\ref{sec:cup}.

\subsection{AI-assisted research workflow}

The systems named in the declaration below were used in developing the
mathematical arguments as well as preparing the text. Their research tasks
included deriving the graph and primary cup-image formulas, constructing
the cyclic families and quotient presentations, and generating and revising
the ancillary exact-checking programs. Subsequent AI-assisted checks
examined proof dependencies, field and support hypotheses, and reference
attributions. The proofs are given in the preceding sections; the programs,
inputs, and finite certificates are described above and supplied with the
paper. The finite checks and AI-assisted proof checks are distinct from
human evaluation of the arguments.

\section{Further questions}\label{sec:questions}

The calculations above separate three aspects of deformation theory:
the obstruction module, the image of the primary cup product, and the
classes that occur as higher lifting obstructions. Flat counts determine
the fine components of the first, while the cyclic construction realizes
obstructions outside the second at arbitrarily high order. For the dual
two-line family, the realization theorem and quotient presentation connect
these aspects explicitly. The following questions ask how far that
connection extends.

\begin{enumerate}[label=\textup{(\arabic*)}]
\item \emph{The module structure of $T^2$.}
Theorem~\ref{thm:universal-graph} describes the fine components of
$T^2(M)$ as vector spaces. Can the multiplication maps between these
components be described directly in terms of the hyperplane graphs?
The parallel-class maps in \eqref{eq:parallel-reduction-maps} give part
of this structure in negative directions, and
Theorem~\ref{thm:primary-quotient-module} determines the quotient by the
primary image for the dual two-line family. For general matroids, can
one obtain presentations, minimal generators, and annihilators from the
dual flats? Even in the two-line family, the extension
$0\to\Omega\to T^2\to T^2/\Omega\to0$ retains information not specified
by the quotient presentation.

\item \emph{Relations among cup products.}
Theorem~\ref{thm:full-primary-image} determines the entire primary
image, and Corollary~\ref{cor:primary-quotient} gives its purely
negative generation. Can the kernel of the cup map, and relations
among its homogeneous generators, be described from the dual flats?
Such a description would refine the image calculation to explicit
quadratic equations for a deformation base.

\item \emph{Higher-order lifts.}
Theorem~\ref{thm:cyclic-higher-obstruction} gives essential obstructions
of every order and rules out all ungraded lifts of the specified jets.
Theorem~\ref{thm:homogeneous-realization} realizes every nonzero homogeneous
class of the complete primary quotient for the dual two-line family.
Can the same be done for other matroids, or for inhomogeneous sums of
classes from several degrees within a single deformation?
Can the first unavoidable order be determined from their first-order
classes without imposing the fine grading, so that different lower-order
choices cannot postpone it? For two classes, does cup-product vanishing
still suffice for joint integration when positive and negative supports
cancel? Even under the support hypotheses of Theorem~\ref{thm:arcs},
its formal arc for the sum does not by itself answer the independent-parameter question.

\item \emph{Higher cotangent cohomology.}
In characteristic zero, the link-localization theorem of Ilten,
Meazzini, and Petracci \cite{IltenMeazziniPetracci} applies to
$T^i(M)$ for $i\ge3$. Can nonvanishing be detected in a link whose
rank is bounded in terms of $i$? Is the class of matroids with
$T^i(M)=0$ minor-closed, and, if so, does it have finitely many excluded
minors?

\item \emph{Valuative invariants.}
The cocircuit profile of Section~\ref{sec:valuative-profile} is
valuative, whereas the pure-pair rank $\mathfrak q$ is not.
Which other quantities obtained from the hyperplane graphs are
valuative under matroid subdivisions?
\end{enumerate}

\section*{Acknowledgements}

The author thanks the authors of \cite{CKNSV} for formulating the vanishing conjecture and for the combinatorial framework on which this work builds. The author acknowledges Gewu Intelligence Lab for providing the collaborative research environment in which this project was developed.

\section*{Funding}

This research did not receive any specific grant from funding agencies in the public, commercial, or not-for-profit sectors.

\section*{Declaration of competing interest}

The author declares no known competing financial interests or personal relationships that could have appeared to influence the work reported in this paper.

\section*{Data availability}

The verification programs and finite certificates described in Section~\ref{sec:verification} are archived at Zenodo \cite{WangZenodo} and are supplied as supplementary files with this article. That section lists the inputs of each program and the assertions it checks.

\section*{Declaration of generative AI and AI-assisted technologies in the manuscript preparation process}

During the preparation of this work the author used OpenAI \mbox{GPT-5.6} Sol, Anthropic Claude Fable~5, Grok~4.6, OpenAI \mbox{GPT-6} Astra, Anthropic Claude Fable~5.1, OpenAI \mbox{GPT-6.1} Sol, and Anthropic Claude Opus~5.5. At different stages, these systems were used to develop the mathematical arguments and the ancillary checking programs described in Section~\ref{sec:verification}, to check proofs and references, and to revise the exposition and the \LaTeX{} source. After using these tools, the author reviewed and edited the content as needed and takes full responsibility for the content of the published article.

\begingroup
\raggedright

\endgroup

\end{document}